\documentclass[11pt]{amsart}  
\pdfoutput=1
\usepackage{enumitem}
\usepackage{times}
\usepackage{url}
\usepackage[leqno]{amsmath}   
\usepackage{cancel}
\usepackage{amssymb,mystyle,enumitem}
\usepackage{MnSymbol,wasysym}   
\usepackage{tikzsymbols}
\usepackage{tikz-cd}
\usepackage{tabularx}

\usepackage{tikz-cd}
\usepackage[colorlinks=true,linkcolor=blue]{hyperref}
\usepackage{placeins, multirow}

  \usetikzlibrary{decorations.markings}
  
\tikzset{double line with arrow/.style args={#1,#2}{decorate,decoration={markings,%
mark=at position 0 with {\coordinate (ta-base-1) at (0,1pt);
\coordinate (ta-base-2) at (0,-1pt);},
mark=at position 1 with {\draw[#1] (ta-base-1) -- (0,1pt);
\draw[#2] (ta-base-2) -- (0,-1pt);
}}}}

 \newcommand{\Sp}{{\mathrm{Sp}}}
    
\newcommand{\SO}{{\mathrm{SO}}}

\DeclareMathOperator{\proj}{proj}

\newcommand{\scoeff}{\sigma}

\newcommand{\absW}{\tilde{W}}
\newcommand{\Wtame}{W^{\dtame}}

\newcommand{\Wtamesimnon}{W^{\dtame}_{\sim}}

\newcommand{\Ctame}{\mathcal{C}^{\dtame}}

\DeclareMathOperator{\aff}{aff}

\usepackage{scalerel}

\DeclareMathOperator{\dtame}{t}
\DeclareMathOperator{\ee}{e}
\DeclareMathOperator{\Ad}{Ad}

\DeclareMathOperator{\X}{X}

\usepackage{calc}
\newsavebox\CBox
\newcommand\hcancel[2][0.5pt]{%
  \ifmmode\sbox\CBox{$#2$}\else\sbox\CBox{#2}\fi%
  \makebox[0pt][l]{\usebox\CBox}%
  \rule[0.5\ht\CBox-#1/2]{\wd\CBox}{#1}}

  \renewcommand{\gg}{{\mathfrak{g}}}
  \newcommand{\bsfT}

 \newtheoremstyle{TheoremNum}
        {\topsep}{\topsep}              
        {\itshape}                      
        {}                              
        {\bfseries}                     
        {.}                             
        { }                             
        {\thmname{#1}\thmnote{ \bfseries #3}}
    \theoremstyle{TheoremNum}
    \newtheorem{thmn}{Theorem}
     \newtheorem{corn}{Corollary}

\usepackage{amssymb}
\usepackage{mathtools}
\usetikzlibrary{chains}
\usetikzlibrary{decorations.pathreplacing, calligraphy}

\tikzset{node distance=2em, ch/.style={circle,draw,on chain,inner sep=2pt},chj/.style={ch,join},every path/.style={shorten >=4pt,shorten <=4pt},line width=1pt,baseline=-1ex}

\tikzset{noch/.style={circle,draw,on chain,inner sep=2pt, transparent},nochj/.style={noch,join},every path/.style={shorten >=4pt,shorten <=4pt}}
  
\let\dlabel=\alabel

\let\elabel=\nolabel

\newcommand{\enode}[2][nochj]{%
\node[#1,label={center:\elabel{#2}}] {};
}

\newcommand{\unode}[2][chj]{%
\node[#1,label={above:\dlabel{#2}}] {};
}

\newcommand{\unodenj}[1]{%
\unode[ch]{#1}
}

\newcommand{\enodenj}[1]{%
\enode[noch]{#1}
}

\newcommand{\dnodebr}[1]{%
\node[chj,label={right:\dlabel{#1}}] {};
}

\newcommand{\enodebr}[1]{%
\node[nochj,label={center:\elabel{#1}}] {};
}

\newcommand{\dydots}{%
\node[chj,draw=none,inner sep=1pt] {\dots};
}

\newcommand{\QRightarrow}{%
\begingroup
\tikzset{every path/.style={}}%
\tikz \draw (0,3pt) -- ++(1em,0) (0,1pt) -- ++(1em+1pt,0) (0,-1pt) -- ++(1em+1pt,0) (0,-3pt) -- ++(1em,0) (1em-1pt,5pt) to[out=-75,in=135] (1em+2pt,0) to[out=-135,in=75] (1em-1pt,-5pt);
\endgroup
}

\author{Stephen DeBacker}
\address{University of Michigan\\
Ann Arbor, MI 48109-1043, USA}
\email{stephendebacker@umich.edu}

\author{Jacob Haley}
\address{Vanderbilt University\\
Nashville, TN 37240, USA}
\email{jacob.a.haley@Vanderbilt.Edu}

 \makeatletter
\let\@wraptoccontribs\wraptoccontribs
\makeatother

  \subjclass[2020]{Primary 20G07; Secondary 20F55, 20E45}

  \date{\today}

\title{Kac diagrams for elliptic Weyl group elements}

\begin{document}

\begin{abstract}
Suppose $\gg$ is a semisimple complex Lie algebra and $\hh$ is a Cartan subalgebra of $\gg$.  To the pair $(\gg,\hh)$ one can associate both a Weyl group and a set of Kac diagrams.  There is a natural map  from the set of elliptic  conjugacy classes in the Weyl group to the set of  Kac diagrams.  In both this setting and the twisted setting, this paper (a) shows that this map is injective and (b) explicitly describes this map's image.
\end{abstract}

\maketitle

\section*{Introduction}
Suppose $\gg$ is a semisimple complex Lie algebra and $\hh$ is a Cartan subalgebra of $\gg$.  To the pair $(\gg,\hh)$ one can associate both a Weyl group and a set of Kac diagrams.  There is a natural map  from the set of elliptic  conjugacy classes in the Weyl group to the set of  Kac diagrams.   In both this setting and the twisted setting, this paper (a) shows that this map is injective and (b) explicitly describes this map's image.

\subsection*{The map \texorpdfstring{$\psi_G$}{psiG}}
Let $G$ denote a connected semisimple complex group such that the Lie algebra of  $G$ is $\gg$. Let $A$ denote the maximal torus in $G$ such that the Lie algebra of $A$ is $\hh$. Fix a pinning $(A,B, \{X_\alpha \, | \, \alpha \in \Delta \})$  of $G$ where  $B$ is a Borel subgroup of $G$ that contains $A$, $\Delta$ is the set of simple roots of $G$ with respect to $A$ and $B$, and $X_\alpha$ is an element of the Lie algebra of the root group corresponding to $\alpha \in \Delta$.  Fix $\vartheta \in \Aut(\gg)$ that stabilizes this pinning.  The automorphism $\vartheta$ is called a pinned automorphism and has the following properties: $\vartheta(A) = A$, $\vartheta(B) = B$, and $\vartheta$ permutes the elements of the set  $\{X_\alpha \, | \, \alpha \in \Delta \}$.    If $H$ is a group on which $\vartheta$ acts, then for $x,y \in H$ we say that $x$ is $\vartheta$-conjugate to $y$ in $H$ provided that there exists $h \in H$ such that $x = h y \vartheta(h)\inv$.

Let $\absW$ denote the Weyl group $N_G(A)/A$.  Since $\absW$ is generated by the simple reflections $w_\alpha$ for $\alpha \in \Delta$, we have that $\vartheta$ acts on $\absW$.  
An element $w \in \absW$ is said to be $\vartheta$-elliptic provided that the $\vartheta$-conjugacy class of $w$ in $\absW$ does not intersect any {proper} subgroup of $\absW$ of the form $\absW_\theta$ for $\theta \subsetneq \Delta$ with $\vartheta(\theta) = \theta$.   Here $\absW_\theta$ is the subgroup of $\absW$ generated by the simple reflections $w_\beta$ for $\beta \in \theta$.

Suppose $w_1, w_2 \in \absW$ are $\vartheta$-elliptic.  Let $n_j \in N_G(A)$ be a lift of $w_j$.
Since all lifts of $w_j$ into $N_G(A)$ are $\vartheta$-conjugate by an element of $A$  (see, for example,~\cite[Lemma~1.1.3]{adams-he-nie:from},~\cite[Remark~4.1]{debacker-reeder:depth-zero}, or~\cite[Theorem~1]{fedotov-affine} for the general case and~\cite[Lemma~1.3]{Levy:KW} or~\cite[Section 4.1]{schellekens-warner:weylI} for the non-twisted case), we conclude that  $w_1$ is $\vartheta$-conjugate to $w_2$ in $W$ if and only if $n_1$ is $\vartheta$-conjugate to $n_2$  in $N_G(A)$.  Moreover, both $n_1 \rtimes \vartheta$ and $n_2 \rtimes \vartheta$ have finite order in $N_G(A) \rtimes \langle \vartheta \rangle$.  Thus, there exists a function $\psi_G$ from $\absW^{\ee}_{\sim}$, the set of $\vartheta$-conjugacy classes of $\vartheta$-elliptic elements in $\absW$,  to $G^{\vartheta-\tor}_{\sim}$, the set of $\vartheta$-conjugacy classes of $\vartheta$-torsion elements in $G$.  An element $g \in G$ is called a $\vartheta$-torsion element provided that $g \rtimes \vartheta$ has finite order in $G \rtimes \langle \vartheta \rangle$.
In Section~\ref{sec:psiinjective} we prove

\vspace{1em}

\begin{thmn}[\ref{thm:psiinjects}] 
The map $\psi_G \, \colon \, \absW^{\ee}_{\sim} \rightarrow G^{\vartheta-\tor}_{\sim}$ is injective.
\end{thmn}

\vspace{1em}
\noindent
Note that there is something to prove here.  From the given  we know that we have a well-defined map from $\absW^{\ee}_{\sim}$ to the $\vartheta$-conjugacy classes of $\vartheta$-torsion elements in $N_G(A)$.    To prove the theorem, one needs to show that if 
$n_1, n_2 \in N_G(A)$ are lifts of $w_1, w_2 \in \absW^{\ee}$ such that $n_1$ and $n_2$ are $\vartheta$-conjugate in $G$, then $n_1$ and $n_2$ are $\vartheta$-conjugate in $N_G(A)$; this is false if we don't require {both} $w_1$ and $w_2$ to be $\vartheta$-elliptic.
The proof of Theorem~\ref{thm:psiinjects} involves case-by-case checking.

\subsection*{Kac diagrams and the map \texorpdfstring{$\chi_G$}{chi G}}    Reeder provides an absolutely beautiful treatment of Kac diagrams in~\cite{reeder:torsion}. Throughout this paper we have tried to align our treatment of Kac diagrams with the treatment found there.

Let $\Phi = \Phi(G,A)$ denote the set of roots of $G$ with respect to $A$.  The  perfect pairing $\langle \, , \, \rangle \, \colon \, \X_*(A) \times \X^*(A)$ allows us to regard $\alpha \in \Phi$ as a linear function on  $V = \X_*(A) \otimes \R$.  
If $S$ denotes the maximal compact torus in $A$, then the map $\exp \, \colon \, V \rightarrow S$, defined by $\alpha(\exp(v)) = e^{2 \pi i \langle \alpha , v \rangle}$ for all $\alpha \in \Phi$, is a surjective map.   The hyperplanes $H_{\alpha,n}$ in $V$ defined by the kernels of the affine roots $v \mapsto \alpha(v) - n$ for $\alpha \in \Phi$ and $n \in \Z$ produce a simplicial decomposition of $V$, and the connected components in the complement of these hyperplanes are called alcoves.  The affine Weyl group $\absW^{\aff} := \X_*(A) \rtimes \absW  $ acts on $V$.  This action preserves the hyperplane structure and $\absW^{\aff}$ acts transitively on the alcoves. Our pinning $\{X_\alpha \, | \, \alpha \in \Delta\}$ determines a unique alcove $C$ that contains the origin of $V$ in its closure.  By uniqueness,  we have  $\vartheta(C) = C$.
Let $A^\vartheta$ (resp. $G^\vartheta$) denote the connected component of the $\vartheta$-fixed points of $A$ (resp. $G$).

 Suppose $g$ is a finite order element of $G$, say of order $m$. Fix a primitive $m^{\text{th}}$ root of unity $\xi$.  There exists $h \in G$ such that $h\inv g \vartheta(h)$ lives in the torus $ A^{\vartheta}$ (see, for example, Lemma~\ref{lem:sigmaconj}).   Thus, there exists $\lambda \in \X_*(A^\vartheta)$ such that $h \inv g \vartheta(h) = \lambda(\xi)$.  
After conjugating $\lambda(\xi)$ by an element of  the group $N_{G^{\vartheta}}(A)$
we may assume that $\lambda/m$ belongs to the closure of
$C_\vartheta$, the alcove for $G^\vartheta$ that contains the origin and intersects $C^\vartheta$ nontrivially. 
A Kac diagram associated to $g$ encodes the location of $\lambda/m$  via a labeling of the affine Dynkin diagram for $G^\vartheta$.
See Section~\ref{sec:kacdiagrams} for the technical details of this labeling.

  Let $\kappa$ denote the set of Kac diagrams that arise as above, and we let $\kappa_\sim$ denote the equivalence classes in $\kappa$ determined by automorphisms of the affine Dynkin diagram of $G^\vartheta$ arising from the action of $\Stab_{\absW^{\aff}_{\vartheta}}(C_{\vartheta})$.  Here   $\absW^{\aff}_{\vartheta} = (\proj_{V^\vartheta} \X_*(A)) \rtimes \absW^\vartheta$ is the affine Weyl group for $G^\vartheta$.
This association between $\vartheta$-torsion elements of $G$ and Kac diagrams descends to a 
 map $\chi_G \, \colon \, G^{\vartheta-\tor}_{\sim} \rightarrow \kappa_\sim$. It is known (see, for example, ~\cite[Chapter X, Theorems~5.15 and~5.16]{helgason:differential},~\cite[Chapter 8]{kac:infinite},~\cite[$7^\circ$ of Section~4 of Chapter~4]{onishchik-vinberg:lie}, or~\cite[Theorem~3.7]{reeder:torsion})   that this map is bijective when $G$ is adjoint.

Under the assumption that $\gg$ is simple, in Section~\ref{sec:kacdiagrams} we give a complete description of the Kac diagrams associated to $\vartheta$-elliptic elements in $\absW$. 

\subsection*{The composition \texorpdfstring{$\zeta_G = \chi_G \circ \psi_G$ and applications}{applications}}
The composition $\zeta_G = \chi_G \circ \psi_G \, \colon \,\absW^{\ee}_{\sim} \rightarrow \kappa_{\sim}$ has been extensively studied.  For example, it is known to be independent of the  isogeny type of $G$ (see, for example,~\cite[comments following Proposition~6.6]{adams-he-nie:from} or~\cite[Remark~2.2.1, Lemma~3.3.6, and the fact that $x_n = x_{T_n}$ encodes the Kac diagram associated to $n$]{debacker:totally}), so we will write $\zeta$ rather than $\zeta_G$. It is also known that every equivalence class in the image of $\zeta$ is a singleton (see~\cite[Proposition~6.8]{adams-he-nie:from} or~\cite[Lemma~3.5.1]{debacker:totally}). The map $\zeta$ is known to be injective for groups of type $A_n$ (there is only one elliptic conjugacy class in $\absW$) as well as groups of types
$\lsup{2}A_2$,
$D_m$ for $4 \leq m \leq 8$,
$\lsup{3}D_4$,
$E_n$,
$\lsup{2}E_6$,
 $F_4$, and
 $G_2$
 (see the tabulations in~\cite{adams-he-nie:from}, 
 \cite{Bouwknegt:Lie}, 
 \cite{Levy:KW},
 \cite{reederetal:gradings}, and \cite{reederetal:epipelagic}). 
 Moreover, if  we denote by $\res_{\text{ell-reg}} \zeta$ the restriction of $\zeta$ to the set of regular (in the sense of Springer~\cite{springer:regular})  $\vartheta$-elliptic $\vartheta$-conjugacy classes in $\absW$, then $\res_{\text{ell-reg}} \zeta$ is known to be injective~\cite[discussion in Section~2.1]{reeder:thomae}.

Since $\zeta$ is independent of the isogeny type of $G$ and  $\chi_G$ is a bijective map when $G$ is adjoint,  Theorem~\ref{thm:psiinjects} implies
\vspace{1em}
\begin{corn}[\ref{cor:injective}]
The map $\zeta  \, \colon \, \absW^{\ee}_{\sim} \rightarrow \kappa_{\sim}$ is injective.
\end{corn}
\vspace{1em}
\noindent

In Section~\ref{sec:application} we show that Corollary~\ref{cor:injective} gives a new proof of the fact that every Weyl group is rational.   In the context of totally ramified tame tori in a group over a field with non-trivial discrete valuation and  cohomological dimension less than one, the Corollary implies that the point in the reduced building associated to  such a torus completely determines the torus up to conjugation.  We believe that this latter fact will play a role in harmonic analysis on reductive $p$-adic groups.   Corollary~\ref{cor:injective} is also implicitly used in the proof of~\cite[Proposition~7.1]{adams-he-nie:from}.

\begin{remark*}
If we replace $\C$ with any algebraically closed field, then the results of this paper still hold with very mild restrictions on $p$, the characteristic of the field.  The theory of Kac diagrams goes through as long as one assumes that $p$ does not divide the order of the $\vartheta$-torsion element under consideration~\cite[Section 2]{Levy:KW}.  This restriction on $p$ ensures that the roots of unity, denoted by  $\xi$ in this paper, that arise in the constructions exist.  
\end{remark*}

\begin{remark*} Since applications will often require that $G$ be reductive rather than semisimple, we note that  if $G$ is reductive, then one defines $\zeta$ to be $\zeta_{G'}$ where $G'$ is the derived group of $G$.
\end{remark*}

Finally, this paper involves a fair bit of notational gymnastics.  Although we independently derived our findings multiple times  and  compared our findings to results  in the literature, there are bound to be typos and more serious errors. So, be cautious.

\section*{Acknowledgements}  We thank Ram Ekstrom, Kaiwen Lu, Loren Spice, and Cheng-Chiang Tsai  for enlightening discussions.  We thank Mark Reeder for discussions about some of the details in~\cite{reeder:thomae}.

\section{Notation and two well-known results}

\subsection{Notation}  In addition to the notation introduced in the Introduction, we will need the following.

Let $K$ denote the square matrix with ones on the anti-diagonal and zeroes elsewhere.   The size of $K$ will be evident from context.  For example, if we are working with $G= \SO_{2\ell + 1}$, then $K$ is a $(2\ell + 1) \times (2 \ell + 1)$ matrix.

If $X$ is an $n \times m$ matrix with entries in $\C$, then $X^T$ denotes the transpose of $X$.

Suppose $\ell \in \Z_{\geq 1}$ and $\vec{\ell} = (\ell_\mu \geq \ell_{\mu-1} \geq \cdots \geq \ell_1)$ is a partition of $\ell$ into $\mu$ parts.   For $1 \leq \nu \leq \mu$ define $\ell'_\nu$ := $\ell_1 + \ell_2 + \cdots + \ell_{\nu - 1}$.  Note that $\ell'_1 = 0$ and $\ell'_\mu + \ell_\mu = \ell$.

For the vector space $\R^n$ we let $e_i$ denote  the $i^\text{th}$ standard basis vector for $1 \leq i \leq n$.  So $e_1 = (1,0,0, \ldots ,0)$, $e_2 =  (0,1,0,0, \ldots ,0)$, etc.  

Suppose $X$ is set.  If $\tau$ is a function from $X$ to $X$,  then $X^\tau$ denotes the set of $x\in X$ such that $\tau(x) = x$.   If $X$ also has the structure of a connected complex group and 
$\tau$ is an automorphism of $X$, then $X^\tau$ denotes the \textbf{connected component} of the $\langle \tau \rangle$-fixed points of $X$.
If $y \in X$ and $H$ is a group that acts on $X$, then $\lsup{H}y := \{ h \cdot y \, | \, h \in H \}$.  

\subsection{Two known results}
We begin by  proving an elementary result from linear algebra.  

\begin{lemma} \label{lem:veryelement} Suppose $F$ is a field and  $U$ is a finite dimensional $F$-vector space.   Let $(\vec{u}_1,\vec{u}_2, \ldots , \vec{u}_m)$ be a basis for $U$.   Suppose $j_1, j_2, \ldots , j_m \in \Z$ and $T \, \colon \, U \rightarrow U$ is a linear map such that $T(\vec{u}_k) = (-1)^{j_k} \vec{u}_{k+1}$ for $1 \leq k < m$ and $T(\vec{u}_m) = (-1)^{j_m} \vec{u}_1$.   If $T^m = \pm \Id_U$, then the characteristic polynomial of $T$ is $t^m \mp 1$.
\end{lemma}

\begin{proof}  If $T^m = \Id_U$, then we conclude that $\sum_{k=1}^m j_k$ is even.  If $T^m = -\Id_U$, then we conclude that $\sum_{k=1}^m j_k$ is odd.   Since the characteristic polynomial of $T$ is 
$$t^m - (-1)^{\sum_{k=1}^m j_k},$$
  the result follows.  
\end{proof}

The next result is well known to experts, but we could not find a reference that did not place conditions on the isogeny class of $G$.

\begin{lemma}  \label{lem:sigmaconj}
  If $g \in G^{\vartheta-\tor}$, then $g$ is $\vartheta$-conjugate in $G$ to an element of $A^{\vartheta}$.   
\end{lemma}

\begin{proof}
Recall that an automorphism of $G$ is said to be quasi-semisimple if it preserves a Borel-torus pair in $G$.

    Since $\vartheta$ is quasi-semisimple and $g \rtimes \vartheta$ has finite order, it follows that $g \rtimes \vartheta$ is quasi-semisimple~\cite[\S7 and \S9]{steinberg:endomorphism}.  Hence, $g \rtimes \vartheta$ preserves a Borel-torus pair $(A',B')$ in $G$.  
    Since all Borel-torus pairs in $G$ are conjugate, there exists $h \in G$ such that $h(A,B) = (A',B')$.  We then have
    $$h(A,B) = (A',B') = g \cdot \vartheta (A',B') = g \vartheta(h) \cdot \vartheta(A,B) =  g \vartheta(h)  (A,B),$$
    which implies that $h\inv g \vartheta(h) \in A$.  Thanks to~\cite[Lemma~7.2]{digne-michel:complements}, we know $A = A^\vartheta \cdot (1 - \vartheta)A$. The result follows.
\end{proof}

\section{An application of the injectivity of \texorpdfstring{$\zeta = \chi_G \circ \psi_G$}{zeta is comp}}   \label{sec:application}

A finite group $H$ is said to be \emph{rational} provided that for any $h \in H$ we have that $h^j$ is $H$-conjugate to $h$ whenever $j$ is relatively prime to the order of $h$.
There are many equivalent definitions of rationality, and the nomenclature is explained by the fact~\cite[Corollary 2 of Section 13.1]{serre:linear} that $H$ is rational if and only if every irreducible character of $H$ takes values in $\Q$.

It is known that every Weyl group is a rational group (see, for example,~\cite[Section 3 of Chapter 2 and Chapter 5]{kletzing:structure} or~\cite[Theorem~8.5]{springer:regular}).  In this section, we give another proof of this fact.

\subsection{Weyl groups are rational}
 
As noted in the Introduction, 
the fact that  $\zeta$ is independent of the isogeny type of $G$ together with Theorem~\ref{thm:psiinjects} implies
\begin{cor}  \label{cor:injective}
The map $\zeta  \, \colon \, \absW^{\ee}_{\sim} \rightarrow \kappa_{\sim}$ is injective.
\end{cor}

Using this, we show:

\begin{lemma} \label{lem:Weylgroupsarerational}
Every Weyl group is rational.
\end{lemma}

 \begin{proof}
 Suppose $W$ is a Weyl group.  We will prove this by induction.   If $W$ is trivial, then there is nothing to prove.  Suppose then that $W$ is not trivial and the statement is true for every parabolic subgroup of $W$.

 Suppose $w \in W$.   If the $W$-orbit of $w$ intersects a proper parabolic subgroup $W'$ of $W$, then without loss of generality we may assume that $w \in W'$.  Since the order of $w$ in $W'$ is equal to the order of $w$ in $W$, by induction we have that if $j$ is relatively prime to the order of $w$ in $W$, then $w^j$ is  $W'$-conjugate, hence $W$-conjugate, to $w$.

 Suppose now that the $W$-orbit of $w$ does not intersect a proper parabolic subgroup of $W$. That is, suppose $w$ is elliptic in $W$.   It will be enough to show that if  $p$ is a prime that does not divide the order  of $w$, then $w^p$ and $w$ are $W$-conjugate. Fix such a prime $p$.

Let $\bH$ be a connected reductive group defined over $\Z$ with maximal $\Z$-torus $\bA$ such that  $N_{\bH}(\bA)/\bA$  is $W$.  Then $\bH$ is a $\Q_p$-split group where $\Q_p$ denotes the $p$-adic completion of $\Q$.  Let $\bar{\Q}_p$ be an algebraic closure of $\Q_p$.  Let $\Q_p^{\unram}$ denote the maximal unramified extension of $\Q_p$ in $\bar{\Q}_p$ and let $\Q_p^{\tame}$ denote the maximal tame extension of $\Q_p$ in $\bar{\Q}_p$.  We have $\Q_p^{\unram} \leq \Q_p^{\tame}$.   
Fix a topological generator $\Fr$ for $\Gal(\Q_p^{\unram}/\Q_p)$ and a topological generator $\sigma$ for $\Gal(\Q_p^{\tame}/\Q_p^{\unram})$ such that if  $\Fr$ and $\sigma$ also denote lifts of these elements into $\Gal({\Q}_p^{\tame}/ \Q_p)$, then $\Fr\inv \sigma \Fr = \sigma^p$.

Let $\Wtame$ denote the set of tame elements in $W$, that is, those elements $w' \in W$ such that $p$ does not divide the order of $w'$.    If $w' \in \Wtame$, then there exists $h \in \bH(\Q_p^{\tame})$ such that the image of $h\inv \sigma(h)$ in $W$ is $w'$ and the map $w' \mapsto \lsup{\bH(\Q_p^{\unram})} (\lsup{h} \bA)$ defines a bijective map $\varphi_\sigma$ from $\Wtamesimnon$, the conjugacy classes of tame elements in $W$, to $\Ctame$, the set of $\bH(\Q_p^{\unram})$-conjugacy classes of $\Q_p^{\unram}$-minisotropic maximal $\Q_p^{\unram}$-tori in $\bH$~\cite[Lemma~3.4.1]{debacker:totally}.  A $\Q_p^{\unram}$-torus $\bT$ in $\bH$ is said to be $\Q_p^{\unram}$-minisotropic provided that $\X^*(\bT/Z(\bH))^{\Gal(\bar{\Q}_p/\Q_p^{\unram})}$ is trivial.

Since $\bH$ is $\Q_p$-split, we have that $\Fr$ and $\sigma$ both  act trivially on  $W$.   Thus, the commutative diagram of bijective maps in~\cite[Lemma~5.1.2]{debacker:totally} becomes
 \begin{center}
\begin{tikzcd}
\Wtamesimnon \arrow[r, "( \text{-} )^p"] \arrow[d, "\varphi_\sigma"]
& \Wtamesimnon \arrow[r, "\Id"] \arrow[d, "\varphi_{\sigma^p}"]
&  \Wtamesimnon \arrow[d,  "\varphi_{\sigma}" ] \\
\Ctame \arrow[r,  "\Id" ]
& \Ctame  \arrow[r, "\Fr"]
& \Ctame
\end{tikzcd}
\end{center}
Since $p$ does not divide the order of $w \in \Wtame$, the diagram implies $\Fr(\varphi_\sigma(\lsup{W}w)) = \varphi_\sigma(\lsup{W}(w^p))$.

 Fix a $\Fr$-stable alcove $D$ in the reduced Bruhat-Tits building of $\bH(\Q_p^{\unram})$.  From~\cite[Section~2.2]{debacker:totally} we know that $\varphi_\sigma(\lsup{W}w)$ and $\Fr(\varphi_\sigma(\lsup{W}w))$ identify  unique points, call them $x$ and $x'$, in the closure of $D$.  Note that $x' = \Fr(x)$.  From~\cite[Section~3.5]{debacker:totally} the barycentric coordinates of $x$ are  determined by $\zeta(\lsup{W}w)$ and those of $x'$ by $\zeta(\lsup{W}(w^p))$.   Since $\bH$ is $\Q_p$-split, we have that $\Fr$ fixes $D$ pointwise.  Thus,  $x' = \Fr(x) = x$.  Thanks to Corollary~\ref{cor:injective}, we know that the map $\zeta$ is injective. Consequently, $w^p$ is $W$-conjugate to $w$.
 \end{proof}

\section{The map \texorpdfstring{$\psi_G$}{psiG} is injective}   \label{sec:psiinjective}

The proof of Theorem~\ref{thm:psiinjects} has two steps. In Section~\ref{sec:reduction}  we show that we may reduce to the case when $G$ is isogenous to a simple complex adjoint  group.  In Sections~\ref{sec:injAlminusone} to~\ref{sec:injexceptional} we show that for a given simple complex adjoint group $G'$, the map $\psi_G$ is injective for  some $G$ that is isogenous to $G'$.   The case-by-case work we carry out was made much easier thanks to the clearly written material in 
~\cite[Sections~3 and~4]{gecketal:minimal}, ~\cite[Section~3.4]{geck:pfeiffer:characters}, and~\cite[Section~2]{geck-pfeiffer:irreducible}.

\begin{remark}
Note that the order  of a lift of a $\vartheta$-elliptic Weyl group element  is independent of the lift. Thus, if $w_1$ and $w_2$ are $\vartheta$-elliptic elements of $\absW$ that lift to $n_1$ and $n_2$ with $n_1 \rtimes \vartheta$ and $n_2 \rtimes \vartheta$  having different orders, then we have $w_1$ is not $\vartheta$-conjugate to $w_2$ in $\absW$.   On the other hand, it sometimes happens  that  $n_1 \rtimes \vartheta $ and $n_2  \rtimes \vartheta $ have the same order, yet  $w_1$ is not $\vartheta$-conjugate to $w_2$ in $\absW$.  So looking at orders to establish injectivity will  not work in general.
\end{remark}

\subsection{Reduction to the almost-simple case}  \label{sec:reduction}

We begin by showing that the injectivity of $\psi_G$ is independent of the isogeny type of $G$.

\begin{lemma} \label{lem:isogenyindependent} Suppose $H$ and $I$ are semisimple complex groups that are isogenous to $G$.   We have that the map
$$\psi_H \, \colon \, \absW^{\ee}_{\sim} \rightarrow H^{\vartheta-\tor}_{\sim}$$
is injective if and only if the map
$$\psi_I \, \colon \, \absW^{\ee}_{\sim} \rightarrow I^{\vartheta-\tor}_{\sim}$$
is injective.
\end{lemma}

\begin{proof} Suppose $\rho \, \colon \, H \rightarrow I$ is an isogeny.  Let $A_H$ (resp. $A_I$) be a maximal torus in $H$ (resp $I$) with $\rho[A_H] = A_I$. 
Fix $w_1, w_2 \in \absW^{\ee}$.  Choose $n_k \in N_H(A_H)$ and $\dot{n}_k \in N_I(A_I)$ lifting $w_k$.

Suppose first that $\psi_H$ is injective.  If there exists $i \in I$ such that $i \dot{n}_1 \vartheta(i)\inv = \dot{n}_2$, then there exists $h \in H$ such that $h {n}_1 \vartheta(h)\inv \in {n}_2 A_H$.  Since all lifts of $w_2 \in \absW^{\ee}$ into $N_H(A_H)$ are $\vartheta$-conjugate by an element of $A_H$, we may assume that $h {n}_1 \vartheta(h)\inv = {n}_2$.  Since $\psi_H$ is injective, we conclude that $w_1$ is $\vartheta$-conjugate to $w_2$ in $\absW$.  That is, $\psi_I$ is injective.

Suppose now that that $\psi_I$ is injective.  If there exists $h \in H$ such that $h {n}_1 \vartheta(h)\inv = {n}_2$, then by taking $i = \rho(h)$ and $n'_k = \rho(n_k)$, we have  $i {n}'_1 \vartheta(i)\inv = {n}'_2$.  Since all lifts of $w_k \in \absW^{\ee}$ into $N_I(A_I)$ are $\vartheta$-conjugate by an element of $A_I$, we may assume that $i \dot{n}_1 \vartheta(i)\inv = \dot{n}_2$.  Since $\psi_I$ is injective, we conclude that $w_1$ is $\vartheta$-conjugate to $w_2$ in $\absW$.  That is, $\psi_H$ is injective.
\end{proof}

\begin{cor} \label{cor:reductionsimple} The map $\psi_G$ is injective if and only if $\psi_H$ is injective for every almost-simple factor $H$ of $G$. 
\end{cor}

\begin{proof}
Since the injectivity of $\psi_G$ is independent of the isogeny type of $G$, we may assume $G$ is simply connected.  Since $G$ is simply connected, we may write~\cite[Theorem~21.51 and discussion in~24.a]{milne:algebraic} the $\C$-group $G$ as a product of almost-simple simply connected complex groups: $G = G_1 \times G_2 \times \cdots \times G_\tau$.  The result follows.
\end{proof}

\begin{theorem}   \label{thm:psiinjects}
The map $\psi_G \, \colon \, \absW^{\ee}_{\sim} \rightarrow G^{\vartheta-\tor}_{\sim}$ is injective.
\end{theorem}

\begin{proof}
Thanks to Corollary~\ref{cor:reductionsimple} it is enough to show that $\psi_{G'}$ is injective for $G'$ being an almost-simple complex  group.  Thanks to Lemma~\ref{lem:isogenyindependent}, it is enough to show that $\psi_G$ is injective for some $G$  that is isogenous to $G'$.  This is done on a case-by-case basis in Sections~\ref{sec:injAlminusone} to~\ref{sec:injexceptional}.
\end{proof}

\subsection{Injectivity of \texorpdfstring{$\psi_G$ for type $A_{\ell-1}$ with $\ell \geq 3$}{A unique}}   \label{sec:injAlminusone}

In this situation there is only one elliptic conjugacy class in the Weyl group.  Thus, the injectivity of $\psi_G$ is automatic.  However, in the interest of establishing notation in a familiar setting, we  carry out a proof that has the same flavor as the proofs for the other root system types.

Let $G = \SL_{\ell} = \SL_{\ell}(\C)$ realized as the set of $\ell \times \ell$ matrices of determinant $1$.  The Lie algebra of $G$ is then $\gg$, the set of $\ell \times \ell$ matrices of trace zero.  We take $A$ to be the diagonal matrices in $G$ and $B$ to be the upper triangular matrices in $G$.  For  $a = \Diag(a_1, a_2, \ldots, a_\ell) \in A \leq \SL_{\ell}$ and $1 \leq k < \ell$, we let $\alpha_k(a) = a_k/a_{k+1}$ be the simple roots of $G$ with respect to $A$ and $B$.

Following Bourbaki~\cite[Planche I]{bourbaki:lie4to6}  we identify $V$ with 
the hyperplane in $\R^{\ell}$ for which the sum of the coordinates is zero, identify $\Phi$ with the vectors $e_i - e_j$ with $1 \leq i \neq j \leq \ell$, and identify $\Delta$ with the roots $\alpha_k = e_k-e_{k+1}$ for $1 \leq k < \ell$.

For $1 \leq k < \ell$ we define $s_k \in \SL_\ell$  to be the $\ell \times \ell$ matrix such that 
$$(s_k)_{ij} = 
\begin{cases}
    1 &\qquad \text{$i=j$ and $i \not\in \{k,k+1\}$}\\
    1 &\qquad \text{$i  = k$ and $j = k+1$} \\
     -1 &\qquad \text{$i  = k+1$ and $j = k$} \\
     0 &\qquad \text{otherwise.}
\end{cases}$$
The image of  the order four element $s_k$ in $\absW$ is the simple reflection corresponding to $\alpha_k$.

The Coxeter class is the only elliptic conjugacy class in $\absW$, and it corresponds to the partition $\vec{\ell} = (\ell)$ of $\ell$. Define $n_{({\ell})} = \prod_{j=1}^{\ell -1} s_j$.   This is a lift of a Coxeter element in $\absW$, and it has order $2 \ell$ if $\ell$ is even or $\ell$ if $\ell$ is odd.

Any element $g \in \SL_{\ell}$ acts on $\C^{\ell}$ by matrix multiplication.  The characteristic polynomial of $g$ with respect to this action is an $\SL_\ell$-conjugation invariant polynomial.  Thanks to Lemma~\ref{lem:veryelement} the characteristic polynomial of $n_{({\ell})}$ with respect to this action is
$ q_{n_{({\ell})}}(t) = t^{\ell} + (-1)^{\ell}$.

Since any two lifts of an elliptic element in $\absW$ are $A$-conjugate and since there exists a unique conjugacy class of elliptic elements in $\absW$,  we conclude that two elliptic elements of $\absW$ are $\absW$-conjugate if and only if their lifts are $\SL_{\ell}$-conjugate.

\subsection{Injectivity of \texorpdfstring{$\psi_G$ for type $B_\ell$ with $\ell \geq 2$}{inject B}}   \label{sec:injBl}   Let $G = \SO_{2\ell + 1} = \SO_{ 2 \ell +1}(\C)$ realized as those $g \in \SL_{2\ell +1}$ such that $g = (K g\inv K)^T$.   Recall that $K$ denotes the $(2\ell + 1) \times (2 \ell + 1)$ matrix with ones on the anti-diagonal and zeroes elsewhere.   
Then $\gg$, the Lie algebra of $G$, is $\solie_{2 \ell+1}$  realized as the Lie subalgebra of $\sllie_{2\ell+1}$ consisting of those matrices $X$ for which $X = -KX^TK$.  We take $A$ to be the diagonal matrices in $G$ and $B$ to be the upper triangular matrices in $G$.   For  $a = \Diag(a_1, a_2, \ldots, a_\ell, 1, a_\ell\inv, \ldots , a_2\inv, a_1\inv) \in A \leq \SO_{2\ell +1}$, we let $\alpha_\ell(a) = a_\ell$  and $\alpha_k(a) = a_k/a_{k+1}$ for $1 \leq k < \ell$ be the  simple roots of $G$ with respect to $A$ and $B$.

Following Bourbaki~\cite[Planche II]{bourbaki:lie4to6} , we identify $V$  with $\R^\ell$, identify $\Phi$  with the set of vectors $\pm e_i$ for $1 \leq i \leq \ell$ and  $\pm e_i \pm e_j$ for $1 \leq i < j \leq \ell$,   and identify $\Delta$ with
$$\alpha_1 = e_1 - e_2,  \,  \alpha_2 = e_2 - e_3,  \, \ldots , \, \alpha_{\ell-1} = e_{\ell-1} - e_{\ell}, \, \alpha_\ell =  e_{\ell}.$$
We define $t_\ell \in \SO_{2 \ell +1}$ to be the $(2 \ell +1) \times (2 \ell + 1)$ matrix such that 
$$(t_\ell)_{ij} = 
\begin{cases}
    1 &\qquad \text{$1 \leq i=j < \ell$ or $\ell + 2 < i=j \leq  2\ell + 1$}\\
    1 &\qquad \text{$i \neq j$ and $i,j \in \{\ell, \ell + 2\}$}\\
    - 1 &\qquad \text{$i = j = \ell + 1$}  \\
     0 &\qquad \text{otherwise,}
\end{cases}$$
and for $1 \leq k < \ell$ we define $s_k \in \SO_{2 \ell +1}$ to be the $(2 \ell +1) \times (2 \ell+1)$ matrix such that 
$$(s_k)_{ij} = 
\begin{cases}
    1 &\qquad \text{$i=j$ and $i \not\in \{k,k+1,2\ell -k+2, 2\ell-k+1 \}$}\\
    1 &\qquad \text{$i \neq j$ and  $i,j \in \{k,k+1\}$}\\
     1 &\qquad \text{$i \neq j$ and  $i,j \in \{2 \ell - k +1, 2 \ell - k+2\}$}  \\
     0 &\qquad \text{otherwise.}
\end{cases}$$
The image of  the order two element $t_\ell$ in $\absW$ is the simple reflection corresponding to $\alpha_\ell$ while the image of the order two element  $s_k$ in $\absW$ is the simple reflection corresponding to $\alpha_k$.    For $1 \leq k < \ell$ define $t_k \in \SO_{2 \ell +1}$ by 
$$t_k = s_k \cdot s_{k+1} \cdot s_{k+2} \cdot \cdots \cdot s_{\ell-2} \cdot s_{\ell -1} \cdot t_\ell \cdot s_{\ell-1} \cdot s_{\ell-2} \cdot  \cdots  \cdot s_{k+1} \cdot s_k.$$
Note that the $t_k$ defined above is denoted by $t_{k-1}$ in~\cite[\S3.4]{geck:pfeiffer:characters}.

The elliptic conjugacy classes are parameterized by the partitions of $\ell$~\cite[7.16]{he:minimal}.  
Fix a partition $\vec{\ell} = (\ell_\mu \geq \ell_{\mu-1} \geq \cdots \geq \ell_1)$  of $\ell$.   Recall that for $1 \leq \nu \leq \mu$ we define $\ell'_\nu$ = $\ell_1 + \ell_2 + \cdots + \ell_{\nu - 1}$.  We define  $n_{\vec{\ell}}$ to be the product $c_1 \cdot c_2 \cdot \cdots \cdot c_\mu$ where
$c_{\nu} = s_{\ell'_{\nu} + 1}\cdot s_{\ell'_{\nu} + 2} \cdot \cdots \cdot s_{\ell'_{\nu} + \ell_{\nu}-1} \cdot t_{\ell'_{\nu} + \ell_{\nu}}$.
Thus
$$n_{\vec{\ell}} = \prod_{\nu = 1}^{\mu}(s_{\ell'_{\nu} + 1}\cdot s_{\ell'_{\nu} + 2} \cdot \cdots \cdot s_{\ell'_{\nu} + \ell_{\nu}-1} \cdot t_{\ell'_{\nu} + \ell_{\nu}}).$$
The image of $n_{\vec{\ell}}$ in $\absW$ is an elliptic element belonging to the conjugacy class associated to $\vec{\ell}$.

\begin{remark}   \label{rem:regellBl}
Since $w \in \absW$ is regular elliptic if and only if $\hh^{w} = \{0\}$ and $\langle w \rangle$ acts freely on $\Phi$,  the regular elliptic elements of $\absW$  correspond to partitions of the form  $(d,d, \ldots , d)$ where $d \in \Z_{\geq 1}$ divides $\ell$. 
\end{remark}

Any element $g \in \SO_{2\ell+1}$ acts on $\C^{2\ell+1}$ by matrix multiplication.  The characteristic polynomial of $g$ with respect to this action is an $\SO_{2\ell+1}$-conjugation invariant polynomial.  Using Lemma~\ref{lem:veryelement}, one finds that  the characteristic polynomial of $n_{\vec{\ell}}$ with respect to this action is
$$q_{n_{\vec{\ell}}}(t) = (t - (-1)^{\mu}) \cdot \prod_{\nu = 1}^\mu (t^{2 \ell_\nu} - 1).$$
  If $\vec{\ell}$ and $\vec{\ell}'$ are two partitions of $\ell$ for which the associated elements $n_{\vec{\ell}}$
and $n_{\vec{\ell}'}$ of $N_G(A)$ are $G$-conjugate, then we have $q_{n_{\vec{\ell}}} = q_{n_{\vec{\ell}'}}$.  The only way this can happen is if $\vec{\ell} = \vec{\ell}'$.

Since any two lifts of an elliptic element in $\absW$ are $A$-conjugate,  we conclude that two elliptic elements of $\absW$ are $\absW$-conjugate if and only if their lifts are $\SO_{2 \ell+1}$-conjugate.

\subsection{Injectivity of \texorpdfstring{$\psi_G$ for type $C_\ell$ with $\ell \geq 2$}{inject C}}  \label{sec:injCl}  
Recall that $K$ denotes the $\ell \times \ell$ matrix with ones on the anti-diagonal and zeroes elsewhere. Denote by $J$  the product of $K$ and  $\Diag(1,1,  \ldots, 1, 1,-1, -1,   \ldots, -1, -1)$; here there are $\ell$ copies of $1$ followed by $\ell$ copies of $-1$.   We realize  $G = \Sp_{2 \ell} = \Sp_{2\ell}( \C)$ as those $g \in \SL_{2 \ell}$ such that $g^T J g= J$ and we realize $\splie_{2\ell}$ as those $X \in \sllie_{2 \ell}$ such that $X^T J = - JX$.
  We take $A$ to be the diagonal matrices in $G$ and $B$ to be the upper triangular matrices in $\Sp_{2 \ell}$.   For  $a = \Diag(a_1, a_2, \ldots, a_\ell, a_\ell\inv, \ldots , a_2\inv, a_1\inv) \in A$, we let $\alpha_\ell(a) = a_\ell^2$  and $\alpha_k(a) = a_k/a_{k+1}$ for $1 \leq k < \ell$ be the  simple roots of $G$ with respect to $A$ and $B$.

Following Bourbaki~\cite[Planche III]{bourbaki:lie4to6} , we identify $V$  with $\R^\ell$, identify $\Phi$ with
the set of vectors $\pm 2 e_i$ for $1 \leq i \leq \ell$ and $\pm e_i \pm e_j$ for $1 \leq i < j \leq \ell$, and identify  $\Delta$ with the roots
$$\alpha_1 = e_1 - e_2,  \,  \alpha_2 = e_2 - e_3,  \, \ldots , \, \alpha_{\ell-1} = e_{\ell-1} - e_{\ell}, \, \alpha_\ell = 2 e_\ell.$$
We define $t_\ell \in \Sp_{2 \ell}$ to be the $2 \ell \times 2 \ell$ matrix such that 
$$(t_\ell)_{ij} = 
\begin{cases}
    1 &\qquad \text{$1 \leq i=j < \ell$ or $\ell + 1 < i=j \leq  2\ell$}\\
    1 &\qquad \text{$i = \ell$ and  $j = \ell + 1$}\\
    - 1 &\qquad \text{$i = \ell + 1$ and  $j = \ell$}  \\
     0 &\qquad \text{otherwise,}
\end{cases}$$
and for $1 \leq k < \ell$ we define $s_k \in \Sp_{2 \ell}$ to be the $2 \ell \times 2 \ell$ matrix such that 
$$(s_k)_{ij} = 
\begin{cases}
    1 &\qquad \text{$i=j$ and $i \not\in \{k,k+1,2\ell -k+1, 2\ell-k \}$}\\
    1 &\qquad \text{$i  \neq j$ and $i,j \in \{k,k+1 \}$ }\\
     1 &\qquad \text{$i  \neq j$ and $i,j \in \{2 \ell - k, 2 \ell - k + 1\}$ }\\
     0 &\qquad \text{otherwise.}
\end{cases}$$
The image of  the order four element $t_\ell$ in $\absW$ is the simple reflection corresponding to $\alpha_\ell$ while the image of the order two element  $s_k$ in $\absW$ is the simple reflection corresponding to $\alpha_k$.    For $1 \leq k < \ell$ define $t_k \in \Sp_{2 \ell}$ by 
$$t_k = s_k \cdot s_{k+1} \cdot s_{k+2} \cdot \cdots \cdot s_{\ell-2} \cdot s_{\ell -1} \cdot t_\ell \cdot s_{\ell-1} \cdot s_{\ell-2} \cdot  \cdots  \cdot s_{k+1} \cdot s_k.$$
Note that the $t_k$ defined above is denoted by $t_{k-1}$ in~\cite[\S3.4]{geck:pfeiffer:characters}.

Since the Weyl group of type $C_\ell$ may be naturally identified with the Weyl group of type $B_\ell$,
the elliptic conjugacy classes in $\absW$ correspond to partitions  $\vec{\ell} = (\ell_\mu \geq \ell_{\mu-1} \geq \cdots \geq \ell_1)$  of $\ell$~\cite[7.16]{he:minimal}.
Recall that for  $1 \leq \nu \leq \mu$ we define $\ell'_\nu$ = $\ell_1 + \ell_2 + \cdots + \ell_{\nu - 1}$.   We define  $n_{\vec{\ell}}$ to be the product $c_1 \cdot c_2 \cdot \cdots \cdot c_\mu$ where
$c_{\nu} = s_{\ell'_{\nu} + 1}\cdot s_{\ell'_{\nu} + 2} \cdot \cdots \cdot s_{\ell'_{\nu} + \ell_{\nu}-1} \cdot t_{\ell'_{\nu} + \ell_{\nu}}$.
Thus
$$n_{\vec{\ell}} = \prod_{\nu = 1}^{\mu}(s_{\ell'_{\nu} + 1}\cdot s_{\ell'_{\nu} + 2} \cdot \cdots \cdot s_{\ell'_{\nu} + \ell_{\nu}-1} \cdot t_{\ell'_{\nu} + \ell_{\nu}}).$$
The image of $n_{\vec{\ell}}$ in $\absW$ is an elliptic element belonging to the conjugacy class associated to $\vec{\ell}$.

\begin{remark}   \label{rem:regellCl}
Since $w \in \absW$ is regular elliptic if and only if $\hh^{w} = \{0\}$ and $\langle w \rangle$ acts freely on $\Phi$,  the regular elliptic elements of $\absW$  correspond to partitions of the form  $(d,d, \ldots , d)$ where $d \in \Z_{\geq 1}$ divides $\ell$. 
\end{remark}

Any element $g \in \Sp_{2\ell}$ acts on $\C^{2\ell}$ by matrix multiplication.  The characteristic polynomial of $g$ with respect to this action is an $\Sp_{2\ell}$-conjugation invariant polynomial.  Using Lemma~\ref{lem:veryelement}, one finds that  the characteristic polynomial of $n_{\vec{\ell}}$ with respect to this action is
$$q_{n_{\vec{\ell}}} (t)=  \prod_{\nu = 1}^\mu (t^{2 \ell_\nu} + 1).$$
  If $\vec{\ell}$ and $\vec{\ell}'$ are two partitions of $\ell$ for which the associated elements $n_{\vec{\ell}}$
and $n_{\vec{\ell}'}$ of $N_G(A)$ are $G$-conjugate, then we have $q_{n_{\vec{\ell}}} = q_{n_{\vec{\ell}'}}$.  The only way this can happen is if $\vec{\ell} = \vec{\ell}'$.

Since any two lifts of an elliptic element in $\absW$ are $A$-conjugate,  we conclude that two elliptic elements of $\absW$ are $\absW$-conjugate if and only if their lifts are $\Sp_{2 \ell}$-conjugate.

\subsection{Injectivity of \texorpdfstring{$\psi_G$ for type $D_\ell$ with $\ell \geq 3$}{inject D}}   \label{sec:injDl}
 Let $G = \SO_{2\ell} = \SO_{2 \ell}(\C)$ realized as those $g \in \GL_{2\ell}$ such that $g = (K g\inv K)^T$ and $\det(g) = 1$.  If we remove the determinant one condition, then $g$ is an element of $\On_{2\ell} = \On_{2 \ell}(\C)$.  Then $\gg$, the Lie algebra of $G$, is $\solie_{2 \ell}$  realized as the Lie subalgebra of $\sllie_{2\ell}$ consisting of those matrices $X$ for which $X = -KX^TK$.  We take $A$ to be the diagonal matrices in $G$ and $B$ to be the upper triangular matrices in $G$.   For  $a = \Diag(a_1, a_2, \ldots, a_\ell, a_\ell\inv, \ldots , a_2\inv, a_1\inv) \in A$, we let $\alpha_\ell(a) = a_{\ell -1} \cdot a_\ell$  and $\alpha_k(a) = a_k/a_{k+1}$ for $1 \leq k < \ell$ be the  simple roots of $G$ with respect to $A$ and $B$.

Following Bourbaki~\cite[Planche IV]{bourbaki:lie4to6} we identify $V$ with $\R^\ell$, identify $\Phi$ with the vectors $\pm e_i \pm e_j$ for $1 \leq i < j \leq \ell$,  and identify $\Delta$ with the roots 
$$\alpha_1 = e_1 - e_2,  \,  \alpha_2 = e_2 - e_3,  \, \ldots , \, \alpha_{\ell-1} = e_{\ell-1} - e_{\ell}, \, \alpha_\ell = e_{\ell-1} + e_{\ell}.$$

 For $1 \leq k \leq \ell$ we define $t_k \in \On_{2 \ell}$ to be the $2 \ell \times 2 \ell$ matrix such that 
$$(t_k)_{ij} = 
\begin{cases}
    1 &\qquad \text{$1 \leq i=j \leq 2l$ and $i,j \not\in \{ k, 2 \ell - k + 1 \} $}\\
    1 &\qquad \text{$i \neq j$ and  $i,j \in \{ k, 2 \ell - k + 1 \} $}\\
     0 &\qquad \text{otherwise.}
\end{cases}$$
For $1 \leq k < \ell$ we define $s_k \in \SO_{2 \ell}$ to be the $2 \ell \times 2 \ell$ matrix such that 
$$(s_k)_{ij} = 
\begin{cases}
    1 &\qquad \text{$i=j$ and $i \not\in \{k,k+1,2\ell -k+1, 2\ell-k \}$}\\
    1 &\qquad \text{$i \neq j$ and  $i,j \in \{k,k+1\}$}\\
     1 &\qquad \text{$i \neq j$ and  $i,j \in \{2 \ell - k, 2 \ell - k+1\}$}  \\
     0 &\qquad \text{otherwise.}
\end{cases}$$
The image of the order two element  $s_k$ in $\absW$ is the simple reflection corresponding to $\alpha_k$.

The elliptic conjugacy classes in $\absW$ correspond to partitions  $\vec{\ell} = (\ell_\mu \geq \ell_{\mu-1} \geq \cdots \geq \ell_1)$  of $\ell$ where $\mu$ is even~\cite[7.20]{he:minimal}.  Fix such a partition.   Recall that for  $1 \leq \nu \leq \mu$ we define $\ell'_\nu$ = $\ell_1 + \ell_2 + \cdots + \ell_{\nu - 1}$.  We define  ${n}_{\vec{\ell}}$ to be the product $c_1 \cdot c_2 \cdot \cdots \cdot c_\mu$ where
$c_{\nu} = s_{\ell'_{\nu} + 1}\cdot s_{\ell'_{\nu} + 2} \cdot \cdots \cdot s_{\ell'_{\nu} + \ell_{\nu}-1} \cdot t_{\ell'_{\nu} + \ell_{\nu}}$.
Note that $c_{\nu} \in \On_{2 \ell}$, but since $\mu$ is even we have that $n_{\vec{\ell}}$ belongs to  $\SO_{2 \ell}$.
We have
$${n}_{\vec{\ell}} = \prod_{\nu = 1}^{\mu}(s_{\ell'_{\nu} + 1}\cdot s_{\ell'_{\nu} + 2} \cdot \cdots \cdot s_{\ell'_{\nu} + \ell_{\nu}-1} \cdot t_{\ell'_{\nu} + \ell_{\nu}}).$$
Let $\dorbit_{\vec{\ell}}$ denote the conjugacy class  in $\absW$ of the image of ${n}_{\vec{\ell}}$ in  $\absW$.   The correspondence $\vec{\ell} \leftrightarrow \dorbit_{\vec{\ell}}$ defines a bijective correspondence between partitions of $\ell$ into an even number of parts and $\vartheta$-elliptic $\vartheta$-conjugacy classes in $\absW$.

\begin{remark}   \label{rem:regellDl}
Since $w \in \absW$ is regular elliptic if and only if $\hh^{w} = \{0\}$ and $\langle w \rangle$ acts freely on $\Phi$,  the regular elliptic elements of $\absW$  correspond to partitions of the form  $(d,d, \ldots , d)$ where $d \in \Z_{\geq 1}$ divides $\ell$ and $\frac{\ell}{d}$ is even or $(e,e, \ldots , e, e, 1)$ where $e \in \Z_{\geq 1}$ divides $\ell-1$ and $\frac{\ell-1}{e}$ is odd.
\end{remark}

Any element $g \in \SO_{2\ell}$ acts on $\C^{2\ell}$ by matrix multiplication.  The characteristic polynomial of $g$ with respect to this action is an $\SO_{2\ell}$-conjugation invariant polynomial.  Using Lemma~\ref{lem:veryelement}, one finds that  the characteristic polynomial of $n_{\vec{\ell}}$ with respect to this action is
$$q_{n_{\vec{\ell}}} (t)=   \prod_{\nu = 1}^\mu (t^{2 \ell_\nu} - 1).$$
 If $\vec{\ell}$ and $\vec{\ell}'$ are two partitions of $\ell$ for which the associated elements $n_{\vec{\ell}}$
and $n_{\vec{\ell}'}$ of $N_G(A)$ are $G$-conjugate, then we have $q_{n_{\vec{\ell}}} = q_{n_{\vec{\ell}'}}$.  The only way this can happen is if $\vec{\ell} = \vec{\ell}'$.

Since any two lifts of an elliptic element in $\absW$ are $A$-conjugate,  we conclude that two elliptic elements of $\absW$ are $\absW$-conjugate if and only if their lifts are $\SO_{2 \ell}$-conjugate.

\subsection{Injectivity of \texorpdfstring{$\psi_G$ for type $\lsup{2}A_{\ell-1}$ with $\ell \geq 3$}{inject 2A}} 
  \label{sec:inj2Alminusone}

We adopt the notation of Section~\ref{sec:injAlminusone}.  In particular, we have $G = \SL_{\ell}$,  $\gg = \sllie_{\ell}$, and the roots in $\Delta$  are  $\alpha_k := e_k - e_{k+1}$  for  $1 \leq k < \ell$.

  We begin by defining an involution of $G$, and hence $\gg$.  Let $K$ denote the $\ell \times \ell$ matrix with ones on the anti-diagonal and zeroes elsewhere. 
  We define an order two element $J$ as follows.  If $\ell$ is odd, then $J = K$. If $\ell$ is even, then $J$ is the product of $K$ with  $\Diag(i,i,  \ldots, i, i,-i, -i,   \ldots, -i, -i)$; here there are $\ell/2$ copies of $i$ followed by $\ell/2$ copies of $-i$. 
In all cases we have that $\det(J) = (-1)^{\lfloor \frac{\ell}{2} \rfloor}$; that is, $J$ has determinant $1$ if $\ell$ is congruent to $0$ or $1$ modulo $4$ and is $-1$ otherwise.   We also have that $J$, and hence $\Ad(J)$, always has order two.
  We define the involution $\vartheta$ on $G$ by $\vartheta(g) = \Ad(J) (g\inv)^T$ for $g \in G$.  We abuse notation and denote the resulting involution on  $\gg$ by $\vartheta$ as well.  Note that $\vartheta(X) = -\Ad(J)(X^T) = -J {X}^T J$ for $X \in \gg$.

 For $1 \leq k < \ell$ we have $s_k \in \SL_\ell$ as in Section~\ref{sec:injAlminusone} and we define $\tilde{s}_k \in \GL_{\ell}$ to be the $\ell \times \ell$ matrix
\begin{equation*}
(\tilde{s}_k)_{ij}  = \begin{cases}
  1 & \qquad \text{$i=j$ and $i \not \in \{k,k+1\}$} \\
  1 & \qquad \text{$i=k$ and $j = k+1$} \\
  1 & \qquad \text{$i=k+1$ and $j = k$} \\
  0 & \qquad \text{otherwise.}
  \end{cases}  
\end{equation*}
  The image of the order four element $s_k$  in $\absW$ is the simple reflection corresponding to $\alpha_k$.

Suppose
$\vec{\ell} = \ell_\mu \geq \ell_{\mu-1} \geq \cdots \geq \ell_1$ is a partition of $\ell$ for which each $\ell_\nu$ is odd.     Recall that $\ell'_{\nu} = \ell_1 + \ell_2 + \cdots + \ell_{\nu-1}$.  For $1 < \nu \leq \mu$, set
 $$c_\nu = s_{\ell'_\nu + 1} \cdot  s_{\ell'_\nu + 2} \cdot  s_{\ell'_\nu + 3} \cdot \cdots  \cdot  s_{\ell'_\nu + (\ell_\nu -2)} \cdot   s_{\ell'_\nu + (\ell_\nu -1)}.$$
 For $1 < \nu \leq \mu$, the element $c_\nu$ has order $\ell_\nu$.
 If $\ell$ modulo $4$ is $0$ or $1$, then define
$$c_1 = s_{1} \cdot  s_{2} \cdot  s_{3} \cdot \cdots  \cdot  s_{(\ell_1 -2)} \cdot   s_{(\ell_1 -1)}$$
and note that this element has order $\ell_1$.
If $\ell$ modulo $4$ is $2$ or $3$ and $\ell_1 > 1$, then define 
$$c_1 = s_{1} \cdot  s_{2} \cdot  s_{3} \cdot \cdots  \cdot  s_{(\ell_1 -2)} \cdot   \tilde{s}_{(\ell_1 -1)}.$$
 If $\ell$ modulo $4$ is $2$ or $3$ and $\ell_1 = 1$, then define $c_1 = \Diag(-1,1,1,1, \ldots , 1,1,1)$.
 Note that if $\ell$ is congruent to $2$ or $3$ modulo $4$, then the order of $c_1$ is $2 \ell_1$.
 
Define $\tilde{n}_{\vec{\ell}} = c_1 \cdot c_2 \cdot \cdots \cdot c_\mu$.  We have $n_{\vec{\ell}} := \tilde{n}_{\vec{\ell}}J \in \SL_{\ell}$.  Let $w_{\vec{\ell}}$ denote the image of $n_{\vec{\ell}}$ in $\absW$ and let $\dorbit_{\vec{\ell}}$ denote the $\vartheta$-conjugacy class of $w_{\vec{\ell}}$ in $\absW$.   The correspondence $\vec{\ell} \leftrightarrow \dorbit_{\vec{\ell}}$ between partitions of $\ell$ into odd parts and $\vartheta$-elliptic $\vartheta$-conjugacy classes in $\absW$ is a bijection~\cite[7.14]{he:minimal}.
Moreover,  $n_{\vec{\ell}}$ is a lift of $w_{\vec{\ell}}$ into $N_G(A)$.

\begin{remark}   \label{rem:regell2Al}
Since $w \in \absW$ is regular $\vartheta$-elliptic if and only if $\hh^{w \rtimes \vartheta} = \{0\}$ and $\langle w \rtimes \vartheta \rangle$ acts freely on $\Phi$, the regular $\vartheta$-elliptic elements of $\absW$  correspond to partitions of the form  $(d,d, \ldots , d, d, 1)$ where $d \in \Z_{\geq 1}$ is odd and divides $\ell -1$ or $(e,e, \ldots , e)$ where $e \in \Z_{\geq 1}$ is odd and divides $\ell$ 
\end{remark}

We wish to calculate the characteristic polynomial, $q_{n_{\vec{\ell}}}$, of the linear map $T_{\vec{\ell}} \in \Hom(\sllie_{\ell}, \sllie_{\ell})$ defined by $T_{\vec{\ell}}(X) =  \Ad(n_{\vec{\ell}}) \vartheta (X)= -\Ad(n_{\vec{\ell}}J)(X^T) = -\Ad(\tilde{n}_{\vec{\ell}})(X^T)$.
Since $\vartheta  \circ \Ad(g) = \Ad(\vartheta(g)) \circ \vartheta $ for all $g \in \GL_{\ell}$, it follows that  $q_{n_{\vec{\ell}}} = q_{g n_{\vec{\ell}} \vartheta(g\inv)}$ for all $g \in \SL_{\ell}$.

\begin{lemma}  We have
 $$  q_{n_{\vec{\ell}}}(t) = \dfrac{1}{t+1} 
 \cdot  \left( \prod_{\nu=1}^\mu (t^{\ell_\nu} + 1) \right) 
 \cdot  \left(    \prod_{\tau=1}^{\mu} (t^{2\ell_\tau} - 1)^{ (\ell_\tau-1)/2} \right) 
 \cdot \left(    \prod_{1 \leq \rho < \sigma \leq \mu} (t^{2 \lcm(\ell_\rho,\ell_\sigma)} - 1)^{ \gcd(\ell_\rho,\ell_\sigma)} \right). $$
\end{lemma}

\begin{proof}
 By identifying the standard ordered basis of $\C^{\ell_\nu}$ with $(e_{\ell'_\nu +1}, e_{\ell'_\nu +2}, \ldots , e_{\ell'_\nu + \ell_\nu - 1}, e_{\ell'_\nu +\ell_\nu}) $ we have an embedding of $\GL_{\ell_\nu}$ into $\GL_{\ell}$.
 Under this identification, we think of $c_\nu$ as being an element of $\GL_{\ell_\nu}$, and hence $\tilde{n}_{\vec{\ell}} \in \prod_{\nu = 1}^{\mu} \GL_{\ell_\nu}$.  Note that the image of $c_\nu$ in the Weyl group of $\GL_{\ell_\nu}$ is a Coxeter element, hence has order $\ell_\nu$ in this Weyl group.

For $1 \leq \rho, \sigma \leq \mu$ we define the block $B_{\rho,\sigma}$ to be the vector subspace of $\sllie_{\ell}$ consisting of matrices $X$ for which $X_{ij} \neq 0$ if and only if  $i \in \{ \ell'_{\rho} + 1, \ell'_{\rho} + 2, \ldots , \ell'_{\rho} + \ell_{\rho} - 1\}$ 
 and $j \in \{  \ell'_{\sigma} + 1, \ell'_{\sigma} + 2, \ldots , \ell'_{\sigma} + \ell_{\sigma} - 1  \}$.   
We let $S_0$ denote the sum of diagonal blocks, that is $S_0 = \sum_{1 \leq \nu \leq \mu} B_{\nu \nu}$.  We identify  $S_0$   with the Lie algebra of the group of determinant one matrices in $\prod_{\nu = 1}^\mu \GL_{\ell_\nu}$. 
For $1 \leq \rho < \sigma \leq \mu$, we define the vector subspace $S_{\rho \sigma}$ of $\sllie_{\ell}$ by $S_{\rho \sigma} = B_{\rho \sigma} + B_{\sigma \rho}$.   If we think of $\GL_{\ell_\rho} \times GL_{\ell_\sigma}$ as a subgroup of $\prod_{\nu = 1}^\mu \GL_{\ell_\nu} \subset \GL_{\ell}$ in the natural way, then  $S_{\rho \sigma}$ is a  $(\GL_{\ell_\rho} \times GL_{\ell_\sigma}) \rtimes \vartheta'$-module where $\vartheta'(X) = -X^T$.

 Treating $\sllie_{\ell}$ as a $(\prod_{\nu = 1}^\mu \GL_{\ell_\nu}) \rtimes \vartheta'$-module, we have
 $$ \sllie_{\ell} = S_{0} \oplus \bigoplus_{1 \leq \rho  < \sigma \leq \mu} S_{\rho \sigma}  .$$
The polynomial $q_{n_{\vec{\ell}}}$ is the product of the characteristic polynomials for the action of $\tilde{n}_{\vec{\ell}} \rtimes \vartheta'$ on each of the $(\mu^2 - \mu)/2 + 1$ summands above.  We now calculate these characteristic polynomials.

The $(\GL_{\ell_{\rho}} \times \GL_{\ell_{\sigma}}) \rtimes \vartheta'$-module $S_{\rho \sigma}$ is also an $A$-module and, as such, may be written as a direct sum of root spaces.  The image of $c_{\rho} c_{\sigma} \rtimes \vartheta'$ in the twisted Weyl group of $(\GL_{\ell_\rho} \times \GL_{\ell_\sigma}) \rtimes \vartheta'$ acts  freely on these roots, with orbits of size $2 \lcm(\ell_\rho,\ell_\sigma)$.   Since the dimension of $S_{\rho \sigma}$ is $ 2 \ell_\rho \ell_\sigma$, there are $\gcd(\ell_\rho,\ell_\sigma)$ of these orbits.   Since the image of $(c_{\rho} c_{\sigma} )^{2\lcm(\ell_\rho,\ell_\sigma)}$ in $A$ is trivial, it follows from Lemma~\ref{lem:veryelement} that
 the characteristic polynomial of $\Ad(c_{\rho} c_{\sigma}) \rtimes \vartheta'$, hence of $\Ad(\tilde{n}_{\vec{\ell}}) \rtimes \vartheta'$, acting on $S_{\rho,\sigma}$ is $(t^{2\lcm(\ell_\rho,\ell_\sigma)} -1)^{\gcd(\ell_\rho,\ell_\sigma)}$.

We now calculate the characteristic polynomial for the action of  $\tilde{n}_{\vec{\ell}} \rtimes \vartheta'$ acting on $S_{0}$.   Since $\tilde{n}_{\vec{\ell}} \rtimes \vartheta'$ fixes the center of  $\gllie_{\ell_\nu}$, this characteristic polynomial 
is equal to $\frac{1}{t+1}$ times the product for $1 \leq \nu \leq \mu$ of the characteristic polynomials for the action of $c_{\nu} \rtimes \vartheta'$ on $\gllie_{\ell_\nu}$.    Since the transpose operation is trivial on $\mathfrak{d}_\nu$, the set of diagonal elements in $\gllie_{\ell_\nu}$, the action of  $c_{\nu} \rtimes \vartheta'$ factors through to an action of $w_\nu \rtimes -1$  where $w_\nu$ is a Coxeter element in $\absW_\nu$, the Weyl group of $\GL_{\ell_\nu}$. Thus the characteristic polynomial for the action $c_{\nu} \rtimes \vartheta'$ on $\mathfrak{d}_\nu$ is $(t^{\ell_\nu} + 1)$.     The image of $c_{\nu} \rtimes \vartheta'$ in $\absW_\nu \rtimes \vartheta'$ acts freely on the roots of $\GL_{\ell_\nu}$ with respect to $A$.  Since the orbits for this action have $2 \ell_\nu$ elements and  the dimension of $\gllie_{2 \ell_\nu} / \mathfrak{d}_\nu$ is $\ell_\nu^2 -\ell_\nu$, there are $(\ell_\nu - 1)/2$ of these orbits.   Since  $c_{\nu}^{2\ell_\nu}$  is trivial, it follows from Lemma~\ref{lem:veryelement} that the characteristic polynomial  for the action $c_{\nu} \rtimes \vartheta'$ on $\gllie_{\ell_\nu}(\C)/\mathfrak{d}_\nu$ 
is $(t^{2\ell_\nu} -1)^{(\ell_\nu -1)/2}$.   Thus,  the characteristic polynomial for the action of  $\tilde{n}_{\vec{\ell}} \rtimes \vartheta'$ acting on $S_{0}$ is
$ \frac{1}{t+1} \prod_{\nu =1}^\mu (t^{\ell_\nu} + 1) \prod_{\tau =1}^\mu (t^{2 \ell_\tau} -1)^{(\ell_\tau -1)/2}$.
\end{proof}

\begin{lemma}  \label{lem:polygivesl}
If $\vec{\ell}$ and $\vec{\ell}'$ are two partitions of $\ell$ with $q_{n_{\vec{\ell}}}= q_{n_{\vec{\ell}'}}$, then $\vec{\ell} = \vec{\ell}'$.
\end{lemma}

\begin{proof}
Recall that $\vec{\ell}$ is a partition of $\ell$ into odd parts.
List out all the roots of $q_{n_{\vec{\ell}}}$ and call the list $S$. The list $S$ has $\dim(\gg)$ entries. There exists an $m \in \N$ so that every element of $S$ looks like $\xi^\nu_{2m}$  where $\xi_{2m}$ is a fixed primitive $(2m)^\text{th}$ root of unity and $\nu \in \N$.  Choose the largest odd $h \in \N$ for which there exists a primitive $(2h)^{\text{th}}$ root of unity, call it $\xi$, in $S$ with the property that every power of $\xi$ occurs in $S$.  For such a $\xi$ remove one copy of each power of $\xi$ from $S$ -- note, we are only removing $2h$ elements from $S$.  Continue this process until it cannot be repeated -- that is, until the list $S$ consists of roots of unity $\xi$ such that $\xi^j$ is not in $S$ for some $j$. The number of times this process can happen is equal to the number of times one appears as a root of $q_{n_{\vec{\ell}}}$.   

There will be $\ell-1$ roots remaining, and these will be the roots of  
\begin{equation} \label{equ:pl}
p_{\vec{\ell}}(t) := \frac{1}{t+1} \prod_{\nu = 1}^\mu (t^{\ell_\nu} + 1).
\end{equation}
That is, this algorithm recovers $p_{\vec{\ell}}$ from $q_{n_{\vec{\ell}}}$, and hence it recovers $\vec{\ell}$.
\end{proof}

By construction, we have $q_{n_{\vec{\ell}}} =  q_{xn_{\vec{\ell}}\vartheta(x\inv) }$  for all $x \in G$.  Since any two lifts of a $\vartheta$-elliptic element in $\absW$ to elements in $N_G(A) \rtimes \vartheta$ are $\vartheta$-conjugate by an element of $A$, we conclude from Lemma~\ref{lem:polygivesl} that two $\vartheta$-elliptic elements of $\absW$ are $\vartheta$-conjugate in $\absW$ if and only if their lifts into $N_G(A)$ are $\vartheta$-conjugate in $\SL_{\ell}$.

\subsection{Injectivity of \texorpdfstring{$\psi_G$ for type $\lsup{2}D_{\ell +1}$ with $\ell \geq 2$}{inject 2D}} 
 \label{sec:inj2Dlplusone} 
Let $G = \SO_{2\ell +2}$ and let $\gg$ be  $\solie_{2 \ell+2}$  realized as in Section~\ref{sec:injDl}.  We take $A$ to be the diagonal matrices in $G$ and $B$ to be the upper triangular matrices in $G$.  We realize $V$, $\Phi$, and $\Delta$ as in Section~\ref{sec:injDl}.  In particular, the elements of $\Delta$ are
$$\alpha_1 = e_1 - e_2,  \,  \alpha_2 = e_2 - e_3,  \, \ldots , \, \alpha_{\ell-1} = e_{\ell-1} - e_\ell, \alpha_{\ell} = e_{\ell} - e_{\ell+1}, \, \alpha_{\ell+1} = e_{\ell} + e_{\ell+1}.$$

Let $J$ be the $(2 \ell +2) \times (2 \ell+2)$ matrix such that $J_{ii} = 1$ for $i \neq \ell+1, \ell+2$, $J_{\ell+1,\ell+2} = J_{\ell+2,\ell+1} = 1$, and all other entries are zero.  Then $J \in \On_{2 \ell+2}$.   We define an involution $\vartheta$ on $\gg$ by $\vartheta(X) = JXJ$.

For $1 \leq k \leq \ell$ we define $s_k \in \SO_{2 \ell +2}$ to be the $(2 \ell +2 ) \times (2 \ell +2)$ matrix such that 
$$(s_k)_{ij} = 
\begin{cases}
    1 &\qquad \text{$i=j$ and $i \not\in \{k,k+1,2\ell -k+3, 2\ell-k +2 \}$}\\
    1 &\qquad \text{$i \neq j$ and  $i,j \in \{k,k+1\}$}\\
     1 &\qquad \text{$i \neq j$ and  $i,j \in \{2 \ell - k+2, 2 \ell - k+3\}$}  \\
     0 &\qquad \text{otherwise.}
\end{cases}$$
The image of the order two element  $s_k$ in $\absW$ is the simple reflection corresponding to $\alpha_k$.    For $1 \leq k \leq \ell+1$ define $t_k \in \On_{2 \ell+2}$ by 
$$t_k = s_k \cdot s_{k+1} \cdot s_{k+2} \cdot \cdots \cdot s_{\ell-2} \cdot s_{\ell -1} \cdot s_{\ell} \cdot J \cdot s_{\ell} \cdot s_{\ell-1} \cdot s_{\ell-2} \cdot  \cdots  \cdot s_{k+1} \cdot s_k.$$
Note that $t_{\ell+1} = J$.

The $\vartheta$-elliptic conjugacy classes in $\absW$ correspond to partitions  $\vec{\ell} = (\ell_\mu \geq \ell_{\mu-1} \geq \cdots \geq \ell_1)$  of $\ell+1$ where $\mu$ is odd~\cite[7.20]{he:minimal}.  Fix such a partition.   Recall that for $1 \leq \nu \leq \mu$ we define $\ell'_\nu$ = $\ell_1 + \ell_2 + \cdots + \ell_{\nu - 1}$.   We define  $\tilde{n}_{\vec{\ell}}$ to be the product $c_1 \cdot c_2 \cdot \cdots \cdot c_\mu$ where
$c_{\nu} = s_{\ell'_{\nu} + 1}\cdot s_{\ell'_{\nu} + 2} \cdot \cdots \cdot s_{\ell'_{\nu} + \ell_{\nu}-1} \cdot t_{\ell'_{\nu} + \ell_{\nu}}$.
Thus
$$\tilde{n}_{\vec{\ell}} = \prod_{\nu = 1}^{\mu}(s_{\ell'_{\nu} + 1}\cdot s_{\ell'_{\nu} + 2} \cdot \cdots \cdot s_{\ell'_{\nu} + \ell_{\nu}-1} \cdot t_{\ell'_{\nu} + \ell_{\nu}}).$$
Let $\dorbit_{\vec{\ell}}$ denote the $\vartheta$-conjugacy class  in $\absW$ of the image of ${n}_{\vec{\ell}} = \tilde{n}_{\vec{\ell}}J$ in  $\absW$.   The correspondence $\vec{\ell} \leftrightarrow \dorbit_{\vec{\ell}}$ defines a bijective correspondence between partitions of $\ell+1$ with an odd number of parts and $\vartheta$-elliptic $\vartheta$-conjugacy classes in $\absW$.

\begin{remark}   \label{rem:regell2Dl}
Since $w \in \absW$ is regular $\vartheta$-elliptic if and only if $\hh^{w \rtimes \vartheta} = \{0\}$ and $\langle w \rtimes \vartheta \rangle$ acts freely on $\Phi$, the regular $\vartheta$-elliptic elements of $\absW$ correspond to partitions of the form $(d,d, \ldots , d)$ where $d \in \Z_{\geq 1}$ divides $\ell +1$ and $\frac{\ell + 1}{d}$ is odd or $(e,e, \ldots , e,e,1)$ where $e \in \Z_{\geq 1}$ divides $\ell$ and $\frac{\ell}{e}$ is even. 
\end{remark}

Any element $g \in \SO_{2\ell+2}$ acts on $\C^{2\ell+2}$ by matrix multiplication.  The characteristic polynomial, call it $q_{g}$, of $gJ$ with respect to this action is an $\SO_{2\ell+2}$-conjugation invariant polynomial, hence $q_{g} = q_{x g \vartheta(x\inv)}$ for all $g, x \in \SO_{2 \ell+2}$.  
Using Lemma~\ref{lem:veryelement}, one finds that 
$$q_{n_{\vec{\ell}}}(t) = \prod_{\nu = 1}^\mu (t^{2 \ell_\nu} - 1).$$
 If $\vec{\ell}$ and $\vec{\ell}'$ are two partitions of $\ell$ for which the associated elements $n_{\vec{\ell}}$
and $n_{\vec{\ell}'}$ of $N_G(A)$ are $\vartheta$-conjugate in $G$, then we have $q_{n_{\vec{\ell}}} = q_{n_{\vec{\ell}'}}$.  The only way this can happen is if $\vec{\ell} = \vec{\ell}'$.

Since any two lifts of an elliptic element in $\absW \rtimes \vartheta$ into $N_G(A) \rtimes \vartheta$ are $A$-conjugate,  we conclude that two $\vartheta$-elliptic elements of $\absW$ are $\vartheta$-conjugate in $\absW$ if and only if their lifts into $N_G(A)$ are $\vartheta$-conjugate in $\SO_{2 \ell+2}$.

\subsection{Injectivity of \texorpdfstring{$\psi_G$ for the remaining types of groups}{inject rest}} 
\label{sec:injexceptional}

Suppose $G$ is a simple adjoint group of type $F_4$, $G_2$, $E_6$, $E_7$, $E_8$, $\lsup{3}D_4$, or $\lsup{2}E_6$.  An inspection of the tables in~\cite[Section~9]{adams-he-nie:from}  shows that the map $\zeta$ is injective.   Since $G$ is adjoint, we know $\chi_G$ is bijective.   Since $\zeta = \chi_G \circ \psi_G$, we conclude that $\psi_G$ must be injective.  

From Lemma~\ref{lem:isogenyindependent} we conclude that $\psi_G$ must be injective independent of the isogeny type of $G$.  In particular, $\psi_G$ is injective when $G$ is an almost-simple  group of type $F_4$, $G_2$, $E_6$, $E_7$, $E_8$, $\lsup{3}D_4$, or $\lsup{2}E_6$.

\section{Kac diagrams for \texorpdfstring{$\vartheta$-elliptic conjugacy classes in $\absW$}{elliptic diagrams}}  \label{sec:kacdiagrams} 

In this section $\gg$ is a simple complex Lie algebra and $G$ is a semisimple complex group with Lie algebra $\gg$.  Suppose  the pinned automorphism $\vartheta$ has order $f$.

\subsection{Notation for Kac diagrams}  We begin by recalling the setup of~\cite{reeder:torsion} for Kac diagrams.

Let $(\check{\omega}_\alpha)_{\alpha \in \Delta} \subset V = \X_*(A) \otimes \R$ be the fundamental coweights with respect to $\Delta$.  That is, $\check{\omega}_\alpha(\beta) = \delta_{\alpha \beta}$ for $\alpha, \beta \in \Delta$.   Here $\delta_{\alpha \beta} = 1$ if $\alpha = \beta$ and is zero otherwise.

For $\alpha \in \Phi = \Phi(G,A)$, define $\bar{\alpha} = \res_{V^\vartheta} \alpha$ and set
$\gamma_\alpha =  \sum \bar{\beta}$ where the sum is over those $\beta  \in \Phi $ such that $\bar{\beta} \in \R_{>0} \bar{\alpha}$.   The set $\Phi_\vartheta := \{\gamma_\alpha \, | \, \alpha \in \Phi \}$ is the reduced root system for $G^\vartheta$ with respect to $A^\vartheta$.  The set $\Delta_\vartheta = \{ \gamma_\alpha  \, | \, \alpha \in \Delta\}$ is a basis for $\Phi_\vartheta$ with respect to $G^\vartheta$, $B^{\vartheta}$, and $A^\vartheta$.    If $\vartheta$ is nontrivial, then $\dabs{\Delta_\vartheta} < \dabs{\Delta}$.

The Weyl group  $N_{G^\vartheta}(A)/A^\vartheta$ is $\absW^\vartheta$, and $V^\vartheta$ is the reflection representation for $\absW^\vartheta$.    As before we have an affine simplicial decomposition of $V^\vartheta$, and the set of simple roots $\Delta_\vartheta$ determines a unique alcove $C_\vartheta$ that contains the origin of $V^\vartheta$.   Note that $C_\vartheta \cap C^\vartheta \neq \emptyset$, and it can happen that $C_{\vartheta}$ is not $C^\vartheta$. 

\subsubsection{The constants \texorpdfstring{$b_\gamma$ and $c_\gamma$}{constants b and c}}

 Let $\tilde{\gamma}_0$ denote the highest root in $\Phi_\vartheta$ with respect to $\Delta_{\vartheta}$ and write $\tilde{\gamma}_0 = \sum_{\gamma \in \Delta_{\vartheta}} c_{\gamma} \gamma$ for $c_\gamma \in \Z_{>0}$. 
The constants $c_\gamma$ may be found in~\cite[Planche I -- IX]{bourbaki:lie4to6}.

For $\gamma \in \Delta_{\vartheta}$ choose $\alpha \in \Delta$ such that $\gamma = \gamma_\alpha$.  Let $\dabs{\gamma}_\vartheta$ denote the size of the $\langle \vartheta \rangle$-orbit of $\alpha$ and set
$$\check{\mu}_{\gamma} = \frac{1}{f} \sum_{j=0}^{f -1} \vartheta^j \check{\omega}_\alpha.$$  
For $\gamma \in \Delta_{\vartheta}$ we define $b_{\gamma} \in \Z_{>0}$ by requiring
$$\check{\mu}_{\gamma} (\rho) = \frac{b_{\gamma} f}{c_\gamma \dabs{\gamma}_{\vartheta}} \delta_{\gamma \rho}$$
for all $\rho \in \Delta_{\vartheta}$.    If $\vartheta$ is trivial, then $\Delta = \Delta_\vartheta$, $f = 1$, $\dabs{\gamma}_\vartheta = 1$, and $\check{\mu}_\gamma = \check{\omega}_\gamma$ for all $\gamma \in \Delta_\vartheta = \Delta$.  It follows that  $b_\gamma = c_\gamma$ for all $\gamma \in \Delta$.  If $\vartheta$ is not trivial, then  the constants $b_\gamma$ may be found in, for example,~\cite[Table~2]{reeder:torsion}.  Whenever we need to know the values of $b_\gamma$, we will provide them.  

  If we define $\gamma_0 = 1 - \tilde{\gamma}_0$, then $\tilde{\Delta}_\vartheta = \Delta \cup \{\gamma_0\}$ is a basis with respect to $C_{\vartheta}$ for the affine roots determined by $G^\vartheta$ and $A^\vartheta$.  We set $c_{\gamma_0} = b_{\gamma_0} = 1$.

\subsubsection{Kac diagrams}
As discussed in the introduction, if $g \in G$ has order $m$, then $g$ determines a point, call it $x_g$, in the closure of $C_\vartheta$.  A Kac diagram is a labeling of the affine Dynkin diagram that encodes the location of this point.  We explain how this works.

Let $(v_{\gamma} \, | \, \gamma \in \tilde{\Delta}_\vartheta)$ denote the vertices of $C_{\vartheta}$.   We let $v_{\gamma_0}$ be the origin and then 
$$v_\gamma =\dfrac{\dabs{\gamma}_\vartheta}{b_\gamma f} \cdot \check{\mu}_\gamma $$
for $\gamma \in \Delta_{\vartheta}$.   Every element $x$ in the closure  of $C_\vartheta$ may be written in barycentric coordinates as 
$$x = \sum_{\gamma \in \tilde{\Delta}_\vartheta} x_\gamma v_\gamma$$
where $\sum x_\gamma = 1$ and $x_\gamma \geq 0$ for all $\gamma$.  From, for example~\cite[Theorem 3.7]{reeder:torsion}, these conditions imply the existence of $s_\gamma \in \Z_{\geq 0}$ such that 
\begin{equation} \label{equ:equationxg}
x_g =  \frac{f}{m} \sum_{\gamma \in \tilde{\Delta}_\vartheta} s_\gamma b_\gamma \cdot v_\gamma =  \frac{f}{m} \sum_{\gamma \in {\Delta}_\vartheta} s_\gamma b_\gamma \cdot v_\gamma  =  \frac{1}{m} \sum_{\gamma \in {\Delta}_\vartheta} s_\gamma \dabs{\gamma}_{\vartheta} \cdot \check{\mu}_{\gamma}.
\end{equation}
We will also have
\begin{equation} \label{equ:equationform}
m = f \sum_{\gamma \in \tilde{\Delta}_\vartheta} s_\gamma b_\gamma.
\end{equation}

The list $(s_\gamma \, | \, \gamma \in \tilde{\Delta}_\vartheta)$ determines a Kac diagram attached to $g$ as follows.  First, remove any factors that are common to all of the $s_\gamma$ so that we may assume that the  $s_\gamma$ are relatively prime.  Then label  the node in the affine Dynkin diagram  corresponding to $\gamma \in \tilde{\Delta}_\vartheta$  with the corresponding non-negative integer $s_\gamma$.  

\begin{remark}  Suppose we are given a Kac diagram for $G^\vartheta$ with labels $(s_\gamma \, | \, \gamma \in \tilde{\Delta}_\vartheta)$.  If we  \emph{define} $m$ using Equation~\ref{equ:equationform} and these $s_\gamma$, then Equation~\ref{equ:equationxg} determines a point in $C_\vartheta$.
\end{remark}

\begin{remark} If $G$ is adjoint, then the $s_\gamma$ are automatically relatively prime and constructing an element $g$ from a given Kac diagram is straightforward.  However, if $G$ is not adjoint, then the condition that the labels of the Kac diagram be relatively prime is often too restrictive -- there may not even be an element of order $m =f \sum s_\gamma b_\gamma$ in $G$.   In this case one can first construct an element in the adjoint form of $G$ and then lift this element into $G$.
\end{remark}

\begin{remark}  \label{rem:decreasing}
In this paper we restrict our attention to finite order elements of $G$ that normalize $A$ and have $\vartheta$-elliptic image in $\absW$.
Let $n$ be such an element, and assume it has order $m$.  Let $\xi$ be a primitive $m^{th}$ root of unity. In this case, there exists $\lambda \in \X_*(A^{\vartheta})$ such that $n$ is $\vartheta$-conjugate to $\lambda(\xi)$ and the point $x_n \in C_\vartheta$ is equal to $\lambda/m$  (see, for example,~\cite[Lemma~3.5.1]{debacker:totally}).  From Equation~(\ref{equ:equationxg}) we must have $\lambda = \sum_{\gamma \in {\Delta}_\vartheta} s_\gamma \dabs{\gamma}_{\vartheta} \cdot \check{\mu}_{\gamma}$ with $s_\gamma \dabs{\gamma}_{\vartheta} \geq 0$.  In terms of the explicit realizations of $G$ we have adopted, this means that this $\lambda$ must look like 
$$\lambda(t) = \Diag (t^{a_1}, t^{a_2}, \ldots )$$
for $a_i \in \Z$ with $a_1 \geq a_2 \geq \cdots$.
\end{remark}

Remark~\ref{rem:decreasing} will be used repeatedly in Sections~\ref{sec:kacdiagramAl} --~\ref{sec:kacdiagram2Dlpo}.

\subsection{How to create the Kac diagram for the elliptic conjugacy class in \texorpdfstring{$\absW$ for $A_{\ell-1}$ with $\ell \geq 2$}{kac for A}}
\label{sec:kacdiagramAl}
We adopt the notation of Section~\ref{sec:injAlminusone}. 
As discussed earlier, there is only one elliptic conjugacy class in the Weyl group.  We will show how to construct the Kac diagram associated to this elliptic conjugacy class.  We adopt the notation of Section~\ref{sec:injAlminusone}.

Since $\vartheta$ is trivial, we have $f = 1$.  We take the simple roots $\gamma$ in $\Delta_\vartheta = \Delta$ to be the roots $\gamma_k = \alpha_k$ for $1 \leq k < \ell$.     For coweights we take $\check{\mu}_i = e_1 + e_2 + \cdots + e_i - \frac{i}{\ell} (e_1 + e_2 + \cdots + e_\ell)$ for $1 \leq i < \ell$.

The affine Dynkin diagram is 
$$
\begin{tikzpicture}[start chain]
\unodenj{0}
\unodenj{1}
\path (chain-1) -- node[anchor=mid] {\(\Longleftrightarrow\)} (chain-2);
\end{tikzpicture}
$$
for $\ell = 2$, and it is 
$$
\begin{tikzpicture}[start chain,node distance=1ex and 2em]
\unode{1}
\unode{2}
\dydots
\unode{\ell-2}
\unode{\ell-1}
\begin{scope}[start chain=br going below]
\chainin(chain-3);
\node[ch,join=with chain-1,join=with chain-5,label={below:\(\gamma_0\)}] {};
\end{scope}
\end{tikzpicture}
$$
for $\ell > 2$.   We have $ b_k = c_k = 1$ for $0 \leq k < \ell$, so the $b_\gamma$ are given by the diagram
$$
\begin{tikzpicture}[start chain]
\enodenj{1}
\enodenj{1}
\path (chain-1) -- node[anchor=mid] {\(\Longleftrightarrow\)} (chain-2);
\end{tikzpicture}
$$
for $\ell = 2$, and by the diagram
$$
\begin{tikzpicture}[start chain,node distance=1ex and 2em]
\enode{1}
\enode{1}
\dydots
\enode{1}
\enode{1}
\begin{scope}[start chain=br going below]
\chainin(chain-3);
\node[noch,join=with chain-1,join=with chain-5,label={center:\(1\)}] {};
\end{scope}
\end{tikzpicture}
$$
for $\ell > 2$.

The elliptic-conjugacy class in $\absW$ corresponds to the partition $(\ell)$ of $\ell$, and we constructed a representative $n_{(\ell)}$ for this conjugacy class in Section~\ref{sec:injAlminusone}. 
 The characteristic polynomial of $n_{({\ell})}$ with respect to its natural action on $\C^\ell$ is
$ q_{n_{(\ell)}}(t) = t^{\ell} + (-1)^{\ell}$.   We need to consider two cases.

\subsubsection{\texorpdfstring{$\ell$ is even}{ell is even}}
Since  $\ell$ is even we have  $ q_{n_{(\ell)}}(t) = t^{\ell} + 1$.
Let $m = 2 \ell$ and let $\xi$ be a primitive $m^{\text{th}}$ root of unity. We have
$$q_{n_{(\ell)}}(t) = \prod_{a = 1}^{\ell/2} (t - \xi^{2 a - 1 })(t - \xi^{-(2 a - 1) }).$$
Guided by  Remark~\ref{rem:decreasing} we create the length $\ell$ decreasing list $(\ell - 1,  \ell - 3, \ldots, 3,1,-1,-3, \ldots  , -(\ell-1)    )$. Denote by $\scoeff_i$ the $i^{\text{th}}$ item in this list to obtain a decreasing list $(\scoeff_1 \geq \scoeff_2 \geq \scoeff_3 \geq \cdots \geq \scoeff_{\ell-1} \geq \scoeff_\ell)$.  The element $d_{(\ell)} = \Diag(\xi^{\scoeff_1},  \xi^{\scoeff_2}, \ldots , \xi^{\scoeff_{\ell -1}}, \xi^{\scoeff_{\ell}} )$ in $\SL_{\ell}$ has characteristic polynomial $q_{n_{(\ell)}}$  for its standard action on  $\C^{\ell}$. 

 Since the linear factors $(t-\xi^{\scoeff_j})$ for $1 \leq j \leq \ell$ must occur in $q_{n_{(\ell)}}$,   we conclude that, up to the action of $\absW$, $d_{(\ell)}$ is the unique element of $A$ that is $G$-conjugate to $n_{(\ell)}$.

We can now read off the Kac diagram for $n_{(\ell)}$ from $d_{(\ell)}$. 
Note that $d_{(\ell)} = \lambda_{({\ell})}(\xi)$ where 
$$\lambda_{({\ell})} =  2 \check{\mu}_1 + 2 \check{\mu}_2  +  \ldots  + 2 \check{\mu}_{\ell-1}.$$
Since  $\dabs{\gamma}_\vartheta$ is $1$ for all $\gamma \in \Delta_\vartheta$ and 
$$\lambda_{({\ell})}/m = 
\frac{1}{m} [ 2 \check{\mu}_1 + 2 \check{\mu}_2  +  \ldots  + 2 \check{\mu}_{\ell-1}],$$
from Equation~\ref{equ:equationxg} we conclude that the $s_\gamma$ are all $2$ for $\gamma \in \Delta$.  In order to satisfy Equation~\ref{equ:equationform} we must also have $s_{\gamma_0} = 2$. Dividing all terms by $2$, we  have:

\begin{lemma}
The Kac diagram for the elliptic conjugacy class in a group of type $A_{\ell-1}$ with $\ell \geq 2$ even is the same as the diagram for $b_\gamma$ above.  
\end{lemma}

\begin{example}        
In Figure~ \ref{fig:A1location} for groups of type $A_1$ we show the location of the vertices $v_0$ and $v_1$ of the fundamental alcove as well as the point determined by $\lambda_{(2)}/4$. The Kac diagram for $n_{(2)}$ is 
\(
\begin{tikzpicture}[start chain]
\enodenj{1}
\enodenj{1}
\path (chain-1) -- node[anchor=mid] {\(\Longleftrightarrow\)} (chain-2);
\end{tikzpicture}.
\)
\begin{figure}[ht]
\centering
\begin{tikzpicture}

\draw (-0.1,0)   -- (5.296,0);
\draw[black,fill=white] (0,0) circle (.5ex);
\draw(0.2,-0.6) node[anchor=south]{$v_0$};

\draw[black,fill=white] (5.196,0) circle (.5ex);
\draw(5.396,-0.6) node[anchor=south]{$v_1$};

\draw(2.981,0) node[anchor=north]{$\frac{\lambda_{(2)}}{4}$};
\draw[black,fill=black] (2.598,0) circle (.5ex);

\end{tikzpicture}
\caption{The location of the point determined by  $\lambda_{(2)}/4$ for groups of type $A_1$ \label{fig:A1location}}
\end{figure}
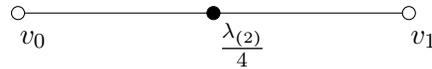
\end{example}

\subsubsection{\texorpdfstring{$\ell$ is odd}{ell is odd}}
Since  $\ell$ is odd we have  $ q_{n_{(\ell)}}(t) = t^{\ell} - 1$.
Let $m = \ell$ and let $\xi$ be a primitive $m^{\text{th}}$ root of unity. We have
$$q_{n_{(\ell)}}(t) = (t-1) \prod_{a = 0}^{(\ell-1)/2} (t - \xi^{a})(t - \xi^{-a}).$$
Guided by  Remark~\ref{rem:decreasing} we create the length $\ell$ decreasing list $((\ell - 1)/2,  (\ell - 3)/2, \ldots, 1,0,-1, \ldots ,  -(\ell -1)/2   )$.    Denote the  $i^{\text{th}}$ item in this list by $\scoeff_i$ to obtain a decreasing list $(\scoeff_1 \geq \scoeff_2 \geq \scoeff_3 \geq \cdots \geq \scoeff_{\ell-1} \geq \scoeff_\ell)$.  The element $d_{(\ell)} = \Diag(\xi^{\scoeff_1},  \xi^{\scoeff_2}, \ldots , \xi^{\scoeff_{\ell -1}}, \xi^{\scoeff_{\ell}} )$ in $\SL_{\ell}$ has characteristic polynomial $q_{n_{(\ell)}}$  for its standard action on  $\C^{\ell}$. 

 Since the linear factors $(t-\xi^{\scoeff_j})$ for $1 \leq j \leq \ell$ must occur in $q_{n_{(\ell)}}$,   we conclude that, up to the action of $\absW$, $d_{(\ell)}$ is the unique element of $A$ that is $G$-conjugate to $n_{(\ell)}$.

We can now read off the Kac diagram for $n_{(\ell)}$ from $d_{(\ell)}$. 
Note that $d_{(\ell)} = \lambda_{({\ell})}(\xi)$ where 
$$\lambda_{({\ell})} =   \check{\mu}_1 +  \check{\mu}_2  +  \ldots  +  \check{\mu}_{\ell-1}.$$
Since  $\dabs{\gamma}_\vartheta$ is $1$ for all $\gamma \in \Delta_\vartheta$ and 
$$\lambda_{({\ell})}/m = 
\frac{1}{m} [  \check{\mu}_1 +  \check{\mu}_2  +  \ldots  + \check{\mu}_{\ell-1}],$$
from Equation~\ref{equ:equationxg} we conclude that the $s_\gamma$ are all $1$ for $\gamma \in \Delta$.  In order to satisfy Equation~\ref{equ:equationform} we must also have $s_{\gamma_0} = 1$. Thus, we have:

\begin{lemma}
The Kac diagram for the elliptic conjugacy class in a group of type $A_{\ell-1}$ with $\ell \geq 3$ odd is the same as the diagram for $b_\gamma$ above.  
\end{lemma}

\begin{example}  
In Figure~ \ref{fig:A2location} for groups of type $A_2$ we show the location of the vertices $v_0$, $v_1$, and $v_2$ of the fundamental alcove as well as the point determined by $\lambda_{(3)}/3$. The Kac diagram for $n_{(3)}$ is
\(
\begin{tikzpicture}[start chain,node distance=1ex and 2em]
\enode{1}
\enode{1}
\begin{scope}[start chain=br going below]
\chainin(chain-2);
\node[noch,join=with chain-1,join=with chain-2,label={center:\(1\)}] at (.5,-0.1) {};
\end{scope}
\end{tikzpicture}
\)
\begin{figure}[ht]
\centering
\begin{tikzpicture}

\draw (0,0)   -- (2.5,4.33) --  (5,0) -- cycle;
\draw[black,fill=white] (0,0) circle (.5ex);
\draw(0.2,-0.5) node[anchor=south]{$v_0$};

\draw[black,fill=white] (5,0) circle (.5ex);
\draw(5.196,-0.5) node[anchor=south]{$v_1$};

\draw[black,fill=white] (2.5,4.33) circle (.5ex);
\draw(2.5,4.33) node[anchor=south]{$v_2$};

\draw(2.85,1.44) node[anchor=north]{$\frac{\lambda_{(3)}}{3}$};
\draw[black,fill=black] (2.5,1.44) circle (.5ex);

\end{tikzpicture}
\caption{The location of the point determined by  $\lambda_{(3)}/3$ for groups of type $A_2$ \label{fig:A2location}}
\end{figure}
\end{example}

\subsection{How to create a Kac diagram for \texorpdfstring{$n_{\vec{\ell}}$ for $B_\ell$ with $\ell \geq 3$}{kac B}}  
We adopt the notation of Section~\ref{sec:injBl}.  Fix a partition $\vec{\ell}$ of $\ell$.   We will show how to construct the Kac diagram for   $n_{\vec{\ell}}$.

Since $\vartheta$ is trivial, we have $f = 1$.  We take the simple roots $\gamma$ in $\Delta_\vartheta = \Delta$ to be the roots $\gamma_k = \alpha_k$ for $1 \leq k \leq \ell$.     For coweights we take $\check{\mu}_i = e_1 + e_2 + \cdots + e_i $ for $1 \leq i \leq \ell$.

The affine Dynkin diagram is 
\[
\begin{tikzpicture}
\begin{scope}[start chain]
\unode{1}
\unode{2}
\unode{3}
\dydots
\unode{\ell - 2}
\unode{\ell - 1}
\unodenj{\ell }
\path (chain-6) -- node{\(\Rightarrow\)} (chain-7);
\end{scope}
\begin{scope}[start chain=br going below]
\chainin(chain-2);
\dnodebr{0};
\end{scope}
\end{tikzpicture}
\]
   We have $ b_{\gamma_k} = c_{\gamma_k}$ for $0 \leq k < \ell$, so the $b_\gamma$ are given by the diagram
\[
\begin{tikzpicture}
\begin{scope}[start chain]
\enode{1}
\enode{2}
\enode{2}
\enode{2}
\dydots
\enode{2}
\enode{2}
\enodenj{2}
\path (chain-7) -- node{\(\Rightarrow\)} (chain-8);
\end{scope}
\begin{scope}[start chain=br going below]
\chainin(chain-2);
\enodebr{1};
\end{scope}
\end{tikzpicture}
\]

A Kac diagram describes a diagonal matrix in $\SO_{2 \ell +1}$ that is $G$-conjugate to $n_{\vec{\ell}}$.  Thus, we want to find a diagonal matrix in $\SO_{2 \ell +1}$ such that its characteristic polynomial for the standard action on  $\C^{2\ell+1}$ is  
$$q_{n_{\vec{\ell}}} (t)=  (t-(-1)^\mu) \cdot \prod_{\nu = 1}^\mu (t^{2 \ell_\nu} - 1).$$

Let $m = 2\lcm(\ell_1, \ell_2, \ldots , \ell_\mu)$ and let $\xi$ be a primitive $m^{\text{th}}$ root of unity. We have
$$q_{n_{\vec{\ell}}}(t) = (t-(-1)^{\mu}) \cdot \prod_{\nu = 1}^\mu (t^{2 \ell_\nu} - 1) = (t-(-1)^{\mu}) \cdot  \prod_{\nu = 1}^\mu (t-1)(t+1) \prod_{a_{\nu} = 1}^{\ell_\nu - 1} (t - \xi^{ma_{\nu}/2 \ell_\nu} )(t - \xi^{-ma_{\nu}/2 \ell_\nu} ).$$
Guided by  Remark~\ref{rem:decreasing} we create a length $\ell$ decreasing list as follows: order the positive integers $ma_{\nu}/2 \ell_\nu$ in decreasing order, then pre-append $\lceil \mu/2 \rceil$ copies of $m/2$ and post-append $\lfloor \mu/2 \rfloor$ zeroes.  We thus obtain a list $(\scoeff_1 \geq \scoeff_2 \geq \scoeff_3 \geq \cdots \geq \scoeff_{\ell-1} \geq \scoeff_\ell)$.  The element $d_{\vec{\ell}} = \Diag(\xi^{\scoeff_1},  \xi^{\scoeff_2}, \ldots , \xi^{\scoeff_{\ell -1}}, \xi^{\scoeff_{\ell}},  1, \xi^{-\scoeff_{\ell}}, \xi^{-\scoeff_{\ell-1}}, \ldots ,  \xi^{-\scoeff_2}, \xi^{-\scoeff_1}) $ in $\SO_{2 \ell +1}$ has characteristic polynomial $q_{n_{\vec{\ell}}}$  for its standard action on  $\C^{2\ell + 1}$. 
Since the linear factors $(t-(1)^{\mu})$ along with  $(t-\xi^{\scoeff_j})$ and $(t - \xi^{-\scoeff_j})$ for $1 \leq j \leq \ell$ must occur in $q_{n_{\vec{\ell}}}$,   we conclude that, up to the action of $\absW$, $d_{\vec{\ell}}$ is the unique element of $A$ that is $G$-conjugate to $n_{\vec{\ell}}$.

 We can now read off the Kac diagram for $n_{\vec{\ell}}$ from $d_{\vec{\ell}}$. 
  Note that $d_{\vec{\ell}} = \lambda_{\vec{\ell}} (\xi)$ where 
$$\lambda_{\vec{\ell}} = (\scoeff_1 - \scoeff_2)\check{\mu}_1 + (\scoeff_2 - \scoeff_3) \check{\mu}_2 + \cdots + (\scoeff_{\ell -1} - \scoeff_\ell) \check{\mu}_{\ell -1} + \scoeff_\ell \check{\mu}_\ell.$$
Since  $\dabs{\gamma}_\vartheta$ is $1$ for all $\gamma \in \Delta_\vartheta$ and 
$$\lambda_{\vec{\ell}}/m = 
\frac{1}{m} [
(\scoeff_1 - \scoeff_2)\check{\mu}_1 + (\scoeff_2 - \scoeff_3) \check{\mu}_2 + \cdots + (\scoeff_{\ell -1} - \scoeff_\ell) \check{\mu}_{\ell -1} + \scoeff_\ell \check{\mu}_\ell
],$$
from Equation~\ref{equ:equationxg} we conclude that the $s_{\gamma_\ell} = \scoeff_\ell$ and $s_{\gamma_k} = \scoeff_{k} - \scoeff_{k+1}$ for  $1 \leq k \leq \ell-1$.
From Equation~\ref{equ:equationform} we conclude that  $s_{\gamma_0} = m - (\scoeff_1 + \scoeff_2)$. Remove any factors that are common to all of the $s_\gamma$ for $\gamma \in \tilde{\Delta}$ and label the extended affine Dynkin diagram with the resulting $s_\gamma$.

 We have proved:
\begin{lemma}
Fix a partition $\vec{\ell} = (\ell_\mu, \ell_{\mu-1}, \ldots , \ell_2, \ell_1)$ of $\ell \geq 3$.  Let $m = 2\lcm(\ell_1, \ell_2, \ldots , \ell_\mu)$.  Append  $\lceil \mu/2 \rceil$  copies of $m/2$ and $\lfloor \mu/2 \rfloor$ copies of zero to the list
$( a_{\nu}m/2 \ell_\nu \, | \, \text{$1 \leq \nu \leq \mu$  and $ 1 \leq a_\nu \leq \ell_\nu -1$} )$
and then place the elements of the resulting list in decreasing order: 
$(\scoeff_1 \geq \scoeff_2 \geq \scoeff_3 \geq \cdots \geq \scoeff_{\ell-1} \geq \scoeff_\ell)$.
After removing any factors that are common to all of the labels, the Kac diagram for   $n_{\vec{\ell}}$ in a group of type $B_{\ell}$ with $\ell \geq 3$ is given by
\[
\begin{tikzpicture}
\begin{scope}[start chain]
\enode{\rotatebox[origin=c]{90}{$\sigma_2 - \sigma_3$}}
\enode{\rotatebox[origin=c]{90}{$\sigma_3 - \sigma_4$}}
\enode{\rotatebox[origin=c]{90}{$\sigma_4 - \sigma_5$}}
\dydots
\enode{\rotatebox[origin=c]{90}{$\sigma_{\ell-2} - \sigma_{\ell-1}$}}
\enode{\rotatebox[origin=c]{90}{$\sigma_{\ell-1} - \sigma_{\ell}$}}
\enodenj{\scoeff_\ell}
\path (chain-6) -- node{\(\Rightarrow\)} (chain-7);
\end{scope}
\begin{scope}[start chain=br going below]
\chainin(chain-1);
\node[nochj,label={below:\elabel{m - (\sigma_1 + \sigma_2)}}] at (-2,0) {};
\end{scope}
\begin{scope}[start chain=br going above]
\chainin(chain-1);
\node[nochj,label={above:\elabel{\sigma_1 - \sigma_2}}] at (-2,0) {};
\end{scope}
\end{tikzpicture}
\]
 \qed
\end{lemma}

 \begin{example}
 As an example, 
here is the diagram for $n_{(5,4,4,1)}$ in $\SO_{29}$.
\[
\begin{tikzpicture}
\begin{scope}[start chain]
\enode{0}
\enode{4}
\enode{1}
\enode{0}
\enode{3}
\enode{2}
\enode{0}
\enode{2}
\enode{3}
\enode{0}
\enode{1}
\enode{4}
\enode{0}
\enodenj{0}
\path (chain-13) -- node{\(\Rightarrow\)} (chain-14);
\end{scope}
\begin{scope}[start chain=br going below]
\chainin(chain-2);
\enodebr{0};
\end{scope}
\end{tikzpicture}
\]
\end{example}

The results here agree with those found in the table in~\cite[Section A.3]{reeder:thomae}. The partitions (in the notation of \cite[Section A.3]{reeder:thomae}) for the Kac diagrams appearing there are (from top to bottom): $(n)$, $(n/2,n/2)$ for $n$ even, $(n/k, n/k, n/k, \ldots , n/k, n/k)$ for $k >2$ even and $k$ divides $n$, and  $(n/k, n/k, n/k, \ldots , n/k, n/k)$ for $k>1$ odd and $k$ divides $n$.   The final two partitions in the previous sentence are identical and this correspondence between Kac diagrams and partitions agrees with~\cite[Table~12]{reederetal:gradings}.       
From Remark~\ref{rem:regellBl} these partitions correspond to the regular elliptic elements in a Weyl group of type $B_n$.

\subsection{How to create a Kac diagram for \texorpdfstring{$n_{\vec{\ell}}$ for $C_\ell$ with $\ell \geq 2$}{kac C}}  
We adopt the notation of Section~\ref{sec:injCl}.  Fix a partition $\vec{\ell}$ of $\ell$.   We will show how to construct the Kac diagram for   $n_{\vec{\ell}}$.

Since $\vartheta$ is trivial, we have $f = 1$.  We take the simple roots $\gamma$ in $\Delta_\vartheta = \Delta$ to be the roots $\gamma_k = \alpha_k$ for $1 \leq k \leq \ell$.   
For fundamental coweights with respect to our basis we take
$\check{\mu}_i = e_1 + e_2 + \cdots + e_i$ for $1 \leq i < \ell$ and $\check{\mu}_\ell = 1/2 (e_1 + e_2 + \cdots + e_\ell)$. 
The affine Dynkin diagram is 
\[
\begin{tikzpicture}[start chain]
\unodenj{0}
\unodenj{1}
\unode{2}
\dydots
\unode{\ell - 2}
\unode{\ell - 1}
\unodenj{\ell}
\path (chain-1) -- node{\(\Rightarrow\)} (chain-2);
\path (chain-6) -- node{\(\Leftarrow\)} (chain-7);
\end{tikzpicture}
\]
 We have $ b_{\gamma_k} = c_{\gamma_k}$ for $0 \leq k < \ell$, so the $b_\gamma$ are given by the diagram
\[
\begin{tikzpicture}[start chain]
\enodenj{1}
\enodenj{2}
\enode{2}
\dydots
\enode{2}
\enode{2}
\enodenj{1}
\path (chain-1) -- node{\(\Rightarrow\)} (chain-2);
\path (chain-6) -- node{\(\Leftarrow\)} (chain-7);
\end{tikzpicture}
\]

A Kac diagram describes a diagonal matrix in $\Sp_{2 \ell}$ that is $G$-conjugate to $n_{\vec{\ell}}$.  Thus, we want to find a diagonal matrix in $\Sp_{2 \ell}$ such that its characteristic polynomial for the standard action on  $\C^{2\ell}$ is  
$$q_{n_{\vec{\ell}}} (t)=  \prod_{\nu = 1}^\mu (t^{2 \ell_\nu} + 1).$$

Let $m = 4\lcm(\ell_1, \ell_2, \ldots , \ell_\mu)$; this is the order of $n_{\vec{\ell}}$.
Let $\xi$ be a primitive $m^{\text{th}}$ root of unity.  We have
$$q_{n_{\vec{\ell}}}(t) = 
\prod_{\nu = 1}^\mu 
\prod_{a_{\nu} = 1}^{\ell_\nu} 
(t - \xi^{(2 a_\nu -1) m / 4 \ell_\nu)} )(t - \xi^{-(2 a_\nu -1) m / 4 \ell_\nu)} ).$$
Guided by  Remark~\ref{rem:decreasing} we create a length $\ell$ decreasing list as follows: order the positive integers $(2 a_\nu -1) m / 4 \ell_\nu$ for $1 \leq \nu \leq \mu$ and $1 \leq \a_\nu \leq \ell_\nu$ in decreasing order.  We thus obtain a list $ S=(\scoeff_1 \geq \scoeff_2 \geq \scoeff_3 \geq \cdots \geq \scoeff_{\ell-1} \geq \scoeff_\ell)$.

The element $d_{\vec{\ell}} = \Diag(\xi^{\scoeff_1},  \xi^{\scoeff_2}, \ldots , \xi^{\scoeff_{\ell -1}}, \xi^{\scoeff_{\ell}},  \xi^{-\scoeff_{\ell}}, \xi^{-\scoeff_{\ell-1}}, \ldots ,  \xi^{-\scoeff_2}, \xi^{-\scoeff_1}) $ in $\Sp_{2 \ell}$ has characteristic polynomial $q_{n_{\vec{\ell}}}$  for its standard action on  $\C^{2\ell}$. 
Since the linear factors   $(t-\xi^{\scoeff_j})$ and $(t - \xi^{-\scoeff_j})$ for $1 \leq j \leq \ell$ must occur in $q_{n_{\vec{\ell}}}$,   we conclude that, up to the action of $\absW$, $d_{\vec{\ell}}$ is the unique element of $A$ that is $G$-conjugate to $n_{\vec{\ell}}$.

 We can now read off the Kac diagram for $n_{\vec{\ell}}$ from $d_{\vec{\ell}}$. 
  Note that $d_{\vec{\ell}} = \lambda_{\vec{\ell}} (\xi)$ where 
$$\lambda_{\vec{\ell}} = (\scoeff_{1} - \scoeff_{2}) \check{\mu}_1 + (\scoeff_{2} - \scoeff_{3}) \check{\mu}_2  + \ldots + (\scoeff_{\ell-1} - \scoeff_{\ell}) \check{\mu}_{\ell -1}  +2\scoeff_\ell \check{\mu}_{\ell}.$$
Since  $\dabs{\gamma}_\vartheta$ is $1$ for all $\gamma \in \Delta_\vartheta$ and 
$$\lambda_{\vec{\ell}}/m = 
\frac{1}{m} [ (\scoeff_{1} - \scoeff_{2}) \check{\mu}_1 + (\scoeff_{2} - \scoeff_{3}) \check{\mu}_2  + \ldots + (\scoeff_{\ell-1} - \scoeff_{\ell}) \check{\mu}_{\ell -1}  +2\scoeff_\ell \check{\mu}_{\ell}],$$
from Equation~\ref{equ:equationxg} we conclude that
 $s_{\gamma_\ell} = 2\scoeff_\ell$ and $s_{\gamma_k} = (\scoeff_{k} - \scoeff_{k+1})$ for  $1 \leq k \leq \ell-1$.
The coefficient $s_{\gamma_0} = m - 2 \scoeff_1$ is derived using Equation~\ref{equ:equationform}. 
Remove any factors that are common to all of the $s_\gamma$ for $\gamma \in \tilde{\Delta}$ and label the extended affine Dynkin diagram with the resulting $s_\gamma$.

 We have proved:
\begin{lemma}
Fix a partition $\vec{\ell} = (\ell_\mu, \ell_{\mu-1}, \ldots , \ell_2, \ell_1)$ of $\ell \geq 2$.  Let $m = 4\lcm(\ell_1, \ell_2, \ldots , \ell_\mu)$.  Take the list
$( (2a_{\nu}-1)m/4 \ell_\nu \, | \, \text{$1 \leq \nu \leq \mu$  and $ 1 \leq a_\nu \leq \ell_\nu$} )$
and place the elements in decreasing order: 
$(\scoeff_1 \geq \scoeff_2 \geq \scoeff_3 \geq \cdots \geq \scoeff_{\ell-1} \geq \scoeff_\ell)$.
After removing any factors that are common to all of the labels, the Kac diagram for   $n_{\vec{\ell}}$ in a group of type $C_{\ell}$ with $\ell \geq 2$ is given by
\[
\begin{tikzpicture}
\begin{scope}[start chain]
\enodenj{\rotatebox[origin=c]{90}{$m - 2\sigma_1$}}
\enodenj{\rotatebox[origin=c]{90}{$\sigma_1 - \sigma_2$}}
\enode{\rotatebox[origin=c]{90}{$\sigma_2 - \sigma_3$}}
\dydots
\enode{\rotatebox[origin=c]{90}{$\sigma_{\ell-2} - \sigma_{\ell-1}$}}
\enode{\rotatebox[origin=c]{90}{$\sigma_{\ell-1} - \sigma_{\ell}$}}
\enodenj{2\scoeff_\ell}
\path (chain-1) -- node{\(\Rightarrow\)} (chain-2);
\path (chain-6) -- node{\(\Leftarrow\)} (chain-7);
\end{scope}
\end{tikzpicture}
\]
 \qed
\end{lemma}

\begin{example}
As examples, here is the diagram for $n_{(6,5,2)}$ in $\Sp_{26}$
\[
\begin{tikzpicture}[start chain]
\enodenj{10}
\enodenj{1}
\enode{9}
\enode{0}
\enode{3}
\enode{7}
\enode{5}
\enode{5}
\enode{7}
\enode{3}
\enode{0}
\enode{9}
\enode{1}
\enodenj{10}
\path (chain-1) -- node{\(\Rightarrow\)} (chain-2);
\path (chain-14) -- node{\(\Leftarrow\)} (chain-13);
\end{tikzpicture}
\]
 here is the diagram for $n_{(2,1)}$ in $\Sp_{6}$
\[
\begin{tikzpicture}[start chain]
\enodenj{2}
\enodenj{1}
\enode{1}
\enodenj{2}
\path (chain-1) -- node{\(\Rightarrow\)} (chain-2);
\path (chain-4) -- node{\(\Leftarrow\)} (chain-3);
\end{tikzpicture}
\]
and here is the diagram for $n_{(3)}$ in $\Sp_{6}$
\[
\begin{tikzpicture}[start chain]
\enodenj{1}
\enodenj{1}
\enode{1}
\enodenj{1}
\path (chain-1) -- node{\(\Rightarrow\)} (chain-2);
\path (chain-4) -- node{\(\Leftarrow\)} (chain-3);
\end{tikzpicture}
\]
\end{example}

\begin{example}
In Figure~ \ref{fig:C2location} for groups of type $C_2$ we show the location\footnote{Note that the (weighted) barycenter for $C_2$ is different from that of $\lsup{2}A_3$ and $\lsup{2}D_3$ -- see Figures~\ref{fig:2A3location} and~\ref{fig:2D3location}.}  of the vertices $v_0$, $v_1$, and $v_2$ of the fundamental alcove as well as the points determined by $\lambda_{\vec{\ell}}/m$. The Kac diagram for $(2)$ is 
\(
\begin{tikzpicture}[start chain]
\enodenj{1}
\enodenj{1}
\enodenj{1}
\path (chain-1) -- node{\(\Rightarrow\)} (chain-2);
\path (chain-3) -- node{\(\Leftarrow\)} (chain-2);
\end{tikzpicture}
\)
and the Kac diagram for $(1,1)$ is 
\(
\begin{tikzpicture}[start chain]
\enodenj{1}
\enodenj{0}
\enodenj{1}
\path (chain-1) -- node{\(\Rightarrow\)} (chain-2);
\path (chain-3) -- node{\(\Leftarrow\)} (chain-2);
\end{tikzpicture}.
\)
\begin{figure}[ht]
\centering
\begin{tikzpicture}
\draw (0,0) 
  -- (5,0)
  -- (5,5)
  -- cycle;
  \draw[black, fill=white] (0,0) circle (.5ex);
\draw(-.3,-.4) node[anchor=south]{$v_0$};
 \draw[black,fill=white] (5,0) circle (.5ex);
\draw(5.3,-.4) node[anchor=south]{$v_1$};
 \draw[black,fill=white] (5,5) circle (.5ex);
\draw(5.3,4.9) node[anchor=south]{$v_2$};
  \draw[black,fill=black] (3.75,1.25) circle (.5ex);
 \draw(4.2,1.35) node[anchor=north]{$\frac{\lambda_{(2)}}{8}$};
  \draw[black,fill=black] (2.5,2.5) circle (.5ex);
\draw(3,2.6) node[anchor=north]{$\frac{\lambda_{(1,1)}}{4}$};
\end{tikzpicture}
\caption{The location of points determined by  $\lambda_{\vec{\ell}}$ for groups of type $C_2$ \label{fig:C2location}}
\end{figure}
\end{example}

During the publishing process, a typo was introduced in the table in~\cite[Section A.4]{reeder:thomae}.  For the case $k \mid n$ and $k> 1$  the Kac diagram should have a $1$ on the $\gamma_0$ node; that is, it should be 
\[
\begin{tikzpicture}[start chain]
\enodenj{1}
\enodenj{0}
\enode{0}
\enodenj{0}
\enode{1}
\enode{0}
\enode{0}
\enodenj{0}
\enode{1}
\enodenj{1}
\enode{0}
\enode{0}
\enodenj{0}
\enodenj{1}
\path (chain-1) -- node{\(\Rightarrow\)} (chain-2);
\path (chain-3) -- node{\(\cdots\)} (chain-4);
\path (chain-7) -- node{\(\cdots\)} (chain-8);
\path (chain-9) -- node{\(\cdots\)} (chain-10);
\path (chain-12) -- node{\(\cdots\)} (chain-13);
\path (chain-14) -- node{\(\Leftarrow\)} (chain-13);
\draw[decorate, decoration={  brace,  amplitude=10pt}] (1,.12)-- node[above=0.35cm]
{$(k-1)$ zeroes}(3,.12);
\draw[decorate, decoration={  brace,  amplitude=10pt}] (4.9,.12)-- node[above=0.35cm]
{$(k-1)$ zeroes}(6.9,.12);
\draw[decorate, decoration={  brace,  amplitude=10pt}] (9.8,.12)-- node[above=0.35cm]
{$(k-1)$ zeroes}(11.8,.12);
\node[above=0.35cm] at (8.3,.12) {$\cdots$};
\end{tikzpicture}
\]
where there are $n/k$ strings of $k-1$ zeroes.  With this change,  the partitions (in the notation of \cite[Section A.4]{reeder:thomae}) for the Kac diagrams appearing there are (from top to bottom): $(n)$ and  $(n/k, n/k, n/k, \ldots , n/k, n/k)$ with $k>1$ and dividing $n$.  
This correspondence between Kac diagrams and partitions agrees with~\cite[Table~13]{reederetal:gradings}.       
From Remark~\ref{rem:regellCl} these partitions correspond to the regular elliptic  elements in a Weyl group of type $C_n$.

\subsection{How to create a Kac diagram for \texorpdfstring{$n_{\vec{\ell}}$ for $D_\ell$ with $\ell \geq 4$}{kac D}}  
We adopt the notation of Section~\ref{sec:injDl}.  Fix a partition $\vec{\ell}$ of $\ell$ with an even number of parts.   We will show how to construct the Kac diagram for  $n_{\vec{\ell}}$.

Since $\vartheta$ is trivial, we have $f = 1$.  We take the simple roots $\gamma$ in $\Delta_\vartheta = \Delta$ to be the roots $\gamma_k = \alpha_k$ for $1 \leq k \leq \ell$.   
For fundamental coweights with respect to our basis we take
$\check{\mu}_i = e_1 + e_2 + \cdots + e_i$ for $1 \leq i \leq \ell-2$, $\check{\mu}_{\ell-1} = 1/2 (e_1 + e_2 + \cdots e_{\ell-2}  +  e_{\ell-1} - e_\ell)$, and $\check{\mu}_{\ell} = 1/2 (e_1 + e_2 + \cdots e_{\ell-2}  +  e_{\ell-1} + e_\ell)$. 
The affine Dynkin diagram is 
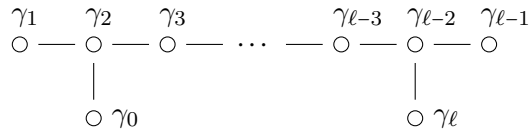
\[
\begin{tikzpicture}
\begin{scope}[start chain]
\unode{1}
\unode{2}
\unode{3}
\dydots
\unode{\ell - 3}
\unode{\ell - 2}
\unode{\ell -1}
\end{scope}
\begin{scope}[start chain=br going below]
\chainin(chain-2);
\dnodebr{0};
\end{scope}
\begin{scope}[start chain=br2 going below]
\chainin(chain-6);
\dnodebr{\ell};
\end{scope}
\end{tikzpicture}
\]
 We have $ b_{\gamma_k} = c_{\gamma_k}$ for $0 \leq k < \ell$, so the $b_\gamma$ are given by the diagram
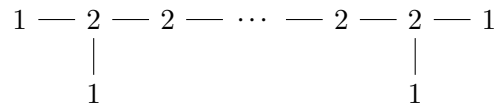
\[
\begin{tikzpicture}
\begin{scope}[start chain]
\enode{1}
\enode{2}
\enode{2}
\dydots
\enode{2}
\enode{2}
\enode{1}
\end{scope}
\begin{scope}[start chain=br going below]
\chainin(chain-2);
\enodebr{1};
\end{scope}
\begin{scope}[start chain=br2 going below]
\chainin(chain-6);
\enodebr{1};
\end{scope}
\end{tikzpicture}
\]

A Kac diagram describes a diagonal matrix in $\SO_{2 \ell}$ that is $G$-conjugate to $n_{\vec{\ell}}$.  Thus, we want to find a diagonal matrix in $\SO_{2 \ell}$ such that its characteristic polynomial for the standard action on  $\C^{2\ell}$ is 
$$q_{n_{\vec{\ell}}} (t)=   \prod_{\nu = 1}^\mu (t^{2 \ell_\nu} - 1).$$ 

Let $m = 2\lcm(\ell_1, \ell_2, \ldots , \ell_\mu)$ and let $\xi$ be a primitive $m^{\text{th}}$ root of unity. We have
$$q_{n_{\vec{\ell}}} (t)= \prod_{\nu = 1}^\mu (t^{2 \ell_\nu} - 1) = \prod_{\nu = 1}^\mu (t-1)(t+1) \prod_{a_{\nu} = 1}^{\ell_\nu - 1} (t - \xi^{ma_{\nu}/2 \ell_\nu} )(t - \xi^{-ma_{\nu}/2 \ell_\nu} ).$$
Guided by  Remark~\ref{rem:decreasing} we create a length $\ell$ decreasing list as follows: order the positive integers $ma_{\nu}/2 \ell_\nu$ in decreasing order, then pre-append $\mu/2$ copies of $m/2$ and post-append $\mu/2$ zeroes.  We thus obtain a list $(\scoeff_1 \geq \scoeff_2 \geq \scoeff_3 \geq \cdots \geq \scoeff_{\ell-1} \geq \scoeff_\ell)$.  The element $d_{\vec{\ell}} = \Diag(\xi^{\scoeff_1},  \xi^{\scoeff_2}, \ldots , \xi^{\scoeff_{\ell -1}}, \xi^{\scoeff_{\ell}},  \xi^{-\scoeff_{\ell}}, \xi^{-\scoeff_{\ell-1}}, \ldots ,  \xi^{-\scoeff_2}, \xi^{-\scoeff_1}) $ in $\SO_{2 \ell}$ has characteristic polynomial $q_{n_{\vec{\ell}}}$  for its standard action on  $\C^{2\ell}$. 

 Since the linear factors  $(t-\xi^{\scoeff_j})$ and $(t - \xi^{-\scoeff_j})$ for $1 \leq j \leq \ell$ must occur in $q_{n_{\vec{\ell}}}$,   we conclude that, up to the action of $\absW$, $d_{\vec{\ell}}$ is the unique element of $A$ that is $G$-conjugate to $n_{\vec{\ell}}$.

We can now read off the Kac diagram for $n_{\vec{\ell}}$ from $d_{\vec{\ell}}$. 
  Note that $d_{\vec{\ell}} = \lambda_{\vec{\ell}} (\xi)$ where 
$$\lambda_{\vec{\ell}} =  ({\scoeff_{1} - \scoeff_{2}}) \check{\mu}_1 + ({\scoeff_{2} - \scoeff_{3}}) \check{\mu}_2  + \ldots + ({\scoeff_{\ell-2} - \scoeff_{\ell-1}}) \check{\mu}_{\ell -2}  + (\scoeff_{\ell-1} - \scoeff_\ell) \check{\mu}_{\ell-1} + (\scoeff_{\ell-1} + \scoeff_\ell) \check{\mu}_{\ell}.$$
Since  $\dabs{\gamma}_\vartheta$ is $1$ for all $\gamma \in \Delta_\vartheta$ and 
$$\lambda_{\vec{\ell}}/m = 
\frac{1}{m} [
({\scoeff_{1} - \scoeff_{2}}) \check{\mu}_1 + ({\scoeff_{2} - \scoeff_{3}}) \check{\mu}_2  + \ldots + ({\scoeff_{\ell-2} - \scoeff_{\ell-1}}) \check{\mu}_{\ell -2}  + (\scoeff_{\ell-1} - \scoeff_\ell) \check{\mu}_{\ell-1} + (\scoeff_{\ell-1} + \scoeff_\ell) \check{\mu}_{\ell}],$$
from Equation~\ref{equ:equationxg} we conclude that
$s_{\gamma_\ell} = \scoeff_{\ell-1} + \scoeff_\ell$ and
 $s_{\gamma_k} = (\scoeff_{k} - \scoeff_{k+1})$ for  $1 \leq k \leq \ell-1$.
The coefficient $s_{\gamma_0} = m - (\scoeff_1+\scoeff_2)$ is derived using Equation~\ref{equ:equationform}. 
Remove any factors that are common to all of the $s_\gamma$ for $\gamma \in \tilde{\Delta}$ and label the extended affine Dynkin diagram with the resulting $s_\gamma$.

We have proved:
\begin{lemma}
Fix a partition $\vec{\ell} = (\ell_\mu, \ell_{\mu-1}, \ldots , \ell_2, \ell_1)$ of $\ell \geq 4$ with $\mu$ even.  Let $m = 2\lcm(\ell_1, \ell_2, \ldots , \ell_\mu)$.  Append  $\mu/2$  copies of $m/2$ and $\mu/2$ copies of zero to the list
$( a_{\nu}m/2 \ell_\nu \, | \, \text{$1 \leq \nu \leq \mu$  and $ 1 \leq a_\nu \leq \ell_\nu -1$} )$
and then place the elements of the resulting list in decreasing order: 
$(\scoeff_1 \geq \scoeff_2 \geq \scoeff_3 \geq \cdots \geq \scoeff_{\ell-1} \geq \scoeff_\ell)$.
After removing any factors that are common to all of the labels, the Kac diagram for   $n_{\vec{\ell}}$ in a group of type $D_{\ell}$ with $\ell \geq 4$ is given by
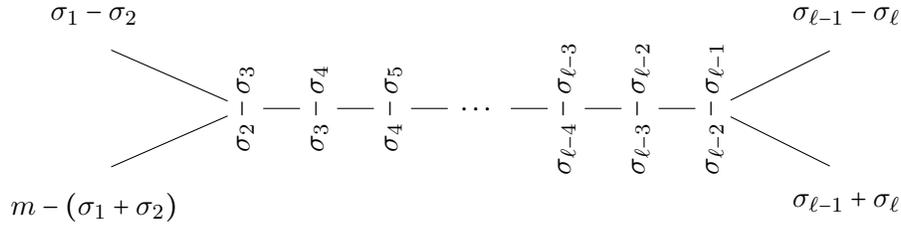
\[
\begin{tikzpicture}
\begin{scope}[start chain]
\enode{\rotatebox[origin=c]{90}{$\sigma_2 - \sigma_3$}}
\enode{\rotatebox[origin=c]{90}{$\sigma_3 - \sigma_4$}}
\enode{\rotatebox[origin=c]{90}{$\sigma_4 - \sigma_5$}}
\dydots
\enode{\rotatebox[origin=c]{90}{$\sigma_{\ell-4} - \sigma_{\ell-3}$}}
\enode{\rotatebox[origin=c]{90}{$\sigma_{\ell-3} - \sigma_{\ell-2}$}}
\enode{\rotatebox[origin=c]{90}{$\sigma_{\ell-2} - \sigma_{\ell-1}$}}
\end{scope}
\begin{scope}[start chain=br going below]
\chainin(chain-1);
\node[nochj,label={below:\elabel{m - (\sigma_1 + \sigma_2)}}] at (-2,0) {};
\end{scope}
\begin{scope}[start chain=br going above]
\chainin(chain-1);
\node[nochj,label={above:\elabel{\sigma_1 - \sigma_2}}] at (-2,0) {};
\end{scope}
\begin{scope}[start chain=br going below]
\chainin(chain-7);
\node[nochj,label={below:\elabel{\sigma_{\ell-1} + \sigma_\ell}}] at (8,0) {};
\end{scope}
\begin{scope}[start chain=br going above]
\chainin(chain-7);
\node[nochj,label={above:\elabel{\sigma_{\ell-1} - \sigma_\ell}}] at (8,0) {};
\end{scope}
\end{tikzpicture}
\]
 \qed
\end{lemma}

\begin{example}
As an example, here is the diagram for $n_{(5,4,4,1)}$ in $\SO_{28}$.
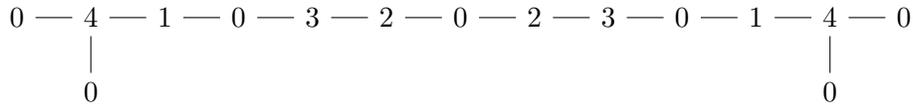
\[
\begin{tikzpicture}
\begin{scope}[start chain]
\enode{0}
\enode{4}
\enode{1}
\enode{0}
\enode{3}
\enode{2}
\enode{0}
\enode{2}
\enode{3}
\enode{0}
\enode{1}
\enode{4}
\enode{0}
\end{scope}
\begin{scope}[start chain=br going below]
\chainin(chain-2);
\enodebr{0};
\end{scope}
\begin{scope}[start chain=br2 going below]
\chainin(chain-12);
\enodebr{0};
\end{scope}
\end{tikzpicture}
\]
\end{example}

 Many Kac diagrams for type $D_\ell$ appear in the literature.  The results here agree with those found in~\cite{Bouwknegt:Lie}.  For $0 \leq k \leq \ell/2$ the Kac diagram of the conjugacy class of $n_{(\ell - k,k)}$, which corresponds to the label $D_{\ell}(a_{k-1})$ in the notation of~\cite{carter:weyl}, may be extracted from~\cite[Appendix B]{schellekens-warner:weylII}, and our results agree with the results found there.  
 They also agree with those found in the table in~\cite[Section A.5]{reeder:thomae}; the partitions (in the notation of \cite[Section A.5]{reeder:thomae}) for the Kac diagrams appearing there are (from top to bottom): $(n-1,1)$, $(n/2,n/2)$ for $n$ even, $(1,1,1, \ldots , 1, 1)$ for $n$ even, $(n/k, n/k, n/k, \ldots , n/k, n/k)$ for $2 < k< n$ even and dividing $n$, and $((n-1)/k, (n-1)/k, (n-1)/k, \ldots , (n-1)/k, (n-1)/k, 1)$ for $1<k<n-1$ odd and $k$ dividing $(n-1)$. This correspondence between Kac diagrams and partitions agrees with~\cite[Table~14]{reederetal:gradings}.     
From Remark~\ref{rem:regellDl} these partitions correspond to the regular elliptic  elements in a Weyl group of type $D_n$.

\subsection{How to create Kac diagrams for \texorpdfstring{$\lsup{2}A_{\ell-1}$ with $\ell \geq 3$}{kac 2A}}
 We adopt the notation of Section~\ref{sec:inj2Alminusone}. 
  Fix a partition $\vec{\ell}$ of $\ell$ for which all parts are odd.   We will show how to construct the Kac diagram for  $n_{\vec{\ell}} \rtimes \vartheta$.

Note that $\vartheta$ acts on $\R^\ell$ by $\vartheta((x_1, x_2, \ldots , x_\ell)) = (-x_\ell, -x_{\ell -1}, \ldots , -x_2, -x_1)$.  Thus, the fixed points in $V$ are 
$$(x_1, x_2, x_3 , \ldots  , -x_3, -x_2, -x_1).$$
If $\ell$ is odd, then the $((\ell+1)/2)^{\text{th}}$ coordinate must be zero.  In either case, we can restrict our attention to the first $\lfloor \ell /2 \rfloor$ coordinates and identify $V^{\vartheta}$ with $\R^{\lfloor \ell /2 \rfloor}$.

\subsubsection{How to create Kac diagrams for \texorpdfstring{$\lsup{2}A_{\ell-1}$ with $\ell = 2n \geq 4$}{kac 2A even}}
Suppose $\ell = 2n$ is even.  

We have $f = 2$ and $V^\vartheta  = \R^n$. We can take the simple roots $\gamma$ in $\Delta_\vartheta$ to be the roots $\gamma_k = \res_{V^\vartheta} 2\alpha_k = 2(x_k - x_{k+1})$ for $1 \leq k < n$ and $\gamma_n = \res_{V^\vartheta} \alpha_n = 
2x_n$.   
The fundamental coweights with respect to our basis are
$\check{\mu}_i = \frac{1}{2}(e_1 + e_2 + \cdots + e_i)$ for $1 \leq i \leq n$. 
 For $n\geq 3$, the affine Dynkin diagram is 
\[
\begin{tikzpicture}
\begin{scope}[start chain]
\unode{1}
\unode{2}
\unode{3}
\dydots
\unode{n - 2}
\unode{n - 1}
\unodenj{n}
\path (chain-6) -- node{\(\Leftarrow\)} (chain-7);
\end{scope}
\begin{scope}[start chain=br going below]
\chainin(chain-2);
\dnodebr{0};
\end{scope}
\end{tikzpicture}
\]
 and the $b_\gamma$ are given by the diagram
\[
\begin{tikzpicture}
\begin{scope}[start chain]
\enode{1}
\enode{2}
\enode{2}
\enode{2}
\dydots
\enode{2}
\enode{2}
\enodenj{1}
\path (chain-7) -- node{\(\Leftarrow\)} (chain-8);
\end{scope}
\begin{scope}[start chain=br going below]
\chainin(chain-2);
\enodebr{1};
\end{scope}
\end{tikzpicture}
\]
If $n = 2$, then  $A_3 \cong D_3$ with $\vartheta$ acting by an automorphism of order two.  In this case the affine Dynkin diagram is 
\[
\begin{tikzpicture}[start chain]
\unode{0}
\unodenj{2}
\unodenj{1}
\path (chain-1) -- node[anchor=mid] {\(\Leftarrow\)} (chain-2);
\path (chain-2) -- node[anchor=mid] {\(\Rightarrow\)} (chain-3);
\end{tikzpicture}
\]
and the $b_\gamma$ are given by the diagram
\[
\begin{tikzpicture}[start chain]
\enode{1}
\enodenj{1}
\enodenj{1}
\path (chain-1) -- node[anchor=mid] {\(\Leftarrow\)} (chain-2);
\path (chain-2) -- node[anchor=mid] {\(\Rightarrow\)} (chain-3);
\end{tikzpicture}
\]

Since each $\ell_\nu$ is odd and $\ell = 2n$ is even, we must have that $\mu$ is even.  Let $m = 2 \lcm(\ell_1, \ell_2, \ldots , \ell_{\mu})$, and let $\xi$ be a primitive $m^{\text{th}}$ root of unity.  Create a length $n$ list of positive odd integers  by appending  $\mu/2$ copies of $m/2$ to the list $( (2a_\nu -1)m/2 \ell_\nu \, | \, 1 \leq \nu \leq \mu, \, 1 \leq a_\nu < \ell_\nu/2 )$.  Place the items on this length $n$ list into decreasing order: $(\scoeff_1 \geq \scoeff_2 \geq \cdots \geq \scoeff_n)$.  We fix a square root $\xi^{1/2}$ of $\xi$ and consider the diagonal matrix 
$$d_{\vec{\ell}} = \Diag( \xi^{\scoeff_1/2},   \xi^{\scoeff_2/2}, \ldots ,  \xi^{\scoeff_n/2},   \xi^{-\scoeff_n/2}, \xi^{-\scoeff_{n -1}/2},  \ldots , \xi^{-\scoeff_{1}/2}).$$
Note that $d_{\vec{\ell}}$ belongs to $A^\vartheta$.

In Lemma~\ref{lem:oddcase} we prove that  $d_{\vec{\ell}}$ is $\vartheta$-conjugate to $n_{\vec{\ell}}$.  Assuming this, 
we now read off the Kac diagram for $n_{\vec{\ell}}$ from $d_{\vec{\ell}}$. 
 Note that $d_{\vec{\ell}} = \lambda_{\vec{\ell}}(\xi)$ where 
$$\lambda_{\vec{\ell}} =  (\scoeff_1 - \scoeff_2) \check{\mu}_1 + ( \scoeff_2 - \scoeff_3 ) \check{\mu}_2  +  \ldots  + ( \scoeff_{n-1} - \scoeff_{n} ) \check{\mu}_{n-1} + \scoeff_n \check{\mu}_{n}.$$
Since $\dabs{\gamma_n}_\vartheta =1 $, $\dabs{\gamma_k}_\vartheta = 2$ for $1 \leq k < n$, and 
$$\lambda_{\vec{\ell}}/m = 
\frac{1}{m} [
2 \cdot ({\scoeff_{1} - \scoeff_{2}})/2 \cdot \check{\mu}_1 + 2 \cdot ({\scoeff_{2} - \scoeff_{3}})/2 \cdot \check{\mu}_2  + \ldots + 2 \cdot ({\scoeff_{n-1} - \scoeff_{n}}) /2  \cdot \check{\mu}_{n -1} +    \scoeff_n \cdot \check{\mu}_{n}],$$
from Equation~\ref{equ:equationxg} we conclude that
$s_{\gamma_n} = \scoeff_{n}$ and $s_{\gamma_k} = (\scoeff_k-\scoeff_{k+1})/2$ for  $1 \leq k \leq n-1$.

 Using Equation~\ref{equ:equationform} we have that 
the coefficient $s_{\gamma_0}$ is $(m-(\sigma_1 +\sigma_2))/2$.
Remove any factors that are common to all of the $s_\gamma$ for $\gamma \in \tilde{\Delta}$ and label the extended affine Dynkin diagram with the resulting $s_\gamma$.

We have proved:
\begin{lemma}
Fix a partition $\vec{\ell} = (\ell_\mu, \ell_{\mu-1}, \ldots , \ell_2, \ell_1)$ of $\ell = 2n$ with $\ell_\nu$ odd for $1 \leq \nu \leq \mu$ and $n \geq 2$.  Let $m = 2\lcm(\ell_1, \ell_2, \ldots , \ell_\mu)$.  Append  $\mu/2$  copies of $m/2$ to the list
$( (2a_{\nu}-1)m/2 \ell_\nu \, | \, \text{$1 \leq \nu \leq \mu$  and $ 1 \leq a_\nu < \ell_\nu/2$} )$
and then place the elements of the resulting list in decreasing order: 
$(\scoeff_1 \geq \scoeff_2 \geq \scoeff_3 \geq \cdots \geq \scoeff_{n-1} \geq \scoeff_n)$.
After removing any factors that are common to all of the labels, the Kac diagram for   $n_{\vec{\ell}} \rtimes \vartheta$ in a group of type $\lsup{2}A_{2n-1} = \lsup{2}A_{\ell - 1}$ is given by
\[
\begin{tikzpicture}
\begin{scope}[start chain]
\enode{\rotatebox[origin=c]{90}{$(\sigma_2 - \sigma_3)/2$}}
\enode{\rotatebox[origin=c]{90}{$(\sigma_3 - \sigma_4)/2$}}
\enode{\rotatebox[origin=c]{90}{$(\sigma_4 - \sigma_5)/2$}}
\dydots
\enode{\rotatebox[origin=c]{90}{$(\sigma_{n-2} - \sigma_{n-1})/2$}}
\enode{\rotatebox[origin=c]{90}{$(\sigma_{n-1} - \sigma_n)/2$}}
\enodenj{\sigma_n}
\path (chain-6) -- node{\(\Leftarrow\)} (chain-7);
\end{scope}
\begin{scope}[start chain=br going below]
\chainin(chain-1);
\node[nochj,label={below:\elabel{(m - (\sigma_1 + \sigma_2))/2}}] at (-2,0) {};
\end{scope}
\begin{scope}[start chain=br going above]
\chainin(chain-1);
\node[nochj,label={above:\elabel{(\sigma_1 - \sigma_2)/2}}] at (-2,0) {};
\end{scope}
\end{tikzpicture}
\]
when $n>2$ and  by
\[
\begin{tikzpicture}[start chain]
\enode{\rotatebox[origin=c]{90}{$(m - (\sigma_1+\sigma_2))/2$}}
\enodenj{\sigma_2}
\enodenj{\rotatebox[origin=c]{90}{$(\sigma_1 - \sigma_2)/2$}}

\path (chain-1) -- node[anchor=mid] {\(\Leftarrow\)} (chain-2);
\path (chain-2) -- node[anchor=mid] {\(\Rightarrow\)} (chain-3);
\end{tikzpicture}.
\]
when $n = 2$. \qed
\end{lemma}

\begin{example}
As an example, here is the Kac diagram for $n_{(7,5,5,3)}$ in  (twisted) $\SL_{20}$.
\[
\begin{tikzpicture}
\begin{scope}[start chain]
\enode{0}
\enode{15}
\enode{6}
\enode{0}
\enode{9}
\enode{5}
\enode{7}
\enode{0}
\enode{3}
\enodenj{15}
\path (chain-9) -- node{\(\Leftarrow\)} (chain-10);
\end{scope}
\begin{scope}[start chain=br going below]
\chainin(chain-2);
\enodebr{0};
\end{scope}
\end{tikzpicture}.
\]
\end{example}

\begin{example}
In Figure~ \ref{fig:2A3location} for groups of type $\lsup{2}A_3$ we show the location of the vertices $v_0$, $v_1$, and $v_2$ of the fundamental alcove as well as the points determined by $\lambda_{\vec{\ell}}/m$. The Kac diagram for $(3,1)$ is 
\(
\begin{tikzpicture}[start chain]
\enodenj{1}
\enodenj{1}
\enodenj{1}
\path (chain-1) -- node{\(\Leftarrow\)} (chain-2);
\path (chain-3) -- node{\(\Rightarrow\)} (chain-2);
\end{tikzpicture}
\)
and the Kac diagram for $(1,1,1,1)$ is 
\(
\begin{tikzpicture}[start chain]
\enodenj{0}
\enodenj{1}
\enodenj{0}
\path (chain-1) -- node{\(\Leftarrow\)} (chain-2);
\path (chain-3) -- node{\(\Rightarrow\)} (chain-2);
\end{tikzpicture}.
\)
\begin{figure}[ht]
\centering
\begin{tikzpicture}
\draw (0,0) 
  -- (5,0)
  -- (5,5)
  -- cycle;
  \draw[black, fill=white] (0,0) circle (.5ex);
\draw(-.3,-.4) node[anchor=south]{$v_0$};
 \draw[black,fill=white] (5,5) circle (.5ex);
\draw(5.3,4.9) node[anchor=south]{$v_1$};
\draw[black,fill=black] (3.33,1.66) circle (.5ex);
\draw(3.83,1.66) node[anchor=north]{$\frac{\lambda_{(3,1)}}{6}$};
  \draw[black,fill=black] (5,0) circle (.5ex);
\draw(6.15,.38) node[anchor=north]{$v_2 = \frac{\lambda_{(1,1,1,1)}}{2}$};
\end{tikzpicture}
\caption{The location of points determined by  $\lambda_{\vec{\ell}}$ for groups of type $\lsup{2}A_3$ \label{fig:2A3location}}
\end{figure}
\end{example}

 The results of this section agree with those found in the table in~\cite[Section A.2]{reeder:thomae}; the partitions (in the notation of \cite[Section A.2]{reeder:thomae}) for the Kac diagrams appearing there are (from top to bottom): $(2n-1,1)$, $(1,1,1, \ldots , 1,1)$, $(n,n,1)$ for odd $n$, $(d,d,d,   \ldots , d, d,1)$ with $d$ odd appearing $2k-1$ times where  $d \cdot (2k-1) + 1 =2n$, and   $(d,d,d,d, \ldots , d, d )$  with $d$ odd and appearing $2k$ times where $kd = n$. This correspondence between Kac diagrams and partitions agrees with~\cite[Table~11]{reederetal:gradings}.        From Remark~\ref{rem:regell2Al} these partitions correspond to the regular $\vartheta$-elliptic  elements in a twisted Weyl group of type $\lsup{2}A_{2n-1}$.

\subsubsection{A proof of \texorpdfstring{Lemma~\ref{lem:oddcase}}{lemma odd}}  \label{sec:oddcase}

Recall that $\ell = 2n$, so we are looking at $\lsup{2}A_{2n-1}$.

The proof involves a bit of notation.  For $1 \leq i,k \leq \ell$ let $E(i,j)$ denote the $\ell \times \ell$ matrix in $\gllie_{\ell}$ such that $E(i,j)_{rc} = \delta_{ir} \delta_{jc}$.
If either $i \leq n$ and $j > n$ or $i > n$ and $j \leq n$, then
$\vartheta(E_{ij}) =  E_{\ell-j+1, \ell-i +1}$.  Otherwise, we have $\vartheta(E_{ij}) = - E_{\ell-j+1, \ell-i +1}$.

\begin{lemma} \label{lem:oddcasechar}
The characteristic polynomial of 
$\Ad(d_{\vec{\ell}}) \circ \vartheta$ acting on $\sllie_\ell$ is 
$$   q_{n_{\vec{\ell}}}(t) =  \frac{1}{t+1} \cdot
 \left( \prod_{\nu=1}^\mu (t^{\ell_\nu} + 1) \right) 
 \cdot  \left(    \prod_{\tau=1}^{\mu} (t^{2\ell_\tau} - 1)^{ (\ell_\tau-1)/2} \right) 
 \cdot \left(    \prod_{1 \leq \rho < \varphi \leq \mu} (t^{2 \lcm(\ell_\rho,\ell_\varphi)} - 1)^{ \gcd(\ell_\rho,\ell_\varphi)} \right). $$
\end{lemma}

\begin{proof}
We have that $\ell = 2n$,  $m = 2 \lcm (\ell_1, \ell_2, \ldots, \ell_\mu)$, and $\xi$ is a primitive $m^{\text{th}}$ root of unity. We first define a diagonal element $d'_{\vec{\ell}}$ that has the same entries as $d_{\vec{\ell}}$, but is more amenable to computation.

\begin{equation*}
\begin{split}
d'_{\vec{\ell}} =\Diag(\xi&^{(\ell_\mu -2)(m/4 \ell_\mu)} ,   \xi^{(\ell_\mu -4)(m/4 \ell_\mu)} , \ldots, \xi^{3(m/4 \ell_\mu)}, \xi^{m/4 \ell_\mu},\\
&\xi^{(\ell_{\mu-1} -2)(m/4 \ell_{\mu-1})} ,   \xi^{(\ell_{\mu-1} -4)(m/4 \ell_{\mu-1})} , \ldots, \xi^{3(m/4 \ell_{\mu-1})}, \xi^{m/4 \ell_{\mu-1}},\\
& \ldots ,\\
&\xi^{(\ell_1 -2)(m/4 \ell_1)} , \xi^{(\ell_1 -4)(m/4 \ell_1)}, \ldots , \xi^{3(m/4 \ell_1)}, \xi^{m/4 \ell_1}, \\
& \xi^{m/4}, \cdots , \xi^{m/4}, \xi^{-m/4}, \ldots, \xi^{-m/4}, \\
& \xi^{-m/4 \ell_1}, \xi^{-3(m/4 \ell_1)}, \ldots , \xi^{-(\ell_1 -4)(m/4 \ell_1)} , \xi^{-(\ell_1 -2)(m/4 \ell_1)}, \\
& \ldots ,\\
&\xi^{-m/4 \ell_{\mu-1} }, \xi^{-3(m/4 \ell_{\mu-1})}, \ldots, \xi^{-(\ell_{\mu-1} -4)(m/4 \ell_{\mu-1})} , \xi^{-(\ell_{\mu-1} -2)(m/4 \ell_{\mu-1})}, \\
&\xi^{-m/4 \ell_{\mu} }, \xi^{-3(m/4 \ell_{\mu})}, \ldots, \xi^{-(\ell_{\mu} -4)(m/4 \ell_\mu)} , \xi^{-(\ell_{\mu} -2)(m/4 \ell_{\mu})}).
\end{split}
\end{equation*}
In the middle of the matrix there are $\mu/2$ copies of $\xi^{m/4}$ followed by $\mu/2$ copies of $\xi^{-m/4}$.
Since $d'_{\vec{\ell}}$ is $W^\vartheta$-conjugate to $d_{\vec{\ell}}$, it is enough to compute the characteristic polynomial of $\Ad(d'_{\vec{\ell}}) \circ \vartheta$ acting on $\sllie_\ell$.  Thus, it is enough to show that the characteristic polynomial of $\Ad(d'_{\vec{\ell}}) \circ \vartheta$ acting on $\gllie_\ell$ is
$$  (t+1) \cdot q_{n_{\vec{\ell}}}(t) =  
 \left( \prod_{\nu=1}^\mu (t^{\ell_\nu} + 1) \right) 
 \cdot  \left(    \prod_{\tau=1}^{\mu} (t^{2\ell_\tau} - 1)^{ (\ell_\tau-1)/2} \right) 
 \cdot \left(    \prod_{1 \leq \rho < \varphi \leq \mu} (t^{2 \lcm(\ell_\rho,\ell_\varphi)} - 1)^{ \gcd(\ell_\rho,\ell_\varphi)} \right). $$

For $1 \leq i \leq \ell$, define $d'_i$ by $(d'_{\vec{\ell}})_{ii} = \xi^{d'_i}$.    That is, the power of $\xi$ appearing in the $i^{\text{th}}$ diagonal element of $d'_{\vec{\ell}}$ is $d'_i$.   Note that $d_i' = - d_{\ell - i + 1}'$. 

For $1 \leq i \leq \ell$  we have $(\Ad(d'_{\vec{\ell}}) \circ \vartheta ) E(i, \ell-i+1) =  \xi^{2d'_i} E(i, \ell-i+1)$.  Thus, the vector space $U_a$ of anti-diagonal elements in $\gllie_\ell$ is $(\Ad(d'_{\vec{\ell}}) \circ \vartheta)$-stable and the characteristic polynomial for the action of $\Ad(d'_{\vec{\ell}}) \circ \vartheta$ on $U_a$ is 
$$\prod_{i=1}^{\ell} (t - \xi^{2d'_i}) =  (t - \xi^{m/2})^{\mu}  \cdot  \prod_{\nu = 1}^{\mu} \prod_{i=1}^{(\ell_\nu-1)/2}(t - \xi^{(2i-1)(m/2 \ell_\nu)}) (t - \xi^{-(2i-1)(m/2 \ell_\nu)})  =   \prod_{\nu = 1}^{\mu} (t^{\ell_\nu} + 1).$$

For $1 \leq i \leq n- \mu/2$  we have $(\Ad(d'_{\vec{\ell}}) \circ \vartheta) E(i,i) = -E(\ell -i + 1, \ell-i+1)$ and $(\Ad(d'_{\vec{\ell}}) \circ \vartheta) E(\ell -i + 1, \ell-i+1) = -E(i,i)$.  Thus, the  vector space $U_d$  spanned by
$$(E(i,i), E(\ell -i + 1, \ell-i+1) \, | \, 1 \leq i \leq n - \mu/2)$$
in $\gllie_\ell$ is  $(\Ad(d'_{\vec{\ell}}) \circ \vartheta)$-stable, and the characteristic polynomial for the action of $\Ad(d'_{\vec{\ell}}) \circ \vartheta$ on $U_d$ is $(t^2-1)^{n-\mu/2}$.

For pairs $(i,j)$ where $1 \leq i \leq n$ and $i < j \leq \ell - i $ we have 
$$(\Ad(d'_{\vec{\ell}}) \circ \vartheta ) E(i, j) = \pm \xi^{d'_i - d'_j} E(\ell - j + 1, \ell-i+1)$$
and 
$$(\Ad(d'_{\vec{\ell}}) \circ \vartheta ) E(\ell - j + 1, \ell-i+1) = \pm \xi^{d'_i - d'_j} E(i,j)$$
where the sign is positive if $j>n$ and is negative otherwise.
Similarly, for pairs $(i,j)$ where $1 \leq j \leq n$ and $j < i \leq \ell - j $ we have 
$$(\Ad(d'_{\vec{\ell}}) \circ \vartheta ) E(i, j) = \pm \xi^{d'_i - d'_j} E(\ell - j + 1, \ell-i+1)$$
and 
$$(\Ad(d'_{\vec{\ell}}) \circ \vartheta ) E(\ell - j + 1, \ell-i+1) = \pm \xi^{d'_i - d'_j} E(i,j)$$
where the sign is positive if $i>n$ and is negative otherwise.  Let  $\mathcal{P}$ denote the set of pairs $(i,j)$ such that (exactly) one of the following is true:
\begin{itemize}
\item $1 \leq i  \leq n$ and $i < j \leq \ell - i $, 
\item $1 \leq j  \leq n$ and $j < i \leq \ell - j$, or
\item $ i = j$ and $n-\mu/2 < i \leq n$.
\end{itemize}
Then, $U$, the span of the elements 
$$(   E(i, j), E(\ell - j + 1, \ell-i+1)  \, | \, \text{ $(i,j) \in \mathcal{P}$ } ),$$
is $(\Ad(d'_{\vec{\ell}}) \circ \vartheta)$-stable since the span of any pair $( E(i, j), E(\ell - j + 1, \ell-i+1))$ occurring in the definition of $U$ is  $(\Ad(d'_{\vec{\ell}}) \circ \vartheta)$-stable.  Moreover, $U$ is a vector space complement in $\gllie_\ell$ to the vector subspace $U_a \oplus U_d$ of  $\gllie_\ell$. 
We are going to break $U$ into $\mu(\mu+1)/2$ smaller $(\Ad(d'_{\vec{\ell}}) \circ \vartheta)$-stable subspaces, and evaluate the characteristic polynomial of 
the action of $\Ad(d'_{\vec{\ell}}) \circ \vartheta$ on each of these smaller spaces.

For $0 \leq \nu \leq \mu$, define $\tilde{\ell}_\nu = (\ell_\nu -1)/2$ and $\ell''_\nu = \sum_{j=\nu+1}^\mu \tilde{\ell}_j$.  Note that $\ell''_0 = n - \mu/2$ and $\ell''_\mu = 0$.  For $1 \leq \nu \leq \mu$, define the length $\ell_\nu$ lists 
$$S^r_\nu = (\ell''_{\nu} +1, \ell''_{\nu} +2, \ldots ,  \ell''_{\nu} + \tilde{\ell}_{\nu},  n + \mu/2 - \nu + 1, \ell - (\tilde{\ell}_\nu + \ell''_\nu)+ 1, \ell - (\tilde{\ell}_\nu + \ell''_\nu)+ 2, \ldots, \ell - \ell''_{\nu})  $$
and
$$S^c_\nu = (\ell''_{\nu} +1, \ell''_{\nu} +2, \ldots ,  \ell''_{\nu} + \tilde{\ell}_{\nu},  n - \mu/2 + \nu, \ell - (\tilde{\ell}_\nu + \ell''_\nu) +1, \ell - (\tilde{\ell}_\nu + \ell''_\nu)+ 2, \ldots, \ell - \ell''_{\nu}).$$
Note that $S^r_\nu$ and $S^c_\nu$ differ only at the middle element.
For $1 \leq \tau \leq \mu$ let $U_\tau$ denote the subspace of $U$ spanned by  
$$(   E(i, j), E(\ell - j + 1, \ell-i+1)  \, | \,  \text{ $(i,j) \in  ( S_\tau^r \times S_\tau^c) \cap \mathcal{P}$ }).$$
For $1 \leq \rho <  \varphi \leq \mu$ let $U_{\rho \varphi}$ denote the subspace of $U$ spanned by  
$$(   E(i, j), E(\ell - j + 1, \ell-i+1)  \, | \,  \text{ $(i,j) \in  (S_\rho^r \times S_\tau^c) \cap \mathcal{P} $}).$$
We have the decomposition 
$$U = \left(\bigoplus_{\tau = 1}^\mu U_\tau \right) \oplus \left( \bigoplus_{1 \leq \rho < \varphi \leq \mu} U_{\rho \varphi} \right)$$
of $U$ into $(\Ad(d'_{\vec{\ell}}) \circ \vartheta)$-invariant subspaces,
and we will now compute the characteristic polynomial of $\Ad(d'_{\vec{\ell}}) \circ \vartheta$ on each of these subspaces.

Fix $\tau$ with $1 \leq \tau \leq \mu$.  By examining ``diagonals'' that are parallel to the main diagonal, we see that the  characteristic polynomial for the action of $\Ad(d'_{\vec{\ell}}) \circ \vartheta$ on $U_\tau$ is 
$$   \prod_{k=1}^{\ell_{\tau}-1} (t^2 - \xi^{k (m/\ell_\tau)})^{(\ell_\tau  - 1)/2} = \left(\frac{(t^{2 \ell_\tau} -1)}{t^2-1} \right)^{(\ell_\tau  - 1)/2}$$

Now fix a pair $(\rho, \varphi)$ with $1 \leq \rho < \varphi \leq \mu$.
For $(i,j) \in \mathcal{P} \cap (S^r_\rho \times S^c_\varphi)$, we have that the characteristic polynomial for the action of $\Ad(d'_{\vec{\ell}}) \circ \vartheta$  on the span of the pair $(E(i, j), E(\ell - j + 1, \ell-i+1))$ is $t^2 - \xi^{2(d'_i - d'_j)}$.  Since the map $(i,j) \mapsto \xi^{2(d'_i - d'_j)}$ from $\mathcal{P} \cap (S^r_\rho \times S^c_\varphi)$ to $\C^\times$ has both (a) fibers of the same cardinality and (b) the  same image as the map $\mu_{\ell_\rho} \times \mu_{\ell_\varphi}   \rightarrow \C^\times$ that sends $(a,b)$ to $a \cdot b\inv$, we conclude that the characteristic polynomial for the action of $\Ad(d'_{\vec{\ell}}) \circ \vartheta$  on $U_{\rho \varphi}$ is 
$$ \prod_{\zeta \in \mu_{\ell_\rho}, \eta \in \mu_{\ell_\varphi}} (t^2 - (\zeta \mu)) = \left( \prod_{\delta \in \mu_{\lcm(\ell_\rho,\ell_\varphi)}} (t^2 - \delta) \right)^{\gcd(\ell_\rho,\ell_\varphi)} = (t^{2\lcm(\ell_\rho, \ell_\varphi)} - 1)^{\gcd(\ell_\rho, \ell_\varphi)}.$$

We now put it all together.
As an $\Ad(d'_{\vec{\ell}}) \circ \vartheta$-module, we have
$$\gllie_\ell = U_a \oplus U_d \oplus \left(\bigoplus_{\tau = 1}^\mu U_\tau \right) \oplus \left( \bigoplus_{1 \leq \rho < \varphi \leq \mu} U_{\rho \varphi} \right).$$
Thus, the characteristic polynomial for the action of $\Ad(d'_{\vec{\ell}}) \circ \vartheta$ on $\gllie_{\ell}$ is
\begin{equation*}
 \left( \prod_{\nu = 1}^{\mu} (t^{\ell_\nu} + 1)  \right) \cdot (t^2-1)^{n-\mu/2}
\cdot  \left( \prod_{\tau = 1}^{\mu} \left(\frac{(t^{2 \ell_\tau} -1)}{t^2-1} \right)^{(\ell_\tau  - 1)/2} \right) \cdot  \left( \prod_{{1 \leq \rho < \varphi \leq \mu}} (t^{2\lcm(\ell_\rho, \ell_\varphi)} - 1)^{\gcd(\ell_\rho, \ell_\varphi)} \right)
\end{equation*}
which, after simplifying, is $(t+1) \cdot q_{n_{\vec{\ell}}}$.
\end{proof}

\begin{lemma} \label{lem:oddcase}
$d_{\vec{\ell}}$ is $\vartheta$-conjugate to $n_{\vec{\ell}}$.
\end{lemma}

\begin{proof}
Recall that $m$ is even.

If $d'' = \Diag(\varepsilon^{a_n}, \varepsilon^{a_{n-1}}, \ldots  \varepsilon^{a_1}, \varepsilon^{a_{n+1}}, \varepsilon^{a_{n+2}}, \ldots ,\varepsilon^{a_{2n}})$ where $a_j = -a_{2n+1-j}$ for $j\geq n + 1$, then  
the characteristic polynomial of $d'' \circ \vartheta$ acting on $\sllie_{2n}$ is
$$ \frac{(t^2 -1)^n}{t+1} \cdot  \prod_{i = 1}^{n} (t - \varepsilon^{2a_i})(t-\varepsilon^{-2a_i}) \cdot  \prod_{i < j \leq 2n - i} (t^2-\varepsilon^{2(a_i - a_j)}) (t^2-\varepsilon^{-2(a_i - a_j)}).$$

Since the eigenvalues of an order $m$ element must be $m^{\text{th}}$ roots of unity, from the above paragraph we see that 
any element $d' \in A^\vartheta$ for which the order of $d' \rtimes \vartheta$ is $m$ can be chosen to look like 
$$d' = \Diag(\xi^{x_n}, \xi^{x_{n-1}}, \ldots , 
\xi^{x_1}, \xi^{x_{n+1}}, \xi^{x_{n+2}}\ldots \xi^{x_{2n}})$$
where $\xi$ is our fixed  $m^{\text{th}}$ root of unity, $x_j = - x_{2n+1-j}$ for $j \geq n+1$, the $x_i \in \frac{1}{2}\Z$ satisfy $m/2 \geq  x_1 \geq x_2 \geq \cdots \geq x_n > 0$, and $x_i - x_j$ is and integer for $1 \leq i,j \leq n$.   Fix a square root of $\xi$ so that $d'$ can be written unambiguously. We will show that $d' = d_{\vec{\ell}}$.

The characteristic polynomial of $d' \circ \vartheta$ acting on $\sllie_{2n}$ is
$$ \frac{(t^2 -1)^n}{t+1} \cdot  \prod_{i = 1}^{n} (t - \xi^{2x_i})(t-\xi^{-2x_i}) \cdot  \prod_{i < j \leq 2n - i} (t^2-\xi^{2(x_i - x_j)}) (t^2-\xi^{-2(x_i - x_j)}).$$
In the polynomial $q_{n_{\vec{\ell}}}$ every root   is paired with its additive inverse, except for the roots that appear in $p_{\vec{\ell}}$, the polynomial described in Equation~\ref{equ:pl} that is associated to $n_{\vec{\ell}}$. Thus, if $d'$ and $n_{\vec{\ell}}$ are $\vartheta$-conjugate, then we must have
$$\frac{ \prod_{i = 1}^{n} (t - \xi^{2x_i})(t-\xi^{-2x_i})}{t+1} = p_{\vec{\ell}} = \frac{ \prod_{i = 1}^{n} (t - \xi^{2\sigma_i})(t-\xi^{-2\sigma_i})}{t+1}.$$
Thanks to Lemma~\ref{lem:oddcasechar}, this implies that there is a bijective map $f$ from the set $\{1,2,\ldots n \}$ to itself such that $x_j \in \scoeff_{f(j)}/2 + m \Z$ for all $1 \leq j \leq n$.     Since $m/2 \geq  x_1 \geq x_2 \geq \cdots \geq x_n >0$ and $m/2 \geq  \sigma_1 \geq \sigma_2 \geq \cdots \geq \sigma_n > 0$, we conclude that $x_j = \sigma_j$ for $1 \leq j \leq n$.
\end{proof}

\subsubsection{How to create Kac diagrams for \texorpdfstring{$\lsup{2}A_{\ell-1}$ with $\ell = 2n+1 > 2$}{kac for 2A odd}}
Suppose $\ell = 2n+1$ is odd.  

 We can take the simple roots $\gamma$ in $\Delta_\vartheta$ to be the roots $\gamma_k = \res_{V^\vartheta} 2\alpha_k = 2(x_k - x_{k+1})$ for $1 \leq k < n$ and $\gamma_n = \res_{V^\vartheta} 2(\alpha_n + \alpha_{n+1} )= 
4x_n$.   
The projections of the fundamental coweights $\check{\omega}_i$ are then
$\check{\mu}_i = \frac{1}{2} (e_1 + e_2 + \cdots + e_i)$ 
for $1 \leq i \leq n$. 
The affine Dynkin diagram is 
\[
\begin{tikzpicture}[start chain]
\unodenj{0}
\unodenj{1}
\path (chain-1) -- node {\QRightarrow} (chain-2);
\end{tikzpicture}
\]
for $n =2$ and the diagram
\[
\begin{tikzpicture}[start chain]
\unodenj{0}
\unodenj{1}
\unode{2}
\dydots
\unode{n - 2}
\unode{n - 1}
\unodenj{n}
\path (chain-1) -- node{\(\Rightarrow\)} (chain-2);
\path (chain-6) -- node{\(\Rightarrow\)} (chain-7);
\end{tikzpicture}
\]
for $n>2$.
The $b_\gamma$ are given by the diagram
\[
\begin{tikzpicture}[start chain]
\enodenj{1}
\enodenj{2}
\path (chain-1) -- node {\QRightarrow} (chain-2);
\end{tikzpicture}
\]
for $n =2$ and
\[
\begin{tikzpicture}[start chain]
\enodenj{1}
\enodenj{2}
\enode{2}
\dydots
\enode{2}
\enode{2}
\enodenj{2}
\path (chain-1) -- node{\(\Rightarrow\)} (chain-2);
\path (chain-6) -- node{\(\Rightarrow\)} (chain-7);
\end{tikzpicture}
\]
for $n>2$.

Since each $\ell_\nu$ is odd and $\ell= 2n+1$ is odd, we must have that $\mu$ is odd.  Let $m = 2 \lcm(\ell_1, \ell_2, \ldots , \ell_{\mu})$, and let $\xi$ be a primitive $m^{\text{th}}$ root of unity.   Create a length $n$ list of non-negative integers  by appending  $(\mu-1)/2$ copies of $0$ to the list $( (a_\nu m)/ (2 \ell_\nu) \, | \, 1 \leq \nu \leq \mu, \, 1 \leq a_\nu < \ell_\nu/2 )$.  Place the items on this  list into decreasing order: $(\scoeff_1 \geq \scoeff_2 \geq \cdots \geq \scoeff_n)$.  Consider the diagonal matrix 
$$d_{\vec{\ell}} = \Diag( \xi^{\scoeff_1},   \xi^{\scoeff_2}, \ldots ,  \xi^{\scoeff_n}, 1,  \xi^{-\scoeff_n}, \xi^{-\scoeff_{n -1}},  \ldots , \xi^{-\scoeff_{1}}),$$
which is an element of $A^{\vartheta}$.    In Lemma~\ref{lem:evencase} we prove that  $d_{\vec{\ell}}$ is $\vartheta$-conjugate to $n_{\vec{\ell}}$.  Assuming this, 
we now read off the Kac diagram for $n_{\vec{\ell}}$ from $d_{\vec{\ell}}$. 
We have $d_{\vec{\ell}} = \lambda_{\vec{\ell}}(\xi)$ where 
$$\lambda_{\vec{\ell}} =  2(\scoeff_1 - \scoeff_2  ) \check{\mu}_1 + 2( \scoeff_2 - \scoeff_3   ) \check{\mu}_2  +  \ldots  + 2( \scoeff_{n-1} - \scoeff_n  ) \check{\mu}_{n-1} +  2\scoeff_n  \check{\mu}_{n}.$$
Since  $\dabs{\gamma}_\vartheta$ is $2$ for all $\gamma \in \Delta_\vartheta$ and 
$$\lambda_{\vec{\ell}}/m = \frac{1}{m}[
2(\scoeff_1 - \scoeff_2  ) \check{\mu}_1 + 2( \scoeff_2 - \scoeff_3   ) \check{\mu}_2  +  \ldots  + 2( \scoeff_{n-1} - \scoeff_n  ) \check{\mu}_{n-1} +  2\scoeff_n  \check{\mu}_{n}],$$
from Equation~\ref{equ:equationxg} we conclude that 
$s_{\gamma_n} = \scoeff_n$ and
 $s_{\gamma_k} = (\scoeff_k - \scoeff_{k+1})$ for  $1 \leq k \leq n-1$.
From Equation~\ref{equ:equationform}, we conclude $s_{\gamma_0} = m/2 -  2\scoeff_1$.
Remove any factors that are common to all of the $s_\gamma$ for $\gamma \in \tilde{\Delta}$, and label the extended Dynkin diagram with the resulting $s_\gamma$.

We have proved:
\begin{lemma}
Fix a partition $\vec{\ell} = (\ell_\mu, \ell_{\mu-1}, \ldots , \ell_2, \ell_1)$ of $\ell = 2n + 1$ for $n \geq 2$ with $\ell_\nu$ odd for $1 \leq \nu \leq \mu$.  Let $m = 2\lcm(\ell_1, \ell_2, \ldots , \ell_\mu)$.  Append  $(\mu-1)/2$  copies of zero to the list
$( ma_{\nu}/2 \ell_\nu \, | \, \text{$1 \leq \nu \leq  \mu$  and $ 1 \leq a_\nu < \ell_\nu/2$} )$
and then place the elements of the resulting list in decreasing order: 
$(\scoeff_1 \geq \scoeff_2 \geq \scoeff_3 \geq \cdots \geq \scoeff_{n-1} \geq \scoeff_n)$.
After removing any factors that are common to all of the labels, the Kac diagram for   $n_{\vec{\ell}} \rtimes \vartheta$ in a group of type $\lsup{2}A_{2n}$ is given by
\[
\begin{tikzpicture}[start chain]
\enode{\rotatebox[origin=c]{90}{$m/2 - 2\sigma_1$}}
\enodenj{\rotatebox[origin=c]{90}{$\sigma_1 - \sigma_2$}}
\enode{\rotatebox[origin=c]{90}{$\sigma_2 - \sigma_3$}}
\dydots
\enode{\rotatebox[origin=c]{90}{$\sigma_{n-2} - \sigma_{n-1}$}}
\enode{\rotatebox[origin=c]{90}{$\sigma_{n-1} - \sigma_n$}}
\enodenj{\sigma_n}
\path (chain-1) -- node[anchor=mid] {\(\Rightarrow\)} (chain-2);
\path (chain-6) -- node[anchor=mid] {\(\Rightarrow\)} (chain-7);
\end{tikzpicture}
\]
when $n>1$ and  by
\[
\begin{tikzpicture}[start chain]
\enodenj{\rotatebox[origin=c]{90}{$m/2 - 2\sigma_1$}}
\enodenj{\sigma_1}
\path (chain-1) -- node {\QRightarrow} (chain-2);
\end{tikzpicture}
\]
when $n = 1$. \qed
\end{lemma}

\begin{example}
As an example, here is the Kac diagram for $n_{(9,5,5,3,3)}$ in (twisted) $\SL_{25}$
\[
\begin{tikzpicture}[start chain]
\enodenj{5}
\enodenj{2}
\enode{0}
\enode{3}
\enode{0}
\enode{0}
\enode{5}
\enode{1}
\enode{0}
\enode{4}
\enode{5}
\enode{0}
\enodenj{0}
\path (chain-1) -- node{\(\Rightarrow\)} (chain-2);
\path (chain-12) -- node{\(\Rightarrow\)} (chain-13);
\end{tikzpicture}.
\]
\end{example}

\begin{example}
In Figure~ \ref{fig:2A2location} for groups of type $\lsup{2}A_2$ we show the location of the vertices $v_0$ and $v_1$ of the fundamental alcove as well as the points determined by $\lambda_{\vec{\ell}}/m$. The Kac diagram for $(3)$ is 
\(
\begin{tikzpicture}[start chain]
\enodenj{1}
\enodenj{1}
\path (chain-1) -- node {\QRightarrow} (chain-2);
\end{tikzpicture}
\)
and the Kac diagram for $(1,1,1)$ is 
\(
\begin{tikzpicture}[start chain]
\enodenj{1}
\enodenj{0}
\path (chain-1) -- node {\QRightarrow} (chain-2);
\end{tikzpicture}
\)
\begin{figure}[ht]
\centering
\begin{tikzpicture}

\draw (-0.1,0)   -- (5.296,0);
\draw[black,fill=black] (0,0) circle (.5ex);
\draw(.8,-0.8) node[anchor=south]{$v_0 = \frac{\lambda_{(1,1,1)}}{2}$};

\draw[black,fill=white] (5.196,0) circle (.5ex);
\draw(5.496,-0.6) node[anchor=south]{$v_1$};

\draw(3.781,0) node[anchor=north]{$\frac{\lambda_{(3)}}{6}$};
\draw[black,fill=black] (3.481,0) circle (.5ex);

\end{tikzpicture}
\caption{The location of points determined by  $\lambda_{\vec{\ell}}$ for groups of type $\lsup{2}A_2$ \label{fig:2A2location}}
\end{figure}
\end{example}

The results here agree with those found in the tables in~\cite[Section A.1]{reeder:thomae}.
The partitions for the Kac diagrams appearing  in the first table  are (from top to bottom): $(3)$ and $(1,1,1)$. This correspondence between Kac diagrams and partitions agrees with~\cite[Table~9]{reederetal:gradings}.     
The partitions (in the notation of \cite[Section A.1]{reeder:thomae}) for the Kac diagrams appearing  in the second table are (from top to bottom): $(2n+1)$, $(1,1,1 , \dots , 1, 1)$, $(d,d, d, \ldots, d,d)$ where $d$ is odd and appears $2k+1$ times where $d(2k+1) = 2n+1$, and $(d,d,d, \ldots, d,d,1)$ where $d$ is odd and appears $2k$ times where $dk = n$.  This correspondence between Kac diagrams and partitions agrees with~\cite[Table~10]{reederetal:gradings}.      From Remark~\ref{rem:regell2Al} these partitions correspond to the regular $\vartheta$-elliptic  elements in a twisted Weyl group of type $\lsup{2}A_{2n+1}$.

\subsubsection{A proof of \texorpdfstring{Lemma~\ref{lem:evencase}}{lemma even}}

Recall that $\ell = 2n + 1$, so we are looking at $\lsup{2}A_{2n}$.

As in Section~\ref{sec:oddcase} for $1 \leq i,k \leq \ell$ let $E(i,j)$ denote the $\ell \times \ell$ matrix such that $E(i,j)_{rc} = \delta_{ir} \delta_{jc}$.  Unlike in Section~\ref{sec:oddcase}
we have $\vartheta(E_{ij}) =  -E_{\ell-j+1, \ell-i +1}$.

\begin{lemma} \label{lem:evencasechar}
The characteristic polynomial of 
$\Ad(d_{\vec{\ell}}) \circ \vartheta$ acting on $\sllie_\ell$ is 
$$   q_{n_{\vec{\ell}}}(t) =  \frac{1}{t+1} \cdot
 \left( \prod_{\nu=1}^\mu (t^{\ell_\nu} + 1) \right) 
 \cdot  \left(    \prod_{\tau=1}^{\mu} (t^{2\ell_\tau} - 1)^{ (\ell_\tau-1)/2} \right) 
 \cdot \left(    \prod_{1 \leq \rho < \varphi \leq \mu} (t^{2 \lcm(\ell_\rho,\ell_\varphi)} - 1)^{ \gcd(\ell_\rho,\ell_\varphi)} \right). $$
\end{lemma}

\begin{proof}
We have that $\ell = 2n+1$,  $m = 2 \lcm (\ell_1, \ell_2, \ldots, \ell_\mu)$, and $\xi$ is a primitive $m^{\text{th}}$ root of unity. We first define a diagonal element $d'_{\vec{\ell}}$ that has the same entries as $d_{\vec{\ell}}$, but is more amenable to computation.   For $1 \leq \nu \leq \mu$, define $\tilde{\ell}_\nu = (\ell_\nu -1)/2$.

\begin{equation*}
\begin{split}
d'_{\vec{\ell}} =\Diag(\xi&^{\tilde{\ell}_\mu(m/2 \ell_\mu)} ,   \xi^{(\tilde{\ell}_\mu -1)(m/2 \ell_\mu)} , \ldots, \xi^{2(m/2 \ell_\mu)}, \xi^{m/2 \ell_\mu},\\
&\xi^{\tilde{\ell}_{\mu-1}(m/2 \ell_{\mu-1})} ,   \xi^{(\tilde{\ell}_{\mu-1}- 1)(m/2 \ell_{\mu-1})} , \ldots, \xi^{2(m/2 \ell_{\mu-1})}, \xi^{m/2 \ell_{\mu-1}},\\
& \ldots ,\\
&\xi^{\tilde{\ell}_1(m/2 \ell_1)} , \xi^{(\tilde{\ell}_1 -1)(m/2 \ell_1)}, \ldots , \xi^{2(m/2 \ell_1)}, \xi^{m/2 \ell_1}, \\
& \xi^{0}, \cdots , \xi^{0}, 1 = \xi^{0}, \xi^{0} \ldots, \xi^{0}, \\
& \xi^{-m/2 \ell_1}, \xi^{-2(m/2 \ell_1)}, \ldots , \xi^{-(\tilde{\ell}_1 -1)(m/2 \ell_1)} , \xi^{-\tilde{\ell}_1 (m/2 \ell_1)}, \\
& \ldots ,\\
&\xi^{-m/2 \ell_{\mu-1} }, \xi^{-2(m/2 \ell_{\mu-1})}, \ldots, \xi^{-(\tilde{\ell}_{\mu-1} -1)(m/2 \ell_{\mu-1})} , \xi^{-\tilde{\ell}_{\mu-1}(m/2 \ell_{\mu-1})}, \\
&\xi^{-m/2 \ell_{\mu} }, \xi^{-2(m/2 \ell_{\mu})}, \ldots, \xi^{-(\tilde{\ell}_{\mu} -1)(m/2 \ell_\mu)} , \xi^{-\tilde{\ell}_{\mu}(m/2 \ell_{\mu})}).
\end{split}
\end{equation*}
In the middle of the matrix there are $\mu$ copies of $\xi^{0}$ centered around the $1 = \xi^0$ that occurs in the center of the matrix. 
Since $d'_{\vec{\ell}}$ is $W^\vartheta$-conjugate to $d_{\vec{\ell}}$, it is enough to compute the characteristic polynomial of $\Ad(d'_{\vec{\ell}}) \circ \vartheta$ acting on $\sllie_\ell$.  Thus, it is enough to show that the characteristic polynomial of $\Ad(d'_{\vec{\ell}}) \circ \vartheta$ acting on $\gllie_\ell$ is
$$  (t+1) \cdot q_{n_{\vec{\ell}}}(t) =  
 \left( \prod_{\nu=1}^\mu (t^{\ell_\nu} + 1) \right) 
 \cdot  \left(    \prod_{\tau=1}^{\mu} (t^{2\ell_\tau} - 1)^{ (\ell_\tau-1)/2} \right) 
 \cdot \left(    \prod_{1 \leq \rho < \varphi \leq \mu} (t^{2 \lcm(\ell_\rho,\ell_\varphi)} - 1)^{ \gcd(\ell_\rho,\ell_\varphi)} \right). $$

For $1 \leq i \leq \ell$, define $d'_i$ by $(d'_{\vec{\ell}})_{ii} = \xi^{d'_i}$.  That is, the power of $\xi$ appearing in the $i^{\text{th}}$ diagonal element of $d'_{\vec{\ell}}$ is $d'_i$.

For $1 \leq i \leq \ell$  we have $(\Ad(d'_{\vec{\ell}}) \circ \vartheta ) E(i, \ell-i+1) =  - \xi^{2d'_i} E(i, \ell-i+1)$.  Thus, the vector space $U_a$ of anti-diagonal elements in $\gllie_\ell$ is $(\Ad(d'_{\vec{\ell}}) \circ \vartheta$)-stable and the characteristic polynomial for the action of $\Ad(d'_{\vec{\ell}}) \circ \vartheta$ on $U_a$ is 
$$\prod_{i=1}^{\ell} (t + \xi^{2d'_i}) =    \prod_{\nu = 1}^{\mu} \prod_{i=1}^{\ell_\nu}(t + \xi^{2i(m/2\ell_\nu)}),$$
and we have 
$$  \prod_{\nu = 1}^{\mu} \prod_{i=1}^{\ell_\nu}(t + \xi^{2i(m/2\ell_\nu)})  =    \prod_{\nu = 1}^{\mu} \prod_{i=1}^{\ell_\nu}(t - \xi^{2i(m/2 \ell_\nu)+m/2}) = \prod_{\nu = 1}^{\mu} \prod_{i=1}^{\ell_\nu}(t - \xi^{(2i + \ell_\nu)(m/2\ell_\nu)}) = \prod_{\nu = 1}^{\mu} (t^{\ell_\nu} + 1).$$
Thus, the characteristic polynomial for the action of $\Ad(d'_{\vec{\ell}}) \circ \vartheta$ on $U_a$ is
$$ \prod_{\nu = 1}^{\mu} (t^{\ell_\nu} + 1).$$

For $1 \leq i \leq n- (\mu-1)/2$  we have $(\Ad(d'_{\vec{\ell}}) \circ \vartheta )E(i,i) = -E(\ell -i + 1, \ell-i+1)$ and $(\Ad(d'_{\vec{\ell}}) \circ \vartheta) E(\ell -i + 1, \ell-i+1) = -E(i,i)$.  Thus, the  vector space $U_d$  spanned by
$$(E(i,i), E(\ell -i + 1, \ell-i+1) \, | \, 1 \leq i \leq n - (\mu-1)/2)$$
in $\gllie_\ell$ is  $(\Ad(d'_{\vec{\ell}}) \circ \vartheta)$-stable, and the characteristic polynomial for the action of $\Ad(d'_{\vec{\ell}}) \circ \vartheta$ on $U_d$ is $(t^2-1)^{n-(\mu-1)/2}$.

For pairs $(i,j)$ where $1 \leq i \leq n$ and $i < j \leq \ell - i $ we have 
$$(\Ad(d'_{\vec{\ell}}) \circ \vartheta ) E(i, j) = \xi^{d'_i - d'_j} E(\ell - j + 1, \ell-i+1)$$
and 
$$(\Ad(d'_{\vec{\ell}}) \circ \vartheta ) E(\ell - j + 1, \ell-i+1) =  \xi^{d'_i - d'_j} E(i,j).$$
Similarly, for pairs $(i,j)$ where $1 \leq j \leq n$ and $j < i \leq \ell - j $ we have 
$$(\Ad(d'_{\vec{\ell}}) \circ \vartheta ) E(i, j) =  \xi^{d'_i - d'_j} E(\ell - j + 1, \ell-i+1)$$
and 
$$(\Ad(d'_{\vec{\ell}}) \circ \vartheta ) E(\ell - j + 1, \ell-i+1) =  \xi^{d'_i - d'_j} E(i,j).$$
Let  $\mathcal{P}$ denote the set of pairs $(i,j)$ such that (exactly) one of the following is true:
\begin{itemize}
\item $1 \leq i  \leq n$ and $i < j \leq \ell - i $, 
\item $1 \leq j  \leq n$ and $j < i \leq \ell - j$, or
\item $ i = j$ and $n-(\mu-1)/2 < i \leq n$.
\end{itemize}
Then, $U$, the span of the elements 
$$(   E(i, j), E(\ell - j + 1, \ell-i+1)  \, | \, \text{ $(i,j) \in \mathcal{P}$ } ),$$
is $(\Ad(d'_{\vec{\ell}}) \circ \vartheta)$-stable since the span of any pair $( E(i, j), E(\ell - j + 1, \ell-i+1))$ occurring in the definition of $U$ is  $(\Ad(d'_{\vec{\ell}}) \circ \vartheta)$-stable.  Moreover, $U$ is a vector space complement in $\gllie_\ell$ to the vector subspace $U_a \oplus U_d$ of  $\gllie_\ell$. 
We are going to break $U$ into $\mu(\mu+1)/2$ smaller $(\Ad(d'_{\vec{\ell}}) \circ \vartheta)$-stable subspaces, and evaluate the characteristic polynomial of 
the action of $\Ad(d'_{\vec{\ell}}) \circ \vartheta$ on each of these smaller spaces.

For $0 \leq \nu \leq \mu$, define  $\ell''_\nu = \sum_{j=\nu+1}^\mu \tilde{\ell}_j$.  Note that $\ell''_0 = n - (\mu-1)/2$ and $\ell''_\mu = 0$.  For $1 \leq \nu \leq \mu$, define the length $\ell_\nu$ lists 
$$S^r_\nu = (\ell''_{\nu} +1, \ell''_{\nu} +2, \ldots ,  \ell''_{\nu} + \tilde{\ell}_{\nu},  n + (\mu-1)/2 - \nu + 2, \ell - (\tilde{\ell}_\nu + \ell''_\nu)+ 1, \ell - (\tilde{\ell}_\nu + \ell''_\nu)+ 2, \ldots, \ell - \ell''_{\nu})  $$
and
$$S^c_\nu = (\ell''_{\nu} +1, \ell''_{\nu} +2, \ldots ,  \ell''_{\nu} + \tilde{\ell}_{\nu},  n - (\mu-1)/2 + \nu , \ell - (\tilde{\ell}_\nu + \ell''_\nu) +1, \ell - (\tilde{\ell}_\nu + \ell''_\nu)+ 2, \ldots, \ell - \ell''_{\nu}).$$
Note that $S^r_\nu$ and $S^c_\nu$ differ only at the middle element.
For $1 \leq \tau \leq \mu$ let $U_\tau$ denote the subspace of $U$ spanned by  
$$(   E(i, j), E(\ell - j + 1, \ell-i+1)  \, | \,  \text{ $(i,j) \in  ( S_\tau^r \times S_\tau^c) \cap \mathcal{P}$ }).$$
For $1 \leq \rho <  \varphi \leq \mu$ let $U_{\rho \varphi}$ denote the subspace of $U$ spanned by  
$$(   E(i, j), E(\ell - j + 1, \ell-i+1)  \, | \,  \text{ $(i,j) \in  (S_\rho^r \times S_\tau^c) \cap \mathcal{P} $}).$$
We have the decomposition 
$$U = \left(\bigoplus_{\tau = 1}^\mu U_\tau \right) \oplus \left( \bigoplus_{1 \leq \rho < \varphi \leq \mu} U_{\rho \varphi} \right)$$
of $U$ into $(\Ad(d'_{\vec{\ell}}) \circ \vartheta)$-invariant subspaces,
and we will now compute the characteristic polynomial of $\Ad(d'_{\vec{\ell}}) \circ \vartheta$ on each of these subspaces.

Fix $\tau$ with $1 \leq \tau \leq \mu$.  By examining ``diagonals'' that are parallel to the main diagonal, we see that the  characteristic polynomial for the action of $\Ad(d'_{\vec{\ell}}) \circ \vartheta$ on $U_\tau$ is 
$$   \prod_{k=1}^{\ell_{\tau}-1} (t^2 - \xi^{k (m/\ell_\tau)})^{(\ell_\tau  - 1)/2} = \left(\frac{(t^{2 \ell_\tau} -1)}{t^2-1} \right)^{(\ell_\tau  - 1)/2}$$

Now fix a pair $(\rho, \varphi)$ with $1 \leq \rho < \varphi \leq \mu$.
For $(i,j) \in \mathcal{P} \cap (S^r_\rho \times S^c_\varphi)$, we have that the characteristic polynomial for the action of $\Ad(d'_{\vec{\ell}}) \circ \vartheta$  on the span of the pair $(E(i, j), E(\ell - j + 1, \ell-i+1))$ is $t^2 - \xi^{2(d'_i - d'_j)}$.  Since the map $(i,j) \mapsto \xi^{2(d'_i - d'_j)}$ from $\mathcal{P} \cap (S^r_\rho \times S^c_\varphi)$ to $\C^\times$ has both (a) fibers of the same cardinality and (b) the  same image as the map $\mu_{\ell_\rho} \times \mu_{\ell_\varphi}   \rightarrow \C^\times$ which sends $(a,b)$ to $a \cdot b\inv$, we conclude that the characteristic polynomial for the action of $\Ad(d'_{\vec{\ell}}) \circ \vartheta$  on $U_{\rho \varphi}$ is 
$$ \prod_{\zeta \in \mu_{\ell_\rho}, \eta \in \mu_{\ell_\varphi}} (t^2 - (\zeta \mu)) = \left( \prod_{\delta \in \mu_{\lcm(\ell_\rho,\ell_\varphi)}} (t^2 - \delta) \right)^{\gcd(\ell_\rho,\ell_\varphi)} = (t^{2\lcm(\ell_\rho, \ell_\varphi)} - 1)^{\gcd(\ell_\rho, \ell_\varphi)}.$$

We now put it all together.
As an $\Ad(d'_{\vec{\ell}}) \circ \vartheta$-module, we have
$$\gllie_\ell = U_a \oplus U_d \oplus \left(\bigoplus_{\tau = 1}^\mu U_\tau \right) \oplus \left( \bigoplus_{1 \leq \rho < \varphi \leq \mu} U_{\rho \varphi} \right).$$
Thus, the characteristic polynomial for the action of $\Ad(d'_{\vec{\ell}}) \circ \vartheta$ on $\gllie_{\ell}$ is
\begin{equation*}
\left(\prod_{\nu = 1}^{\mu} (t^{\ell_\nu} + 1) \right) \cdot (t^2-1)^{n-(\mu-1)/2}
\cdot \left( \prod_{\tau = 1}^{\mu} \left(\frac{(t^{2 \ell_\tau} -1)}{t^2-1} \right)^{(\ell_\tau  - 1)/2} \right) \cdot \left( \prod_{{1 \leq \rho < \varphi \leq \mu}} (t^{2\lcm(\ell_\rho, \ell_\varphi)} - 1)^{\gcd(\ell_\rho, \ell_\varphi)} \right)
\end{equation*}
which, after simplifying, is $(t+1) \cdot q_{n_{\vec{\ell}}}$.
\end{proof}

\begin{lemma} \label{lem:evencase}
$d_{\vec{\ell}}$ is $\vartheta$-conjugate to $n_{\vec{\ell}}$.
\end{lemma}

\begin{proof}
If $d'' = \Diag(\varepsilon^{a_n}, \varepsilon^{a_{n-1}}, \ldots  \varepsilon^{a_1}, \varepsilon^{a_{n+1}}, \varepsilon^{a_{n+2}}, \ldots ,\varepsilon^{a_{2n+1}})$ where $a_j = -a_{2n+2-j}$ for $j\geq n + 1$, then  
the characteristic polynomial of $d'' \circ \vartheta$ acting on $\sllie_{2n+1}$ is
$$ {(t^2 -1)^n} \cdot  \prod_{i = 1}^{n} (t - \varepsilon^{2a_i})(t-\varepsilon^{-2a_i}) \cdot  \prod_{1 \leq i < j \leq 2n + 1 - i} (t^2-\varepsilon^{2(a_i - a_j)}) (t^2-\varepsilon^{-2(a_i - a_j)}).$$     
Note that $a_{n+1} = 0$, and therefore for $1 \leq i \leq n$
the factor $(t^2 - \varepsilon^{2a_i}) = (t^2-\varepsilon^{2(a_i - a_{n+1})})$ occurs in the characteristic polynomial of $d'' \circ \vartheta$ acting on $\sllie_{2n+1}$. 

Since the eigenvalues of an order $m$ element must be $m^{\text{th}}$ roots of unity, from the above paragraph we see that 
any element $d' \in A^\vartheta$ for which the order of $d' \rtimes \vartheta$ is $m$ can be chosen to look like 
$$d' = \Diag(\xi^{x_n}, \xi^{x_{n-1}}, \ldots , 
\xi^{x_1}, \xi^{x_{n+1}}, \xi^{x_{n+2}}\ldots \xi^{x_{2n+1}})$$
where $\xi$ is our fixed  $m^{\text{th}}$ root of unity, $x_j = - x_{2n+2-j}$ for $j \geq n+1$, and  the $x_i \in \Z$ satisfy $m/2 \geq  x_1 \geq x_2 \geq \cdots \geq x_n \geq x_{n+1} =  0$.   We will show that $d' = d_{\vec{\ell}}$.

The characteristic polynomial of $d' \circ \vartheta$ acting on $\sllie_{2n+1}$ is
$$ {(t^2 -1)^n} \cdot  \prod_{i = 1}^{n} (t - \xi^{2x_i})(t-\xi^{-2x_i}) \cdot  \prod_{1 \leq i < j \leq 2n +1 - i} (t^2-\xi^{2(x_i - x_j)}) (t^2-\xi^{-2(x_i - x_j)}).$$
In the polynomial $q_{n_{\vec{\ell}}}$ every root   is paired with its additive inverse, except for the roots that appear in $p_{\vec{\ell}}$, the polynomial described in Equation~\ref{equ:pl} that is associated to $n_{\vec{\ell}}$.   Thus, if $d'$ and $n_{\vec{\ell}}$ are $\vartheta$-conjugate, then we must have
$$   \prod_{i = 1}^{n} (t - \xi^{2x_i})(t-\xi^{-2x_i}) = p_{\vec{\ell}}  =  \prod_{i = 1}^{n} (t - \xi^{2\sigma_i})(t-\xi^{-2\sigma_i}).$$
Thanks to Lemma~\ref{lem:evencasechar}, this implies that there is a bijective map $f$ from the set $\{1,2,\ldots n \}$ to itself such that $x_j \in \scoeff_{f(j)} + m \Z$ for all $1 \leq j \leq n$. 
    Since $m/2 >  x_1 \geq x_2 \geq \cdots \geq x_n \geq 0$ and $m/2 >  \sigma_1 \geq \sigma_2 \geq \cdots \geq \sigma_n \geq 0$, we conclude that $x_j = \sigma_j$ for $1 \leq j \leq n$.
\end{proof}

\subsection{How to create a Kac diagram for \texorpdfstring{$n_{\vec{\ell}}$ for $\lsup{2}D_{\ell+1}$ with $\ell \geq 2$}{kac for 2D}}  
\label{sec:kacdiagram2Dlpo}

We adopt the notation of Section~\ref{sec:inj2Dlplusone}. 
 Fix a partition $\vec{\ell}$ of $(\ell+1)$ with an odd number of parts.   We will show how to construct the Kac diagram for   $n_{\vec{\ell}} \rtimes \vartheta$.

We have $f = 2$.  We take the simple roots $\gamma$ in $\Delta_\vartheta = \Delta$ to be the roots $\gamma_k = \alpha_k$ for $1 \leq k < \ell$ and $\gamma_\ell = \alpha_{\ell -1} + \alpha_\ell = 2e_\ell$.  
For fundamental coweights with respect to our basis we take
$\check{\mu}_i = e_1 + e_2 + \cdots + e_i$ for $1 \leq i < \ell$ and $\check{\mu}_{\ell} = 1/2 (e_1 + e_2 + \cdots e_{\ell-2}  +  e_{\ell-1} + e_\ell)$.
The affine Dynkin diagram is 
\[
\begin{tikzpicture}[start chain]
\unode{0}
\unodenj{1}
\unode{2}
\dydots
\unode{\ell-2}
\unode{\ell-1}
\unodenj{\ell}
\path (chain-1) -- node[anchor=mid] {\(\Leftarrow\)} (chain-2);
\path (chain-6) -- node[anchor=mid] {\(\Rightarrow\)} (chain-7);
\end{tikzpicture}
\]
and the $b_\gamma$ are given by the diagram
\[
\begin{tikzpicture}[start chain]
\enode{1}
\enodenj{1}
\enode{1}
\dydots
\enode{1}
\enode{1}
\enodenj{1}
\path (chain-1) -- node[anchor=mid] {\(\Leftarrow\)} (chain-2);
\path (chain-6) -- node[anchor=mid] {\(\Rightarrow\)} (chain-7);
\end{tikzpicture}
\]

A Kac diagram describes a diagonal matrix in $\SO_{2 \ell + 2}$ that is $\vartheta$-conjugate to $n_{\vec{\ell}}$ in $G$.  Thus, we want to find a diagonal matrix $d_{\vec{\ell}}$ in $\SO_{2 \ell+2}$ such that the characteristic polynomial of $d_{\vec{\ell}} \rtimes \vartheta$ for the standard action on  $\C^{2\ell +2}$ is 
$$q_{n_{\vec{\ell}}}(t) = \prod_{\nu = 1}^\mu (t^{2 \ell_\nu} - 1).$$

Let $m = 2\lcm(\ell_1, \ell_2, \ldots , \ell_\mu)$ and let $\xi$ be a primitive $m^{\text{th}}$ root of unity. We have
$$q_{n_{\vec{\ell}}}(t) = \prod_{\nu = 1}^\mu (t^{2 \ell_\nu} - 1) = \prod_{\nu = 1}^\mu (t-1)(t+1) \prod_{a_{\nu} = 1}^{\ell_\nu - 1} (t - \xi^{ma_{\nu}/2 \ell_\nu} )(t - \xi^{-ma_{\nu}/2 \ell_\nu} )$$
Guided by  Remark~\ref{rem:decreasing} we create a length $\ell$ decreasing list as follows: order the positive integers $ma_{\nu}/2 \ell_\nu$ in decreasing order, then pre-append $(\mu-1)/2$ copies of $m/2$ and post-append $(\mu-1)/2$ zeroes.  We thus obtain a list $(\scoeff_1 \geq \scoeff_2 \geq \scoeff_3 \geq \cdots \geq \scoeff_{\ell-1} \geq \scoeff_\ell)$.  Define $d_{\vec{\ell}} = \Diag(\xi^{\scoeff_1},  \xi^{\scoeff_2}, \ldots , \xi^{\scoeff_{\ell -1}}, \xi^{\scoeff_{\ell}},  1,1,\xi^{-\scoeff_{\ell}}, \xi^{-\scoeff_{\ell-1}}, \ldots ,  \xi^{-\scoeff_2}, \xi^{-\scoeff_1}) $ in $A^{\vartheta}$.  Then $d_{\vec{\ell}} \circ \vartheta$ acting on $\C^{2 \ell + 2}$ has characteristic polynomial $q_{n_{\vec{\ell}}}$.

 Since the linear factors $(t-1)$ and $(t+1)$ along with  $(t-\xi^{\scoeff_j})$ and $(t - \xi^{-\scoeff_j})$ for $1 \leq j \leq \ell$ are exactly the linear factors that must occur in  $q_{n_{\vec{\ell}}}$,   we conclude that, up to the $\vartheta$-conjugation in  $\absW$, $d_{\vec{\ell}}$ is the unique element of $A$ that is $\vartheta$-conjugate to $n_{\vec{\ell}}$ in $G$.

We now read off the Kac diagram for $n_{\vec{\ell}}$ from $d_{\vec{\ell}}$. 
 Note that $d_{\vec{\ell}} = \lambda_{\vec{\ell}}(\xi)$ where 
$$\lambda_{\vec{\ell}} =  ({\scoeff_{1} - \scoeff_{2}}) \check{\mu}_1 + ({\scoeff_{2} - \scoeff_{3}}) \check{\mu}_2  + \ldots  + (\scoeff_{\ell-1} - \scoeff_\ell) \check{\mu}_{\ell-1} + 2 \scoeff_\ell \check{\mu}_{\ell}.$$
Since $\dabs{\gamma_k}_\vartheta = 1$ for $1 \leq k < \ell$, $\dabs{\gamma_\ell}_\vartheta = 2$, and
$$\lambda_{\vec{\ell}}/m = \frac{1}{m} [ ({\scoeff_{1} - \scoeff_{2}}) \check{\mu}_1 + ({\scoeff_{2} - \scoeff_{3}}) \check{\mu}_2  + \ldots  + (\scoeff_{\ell-1} - \scoeff_\ell) \check{\mu}_{\ell-1} + 2 \scoeff_\ell \check{\mu}_{\ell}],$$
we have 
$s_{\gamma_\ell} = \scoeff_\ell$ and 
 $s_{\gamma_k} = (\scoeff_{k} - \scoeff_{k+1})$ for  $1 \leq k \leq \ell-1$.    The coefficient $s_{\gamma_0} = m/2 - \scoeff_1$ is derived using Equation~\ref{equ:equationform}.  Remove any factors that are common to all of the $s_\gamma$ for $\gamma \in \tilde{\Delta}$ and label the corresponding affine Kac diagram with the resulting $s_\gamma$. 

We have proved:
\begin{lemma}
Fix a partition $\vec{\ell} = (\ell_\mu, \ell_{\mu-1}, \ldots , \ell_2, \ell_1)$ of $(\ell+1)$ with $\mu$ odd and $\ell \geq 2$.  Let $m = 2\lcm(\ell_1, \ell_2, \ldots , \ell_\mu)$.  Append  $(\mu-1)/2$ copies of $m/2$ and $(\mu-1)/2$ zeroes to the list
$( ma_{\nu}/2 \ell_\nu \, | \, \text{$1 \leq \nu \leq \mu$  and $ 1 \leq a_\nu \leq \ell_\nu -1$} )$
and then place the elements of the resulting list in decreasing order: 
$(\scoeff_1 \geq \scoeff_2 \geq \scoeff_3 \geq \cdots \geq \scoeff_{\ell-1} \geq \scoeff_\ell)$.
After removing any factors that are common to all of the labels, the Kac diagram for   $n_{\vec{\ell}} \rtimes \vartheta$ in a group of type $\lsup{2}D_{\ell + 1}$ with $\ell \geq 2$ is given by
\[
\begin{tikzpicture}[start chain]
\enode{\rotatebox[origin=c]{90}{$m/2 - \sigma_1$}}
\enodenj{\rotatebox[origin=c]{90}{$\sigma_1 - \sigma_2$}}
\enode{\rotatebox[origin=c]{90}{$\sigma_2 - \sigma_3$}}
\dydots
\enode{\rotatebox[origin=c]{90}{$\sigma_{\ell-2} - \sigma_{\ell-1}$}}
\enode{\rotatebox[origin=c]{90}{$\sigma_{\ell-1} - \sigma_\ell$}}
\enodenj{\sigma_\ell}
\path (chain-1) -- node[anchor=mid] {\(\Leftarrow\)} (chain-2);
\path (chain-6) -- node[anchor=mid] {\(\Rightarrow\)} (chain-7);
\end{tikzpicture}. 
\qed
\]
\end{lemma}

 \begin{example}
As an example, here is the Kac diagram for $n_{(5,4,3)}$ in (twisted) $\SO_{24}$.
\[
\begin{tikzpicture}[start chain]
\enode{0}
\enodenj{12}
\enode{3}
\enode{5}
\enode{4}
\enode{6}
\enode{6}
\enode{4}
\enode{5}
\enode{3}
\enode{12}
\enodenj{0}
\path (chain-1) -- node[anchor=mid] {\(\Leftarrow\)} (chain-2);
\path (chain-11) -- node[anchor=mid] {\(\Rightarrow\)} (chain-12);
\end{tikzpicture} .
\]
\end{example}

\begin{example}
In Figure~ \ref{fig:2D3location} for groups of type $\lsup{2}D_3$ we show the location of the vertices $v_0$, $v_1$, and $v_2$ of the fundamental alcove as well as the points determined by $\lambda_{\vec{\ell}}/m$. The Kac diagram for $(3)$ is 
\(
\begin{tikzpicture}[start chain]
\enodenj{1}
\enodenj{1}
\enodenj{1}
\path (chain-1) -- node{\(\Leftarrow\)} (chain-2);
\path (chain-3) -- node{\(\Rightarrow\)} (chain-2);
\end{tikzpicture}
\)
and the Kac diagram for $(1,1,1)$ is 
\(
\begin{tikzpicture}[start chain]
\enodenj{0}
\enodenj{1}
\enodenj{0}
\path (chain-1) -- node{\(\Leftarrow\)} (chain-2);
\path (chain-3) -- node{\(\Rightarrow\)} (chain-2);
\end{tikzpicture}.
\)
\begin{figure}[ht]
\centering
\begin{tikzpicture}
\draw (0,0) 
  -- (5,0)
  -- (5,5)
  -- cycle;
  \draw[black, fill=white] (0,0) circle (.5ex);
\draw(-.3,-.4) node[anchor=south]{$v_0$};
 \draw[black,fill=white] (5,5) circle (.5ex);
\draw(5.3,4.9) node[anchor=south]{$v_2$};
 \draw[black,fill=black] (3.33,1.66) circle (.5ex);
\draw(3.63,1.66) node[anchor=north]{$\frac{\lambda_{(3)}}{6}$};
  \draw[black,fill=black] (5,0) circle (.5ex);
\draw(6.05,.38) node[anchor=north]{$v_1 = \frac{\lambda_{(1,1,1)}}{2}$};
\end{tikzpicture}
\caption{The location of points determined by  $\lambda_{\vec{\ell}}$ for groups of type $\lsup{2}D_3$ \label{fig:2D3location}}
\end{figure}
Even though nothing else in their derivation agrees, since $\lsup{2}A_3 \cong \lsup{2}D_3$ it is correct that the location of the points  in Figures~\ref{fig:2A3location} and~\ref{fig:2D3location} corresponding to the $\vartheta$-elliptic $\vartheta$-conjugacy classes in their respective Weyl groups  coincide.
\end{example}

During the publishing process, a typo was introduced to the table in~\cite[Section A.6]{reeder:thomae}. For the case $k$ even, $k \mid n$, and $k>2$ the Kac diagram should have a $0$ on the $\gamma_0$ node; that is, it should be 
\[
\begin{tikzpicture}[start chain]
\enodenj{0}
\enodenj{0}
\enode{0}
\enodenj{0}
\enode{1}
\enode{0}
\enode{0}
\enodenj{0}
\enode{1}
\enode{0}
\enode{0}
\enodenj{0}
\enode{1}
\enodenj{1}
\enode{0}
\enode{0}
\enodenj{0}
\enodenj{0}
\path (chain-1) -- node{\(\Leftarrow\)} (chain-2);
\path (chain-3) -- node{\(\cdots\)} (chain-4);
\path (chain-7) -- node{\(\cdots\)} (chain-8);
\path (chain-11) -- node{\(\cdots\)} (chain-12);
\path (chain-13) -- node{\(\cdots\)} (chain-14);
\path (chain-16) -- node{\(\cdots\)} (chain-17);
\path (chain-17) -- node{\(\Rightarrow\)} (chain-18);
\draw[decorate, decoration={  brace,  amplitude=10pt}] (0,.12)-- node[above=0.35cm]
{$k/2$ zeroes}(3,.12);
\draw[decorate, decoration={  brace,  amplitude=10pt}] (4.9,.12)-- node[above=0.35cm]
{$(k-1)$ zeroes}(6.9,.12);
\draw[decorate, decoration={  brace,  amplitude=10pt}] (8.8,.12)-- node[above=0.35cm]
{$(k-1)$ zeroes}(10.8,.12);
\draw[decorate, decoration={  brace,  amplitude=10pt}] (13.7,.12)-- node[above=0.35cm]
{$k/2$ zeroes}(16.7,.12);
\node[above=0.35cm] at (12.3,.12) {$\cdots$};
\end{tikzpicture}
\]
where there are $(n/k-1)$ strings of $(k-1)$ zeroes.  With this change,  the partitions (in the notation of \cite[Section A.6]{reeder:thomae}) for the Kac diagrams appearing there are (from top to bottom): $(n+1)$, $(n/2,n/2,1)$ for $n$ even, $(n/k, n/k, n/k, \ldots , n/k, n/k,1)$ for $2< k$ even and dividing $n$, and $((n+1)/k, (n+1)/k, (n+1)/k, \ldots , (n+1)/k, (n+1)/k)$ for $1 < k$ odd and dividing $n+1$.   This correspondence between Kac diagrams and partitions agrees with~\cite[Table~15]{reederetal:gradings}.     
 From Remark~\ref{rem:regell2Dl} these partitions correspond to the regular $\vartheta$-elliptic  elements in a twisted Weyl group  of type $\lsup{2}D_{n+1}$.

\subsection{Kac diagrams for the remaining types of groups} 
\label{sec:kacexceptional}

As noted in Section~\ref{sec:injexceptional}, the Kac diagrams 
associated to $\vartheta$-elliptic conjugacy classes in $\absW$ for the remaining types of groups ($G_2$, $F_4$, $E_6$, $E_7$, $E_8$, $\lsup{3}D_4$, or $\lsup{2}E_6$) are known.    See~\cite[Section~9]{adams-he-nie:from}.

\end{document}